\documentclass[12pt]{amsart}  	

\usepackage[margin=1in]{geometry}  

\usepackage[all]{xy}						
\usepackage{amssymb}
\usepackage{amscd,latexsym,amsthm,amsfonts,amssymb,amsmath,amsxtra}
\usepackage[mathscr]{eucal}
\usepackage[colorlinks]{hyperref}
\usepackage{mathtools}

\usepackage{chngcntr}
\numberwithin{equation}{section}
\newcounter{keepeqno}

\hypersetup{linkcolor=black, citecolor=black}

\newcommand{\BA}{{\mathbb {A}}}

\newcommand{\BC}{{\mathbb {C}}}

\newcommand{\BF}{{\mathbb {F}}}
\newcommand{\BG}{{\mathbb {G}}}

\newcommand{\BL}{{\mathbb {L}}}

\newcommand{\BQ}{{\mathbb {Q}}}
\newcommand{\BR}{{\mathbb {R}}}

\newcommand{\BZ}{{\mathbb {Z}}}

\newcommand{\CA}{{\mathcal {A}}}
\newcommand{\CB}{{\mathcal {B}}}
\newcommand{\CC}{{\mathcal {C}}}
\newcommand{\CE}{{\mathcal {E}}}
\newcommand{\CF}{{\mathcal {F}}}

\newcommand{\CH}{{\mathcal {H}}}
\newcommand{\CI}{{\mathcal {I}}}

\newcommand{\CK}{{\mathcal {K}}}
\newcommand{\CL}{{\mathcal {L}}}
\newcommand{\CM}{{\mathcal {M}}}

\newcommand{\CP}{{\mathcal {P}}}

\newcommand{\CS}{{\mathcal {S}}}
\newcommand{\CT}{{\mathcal {T}}}

\newcommand{\CW}{{\mathcal {W}}}

\newcommand{\CZ}{{\mathcal {Z}}}

\newcommand{\FC}{{\mathfrak {C}}}
\newcommand{\FD}{{\mathfrak {D}}}

\newcommand{\FR}{{\mathfrak {R}}}

\newcommand{\FX}{{\mathfrak {X}}}

\newcommand{\FZ}{{\mathfrak {Z}}}

\newcommand{\Fc}{{\mathfrak {c}}}

\newcommand{\Ff}{{\mathfrak {f}}}

\newcommand{\Fm}{{\mathfrak {m}}}

\newcommand{\Fo}{{\mathfrak {o}}}
\newcommand{\Fp}{{\mathfrak {p}}}

\newcommand{\Fr}{{\mathfrak {r}}}

\newcommand{\Fz}{{\mathfrak {z}}}

\newcommand{\RB}{{\mathrm {B}}}

\newcommand{\RG}{{\mathrm {G}}}

\newcommand{\RI}{{\mathrm {I}}}

\newcommand{\RM}{{\mathrm {M}}}
\newcommand{\RN}{{\mathrm {N}}}
\newcommand{\RO}{{\mathrm {O}}}
\newcommand{\RP}{{\mathrm {P}}}

\newcommand{\RT}{{\mathrm {T}}}
\newcommand{\RU}{{\mathrm {U}}}

\newcommand{\RW}{{\mathrm {W}}}

\newcommand{\cusp}{{\mathrm{cusp}}}

\newcommand{\cpt}{{\mathrm{cpt}}}

\newcommand{\GJ}{{\mathrm{GJ}}}
\newcommand{\GL}{{\mathrm{GL}}}

\newcommand{\Hom}{{\mathrm{Hom}}}

\newcommand{\Ind}{{\mathrm{Ind}}}

\newcommand{\Id}{{\mathrm{Id}}}

\newcommand{\ord}{{\mathrm{ord}}}

\newcommand{\rank}{{\mathrm{rank}}}

\renewcommand{\Re}{{\mathrm{Re}}}

\newcommand{\Res}{{\mathrm{Res}}}
\newcommand{\Rep}{{\mathrm{Rep}}}

\newcommand{\SL}{{\mathrm{SL}}}

\newcommand{\Sp}{{\mathrm{Sp}}}

\newcommand{\Span}{{\mathrm{Span}}}

\newcommand{\tr}{{\mathrm{tr}}}

\newcommand{\ud}{\,\mathrm{d}}
\newcommand{\vol}{{\mathrm{vol}}}

\newcommand{\ovl}{\overline}

\newcommand{\wt}{\widetilde}
\newcommand{\wh}{\widehat}

\newcommand{\bs}{\backslash}

\def\alp{{\alpha}}
\def\ac{\mathrm{ac}}
\def\bet{{\beta}}

\def\del{{\delta}}
\def\Del{{\Delta}}
\def\diag{{\rm diag}}

\def\veps{{\varepsilon}}

\def\sig{{\sigma}}

\def\zet{{\zeta}}
\def\std{\rm std}
\def\ome{{\omega}}
\def\Ome{{\Omega}}
\def\lam{{\lambda}}
\def\Lam{{\Lambda}}

\def\gam{{\gamma}}
\def\Gam{{\Gamma}}

\def\wb{\overline} 

\def\vpi{\varpi}

\def\vphi{\varphi}

\def\Rep{\mathrm{Rep}}

\newtheorem{thm}{Theorem}[section]
\newtheorem{dfn}[thm]{Definition}

\newtheorem{prp}[thm]{Proposition}
\newtheorem{lem}[thm]{Lemma}
\newtheorem{cor}[thm]{Corollary}
\newtheorem{ass}[thm]{Assumption}
\newtheorem{cnj}[thm]{Conjecture}

\makeatletter

\newcommand{\Rmnum}[1]{\expandafter\@slowromancap\romannumeral #1@}
\makeatother

\begin{document}

\title[Fourier Operators and Poisson Formulae on $\GL_1$]
{Certain Fourier Operators and their Associated Poisson Summation Formulae on $\GL_1$}

\author{Dihua Jiang}
\address{School of Mathematics\\
University of Minnesota\\
Minneapolis, MN 55455, USA}
\email{dhjiang@math.umn.edu}

\author{Zhilin Luo}
\address{Department of Mathematics\\
University of Chicago\\
Chicago IL, 60637, USA}
\email{zhilinchicago@uchicago.edu}



\thanks{The research of this paper is supported in part by the NSF Grant DMS--1901802.}

\date{\today}

\begin{abstract}
In this paper, we explore a possibility to utilize harmonic analysis on $\GL_1$ to understand Langlands
automorphic $L$-functions in general, as a vast generalization of the pioneering work of J. Tate (\cite{Tt50}).
For a split reductive group $G$ over a number field $k$, let $G^\vee(\BC)$ be its complex dual group and
$\rho$ be an $n$-dimensional complex representation of $G^\vee(\BC)$.
For any irreducible cuspidal automorphic representation
$\sig$ of $G(\BA)$, where $\BA$ is the ring of adeles of $k$, we introduce the space
$\CS_{\sig,\rho}(\BA^\times)$ of $(\sig,\rho)$-Schwartz functions on $\BA^\times$ and
$(\sig,\rho)$-Fourier operator $\CF_{\sig,\rho,\psi}$ that takes $\CS_{\sig,\rho}(\BA^\times)$ to
$\CS_{\wt{\sig},\rho}(\BA^\times)$, where $\wt{\sig}$ is the contragredient of $\sig$.
By assuming the local Langlands functoriality for the pair $(G,\rho)$, we show (Theorem \ref{thm:rho-AC}) that the $(\sig,\rho)$-theta functions
\[
\Theta_{\sig,\rho}(x,\phi):=\sum_{\alp\in k^\times}\phi(\alp x)
\]
converges absolutely for all $\phi\in\CS_{\sig,\rho}(\BA^\times)$, and state Conjecture \ref{cnj:main}
and Conjecture \ref{cnj:ref} on $(\sigma,\rho)$-Poisson summation formula on $\GL_1$.
Then we prove Conjectures \ref{cnj:main} and \ref{cnj:ref} when $G=\GL_n$ and $\rho$ is
the standard representation of $\GL_n(\BC)$ (Theorem \ref{thm:PSF}). The proof of Theorem \ref{thm:PSF}
uses substantially the local theory of Godement-Jacquet (\cite{GJ72}) for the standard $L$-functions of
$\GL_n$ and the Poisson summation formula for the classical Fourier transform on affine spaces. As an application of the $\GL_1$-harmonic analysis as developed in Sections 2-7, we
provide (Theorem \ref{zero}) a spectral interpretation of the critical zeros of the standard $L$-functions $L(s,\pi\times\chi)$ for any irreducible cuspidal automorphic representation $\pi$ of $\GL_n(\BA)$
and idele class character $\chi$ of $k$, which is a reformulation of Theorem 2 of \cite{S01} in the adelic framework of A. Connes in \cite{Cn99} and is an extension of
Theorem III.1 of \cite{Cn99} from the Hecke $L$-functions $L(s,\chi)$ to the automorphic $L$-functions $L(s,\pi\times\chi)$.
\end{abstract}

\maketitle
\tableofcontents

\section{Introduction}\label{sec-I}

Let $k$ be a number field and $\BA$ be the ring of adeles of $k$. It is well known that $\BA$ is a locally
compact abelian group and the diagonal embedding of $k$ into $\BA$ is a lattice, i.e. the image, which is
still denoted by $k$, is discrete and the quotient $k\bs\BA$ is compact. The classical theory of harmonic
analysis on the quotient $k\bs\BA$ leads to a great impact, via the famous 1950 Princeton Thesis of
J. Tate (\cite{Tt50}), to the modern development of Number Theory, especially to the theory of automorphic
$L$-functions.

In Tate's thesis, the classical Fourier transform and the associated Poisson summation formula are responsible
for the meromorphic continuation and global functional equation of the Hecke $L$-function $L(s,\chi)$ attached
to automorphic characters $\chi$ of $k^\times\bs\BA^\times$.

In their pioneering work in 1972, R. Godement and H. Jacquet extended the work of Tate on $L(s,\chi)$ to the standard automorphic $L$-function $L(s,\pi)$ attached to any irreducible cuspidal automorphic representation $\pi$ of
$\GL_n(\BA)$ (\cite{GJ72}). In their work, the Fourier transform and the associated Poisson summation
formula for $\RM_n(k)\bs\RM_n(\BA)$ are responsible for the meromorphic continuation and global functional equation of $L(s,\pi)$. Here $\RM_n$ denotes the space of all $n\times n$ matrices.

In 2000, A. Braverman and D. Kazhdan (\cite{BK00}) proposed that there should exist a generalized
Fourier transform $\CF_{\rho, \psi}$ on $G(\BA)$ for any reductive group $G$ defined over $k$ and any
finite dimensional complex representation $\rho$ of the $L$-group $^LG$; and if the associated Poisson
summation formula could be established, then there is a hope to prove the Langlands conjecture on
meromorphic continuation and global functional equation for automorphic $L$-function $L(s,\pi,\rho)$ attached to the pair $(\pi,\rho)$, where $\pi$ is any irreducible cuspidal automorphic representation of
$G(\BA)$.
In his 2020 paper (\cite{N20}), B. C. Ng\^o suggests that such generalized Fourier transforms could
be put in a framework that generalizes the classical Hankel transform for harmonic analysis on $\GL_1$ and
might be more useful in the trace formula approach to establish the Langlands conjecture of
functoriality in general.

In this paper, we explore a possibility that utilizes harmonic analysis on $\GL_1$ to understand Langlands
automorphic $L$-functions in general, as a vast generalization of the classical work of Tate (\cite{Tt50}).
For a $k$-split reductive group $G$, let $G^\vee(\BC)$ be its complex dual group and
$\rho$ be an $n$-dimensional complex representation of $G^\vee(\BC)$.
We introduce, for any irreducible cuspidal automorphic representation
$\sig$ of $G(\BA)$, the space
$\CS_{\sig,\rho}(\BA^\times)$ of $(\sig,\rho)$-Schwartz functions on $\BA^\times$ and the
$(\sig,\rho)$-Fourier operator $\CF_{\sig,\rho,\psi}$ that takes $\CS_{\sig,\rho}(\BA^\times)$ to
$\CS_{\wt{\sig},\rho}(\BA^\times)$, where $\wt{\sig}$ is the contragredient of $\sig$.
Then we show (Theorem \ref{thm:rho-AC}) that the $(\sig,\rho)$-theta functions
\begin{align}\label{Theta-1}
\Theta_{\sig,\rho}(x,\phi):=\sum_{\alp\in k^\times}\phi(\alp x)
\end{align}
converges absolutely for all $\phi\in\CS_{\sig,\rho}(\BA^\times)$. Our conjecture for the corresponding Poisson
summation formula is stated as follows.

\begin{cnj}[$(\sigma,\rho)$-Poisson Summation Formula]\label{cnj:main}
Let $\rho\colon G^\vee(\BC)\to\GL_n(\BC)$ be any finite dimensional representation of the complex
dual group $G^\vee(\BC)$.
For any given $\sigma\in\CA_\cusp(G)$, there exist nontrivial $k^\times$-invariant linear functionals
$\CE_{\sigma,\rho}$ and $\CE_{\wt{\sigma},\rho}$ on $\CS_{\sigma,\rho}(\BA^\times)$ and
$\CS_{\wt{\sigma},\rho}(\BA^\times)$, respectively, such that the
$(\sigma,\rho)$-Poisson Summation Formula:
\[
\CE_{\sigma,\rho}(\phi)
=
\CE_{\wt{\sigma},\rho}(\CF_{\sigma,\rho,\psi}(\phi))
\]
holds for $\phi\in\CS_{\sigma,\rho}(\BA^\times)$, where $\CS_{\sigma,\rho}(\BA^\times)$ and $\CF_{\sigma,\rho,\psi}$ are defined in Section \ref{ssec-AC-sig-rho-TF}.
\end{cnj}

It is expected that such Poisson summation formulae on $\GL_1$ should be
responsible for the Langlands conjecture on the global functional equation of automorphic $L$-functions
associated the pairs $(\sig,\rho)$. Variants of Conjecture \ref{cnj:main} will be discussed in
Section \ref{ssec-RCmain} and see
Conjecture \ref{cnj:ref} for details.

In this paper we prove Conjecture \ref{cnj:main} when $G=\GL_n$ and $\rho$
is the standard representation of $\GL_n(\BC)$ (Theorem \ref{thm:PSF}), which we call a $\pi$-Poisson
summation formula on $\GL_1$ for any irreducible cuspidal automorphic representation $\pi$ of $\GL_n(\BA)$.
A variant of Theorem \ref{thm:PSF} is established in Theorem \ref{thm:weakPSL2} when
$\pi$ is an irreducible square-integrable automorphic representation of $\GL_n(\BA)$.
It is important to
point out that by Theorem \ref{thm:PSF}, Conjecture \ref{cnj:main} holds for $(G,\rho)$ and an
irreducible cuspidal automorphic representation $\sig$ of $G(\BA)$ if the global Langlands functoriality is known for the pair $(G,\rho)$ and the functorial transfer to $\GL_n$ of $\sig$ is cuspidal.
Note that
in Theorem \ref{thm:PSF}, the $k^\times$-invaraint linear functional $\CE_{\sigma,\rho}(\phi)$ in this case
is taken to be the $\pi$-theta function $\Theta_{\pi}(1,\phi)$.
As explained in Corollary \ref{cor:CE}, according to the global Langlands functoriality, this happens in general,
except some singular situations of $\sigma$ with respect to $\rho$.
The desire is to
prove Conjecture \ref{cnj:main} without using the global Langlands functoriality. It is expected that
Conjecture \ref{cnj:main} can be proved directly for a split classical group $G$ and the standard
representation $\rho$ of the complex dual group $G^\vee(\BC)$, by using the doubling method of I. Piatetski-Shapiro and S. Rallis in \cite{GPSR87}) and the recent work of L. Zhang and the authors in \cite{JLZ20} and of
J. Getz and B. Liu in \cite{GL21}.

To formulate and prove the $\pi$-Poisson summation formula on $\GL_1$ for any irreducible
cuspidal automorphic representation $\pi$ of $\GL_n(\BA)$ (Theorem \ref{thm:PSF}), we consider the determinant morphism:
\begin{align}\label{Fa}
\det\colon\RM_n\to\BG_a; \quad \GL_n\to \BG_m,
\end{align}
where $\BG_a(k)=k$ and $\BG_m(k)=\GL_1(k)=k^\times$, and descend the local theory of Godement-Jacquet in \cite{GJ72} from $\GL_n$ to $\GL_1$. With the fully developed local theory for $\GL_1$ in hand, we are able to establish a $\pi$-Poisson summation formula on $\GL_1$ for each irreducible cuspidal automorphic
representation $\pi$ of $\GL_n(\BA)$, by descending the Poisson summation formula for the classical
Fourier transform on the affine space $\RM_n$. In consequence, the global functional equation of
the standard automorphic $L$-functions $L(s,\pi)$ can be established by applying harmonic analysis on
$\GL_1$.

We first develop the local theory for the localization $F=k_\nu$ of the number field $k$ at any local place $\nu$ of $k$. In Section \ref{sec-GJTR}, after recalling the local theory of Mellin transforms, mainly from
\cite[Chapter I]{Ig78}, we review the local theory of Godement-Jacquet (\cite{GJ72}) and reformulate it
in the framework of the Braverman-Kazhdan proposal (\cite{BK00}).

In Section \ref{sec-piSF-FO}, we fully develop the local theory of harmonic analysis on $\GL_1$ for
the Langlands local $L$-functions $L(s,\pi)$ and $\gam$-functions $\gam(s,\pi,\psi)$, attached to any
irreducible admissible representations $\pi$ of $\GL_n(F)$.
When $F$ is non-Archimedean, we take $\pi$ to be irreducible smooth representations of $\GL_n(F)$; and when $F$ is Archimedean, we take $\pi$ to be irreducible Casselman-Wallach representations of $\GL_n(F)$
(\cite{Cas89}, \cite{Wal92}, \cite{SZ11}, and \cite{BerK14}).
The set of equivalence classes of all such representations of $\GL_n(F)$ is denoted by $\Pi_F(n)$.

We introduce the space of $\pi$-Schwartz functions on $F^\times$ for any $\pi\in\Pi_F(n)$, which is denoted by $\CS_\pi(F^\times)$ (Definition \ref{def:piSS}). More precisely,
for any $\pi\in\Pi_F(n)$, the $\pi$-Schwartz functions are constructed through the fiber integration
\begin{align}
\phi_{\xi,\vphi_\pi}(x) :=
\int_{\det g=x}
\xi(g)\vphi_\pi(g)\ud_x g
\end{align}
as defined in \eqref{fibration}, where $\xi(g)=|\det g|_F^{\frac{n}{2}}\cdot f(g)$ with $f\in\CS(\RM_n(F))$, the space of the usual Schwartz functions
on $\RM_n(F)$, and $\vphi_\pi(g)\in\CC(\pi)$, the space of matrix coefficients of $\pi$. It is proved in
Proposition~\ref{prp:smooth} that the integral converges absolutely and defines a smooth function on $F^\times$. The first local result is Theorem \ref{thm:1-zeta}, which establishes the local
theory of zeta integrals on $\GL_1$ for the Langlands local $L$-function $L(s,\pi)$ for any
$\pi\in\Pi_F(n)$. Then we introduce in \eqref{eq:1-FO} the $\pi$-Fourier operator $\CF_{\pi,\psi}$ that
is a linear transformation from the $\pi$-Schwartz space
$\CS_\pi(F^\times)$ to the $\wt{\pi}$-Schwartz space $\CS_{\wt{\pi}}(F^\times)$, where $\wt{\pi}$ is the contragredient of $\pi$.
The second local result is Theorem \ref{thm:1-FE}, which establishes the local functional equation
for the $\GL_1$ zeta integrals with $\CF_{\pi,\psi}$ as the symmetry operator and the Langlands local
$\gam$-function $\gam(s,\pi,\psi)$ as the proportional constant for any $\pi\in\Pi_F(n)$. It is important
to point out that the local theory of Godement-Jacquet (\cite{GJ72}) serves as a base of the local theory
developed in this section.

In Section \ref{sec-PSF}, we develop the global theory of harmonic analysis on $\GL_1$ for the
standard automorphic $L$-functions $L(s,\pi)$ attached to any $\pi\in\CA_\cusp(\GL_n)$, the set of all equivalence classes of irreducible cuspidal automorphic
representations $\pi=\otimes_\nu\pi_\nu$ of $\GL_n(\BA)$. We introduce the $\pi$-Schwartz space
\[
\CS_\pi(\BA^\times)=\otimes_{\nu\in|k|}\CS_{\pi_\nu}(k_\nu^\times)
\]
in \eqref{piSS-BA}, where the restricted tensor product over the set $|k|$ of all local places of $k$ is taken with respect to the baisc function
$\BL_{\pi_\nu}$ (as defined in Theorem \ref{thm:1-zeta}) at almost all finite local places. Then we introduce the $\pi$-Fourier operator
\[
\CF_{\pi,\psi}(\phi)(x)=\prod_{\nu\in|k|}\CF_{\pi_\nu,\psi_\nu}(\phi_\nu)(x_\nu)
\]
in \eqref{FO-BA}, with $\phi=\otimes_\nu\phi_\nu\in\CS_\pi(\BA^\times)$. The main global result in this section is
Theorem \ref{thm:PSF}, which establishes the $\pi$-Poisson summation formula
\[
\sum_{\gam\in k^\times}\phi(\gam x)=\sum_{\gam\in k^\times}\CF_{\pi,\psi}(\phi)(\gam x^{-1})
\]
for any $\phi\in\CS_\pi(\BA^\times)$ and $x\in\BA^\times$. By applying the global Mellin transform,
we are able to obtain the global functional equation for the global zeta integrals on $\GL_1$, which
produces the global functional equation for the standard automorphic $L$-functions $L(s,\pi)$ for
any $\pi\in\CA_\cusp(\GL_n)$.

In order to understand the Poisson summation formulae in Conjecture \ref{cnj:main}, it is desirable to
explore variants of Theorem \ref{thm:PSF} when the automorphic representation $\pi$ may not be cuspidal,
from the point of view of the global Langlands functoriality. In Section \ref{sec-ACGTF}, we first show that
for any irreducible admissible representation $\pi$ of $\GL_n(\BA)$, which may not be automorphic, but
satisfies Assumption \ref{ass-FH}, the $\pi$-theta functions
\[
\Theta_\pi(x,\phi)=\sum_{\gam\in k^\times}\phi(\gam x)
\]
converge absolutely for any $\phi\in\CS_\pi(\BA^\times)$ and any $x\in\BA^\times$
(Theorem \ref{thm:AC}). Then we show that Assumption \ref{ass-FH} holds for any automorphic representation $\pi$ of
$\GL_n(\BA)$ (Proposition \ref{prp:Ass}). With Theorem~\ref{thm:AC}, we are ready to explore more general
situation in order to formulate Conjecture \ref{cnj:main} and its variant (Conjecture \ref{cnj:ref}).

For any $k$-split reductive group $G$ and any finite-dimensional representation $\rho$ of
the complex dual group $G^\vee(\BC)$, we define in Section \ref{ssec-AC-sig-rho-TF}
the relevant Schwartz spaces $\CS_{\sigma,\rho}(\BA^\times)$, called the $(\sigma,\rho)$-Schwartz space, in \eqref{globalSS}, and $(\sigma,\rho)$-Fourier operators $\CF_{\sigma,\rho,\psi}$ in \eqref{globalFO} for any
irreducible cuspidal automorphic representation $\sig$ of $G(\BA)$, under the assumption (Assumption \ref{FarS}) that the local Langlands reciprocity map exists for $G$ over all finite local places $\nu$ of $k$.
We prove in such a generality the absolute convergence of the $(\sig,\rho)$-theta function
$\Theta_{\sig,\rho}(x,\phi)$ as defined in \eqref{Theta-1}
for  any $\phi\in\CS_{\sig,\rho}(\BA^\times)$ and any $x\in\BA^\times$ (Theorem \ref{thm:rho-AC}).

In Section \ref{sec-VCmain}, after we establish a characterization of some special type $(\sig,\rho)$-Schwartz
functions in Theorem \ref{thm:testfunction} at all local places of $k$, we prove a variant of Theorem \ref{thm:PSF}
when $\pi$ is an irreducible square-integrable automorphic representation of $\GL_n(\BA)$
(Theorem \ref{thm:weakPSL2}). Finally we write down  a variant of Conjecture \ref{cnj:main} with more details in Conjecture \ref{cnj:ref}.

In order to understand the Poisson summation formulae in Conjecture \ref{cnj:main} and Conjecture \ref{cnj:ref}, we have to explore
and develop harmonic analysis on $\GL_1$ initiated by the $(\sig,\rho)$-Fourier operator
$\CF_{\sig,\rho,\psi}$
and the $(\sig,\rho)$-Schwartz space $\CS_{\sig,\rho}(\BA^\times)$, both locally and globally. In \cite{JL},
the authors take a first step to define for an irreducible admissible
representation $\pi\in\Pi_F(n)$ with $F=k_\nu$ a distribution $\pi$-kernel $k_{\pi,\psi}(x)$ on $F^\times$, which represents the relevant local $\gam$-function as the Mellin transform of the $\pi$-kernel function $k_{\pi,\psi}(x)$, and makes the relevant Fourier operator as a generalization of the classical Hankel transform,
which is the convolution integral operator with $k_{\pi,\psi}(x)$ as the kernel function. We hope that this study to
understand the Langlands automorphic $L$-functions (and the local Langlands $\gam$-functions) brings new insight towards analysis on $\GL_1$.

As an application of the $\GL_1$-harmonic analysis as developed in Section 2-7, we provide in Section \ref{sec-CZ} a spectral interpretation of the critical zeros of the automorphic $L$-functions
$l(s,\pi\times\chi)$ (Theorem \ref{zero}) for any irreducible cuspidal automorphic representation $\pi$ of $\GL_n(\BA)$ and any character $\chi$ of the idele class group of $k$.
It can be viewed as a reformulation of \cite[Theorem 2]{S01} in the adelic framework of A. Connes in \cite{Cn99} and an extension of \cite[Theorem III.1]{Cn99} from Hecke $L$-functions
$L(s,\chi)$ to automorphic $L$-functions $L(s,\pi\times\chi)$. The proof uses
a combination of arguments in \cite{S01}, and those in \cite{Cn99}, together with the results developed in
Sections 2-7. Further results along the line of \cite{Cn99} will be written in our forthcoming work.

We would like to thank Wee Teck Gan and Michael Harris for their questions and comments on the expression
on the local Langlands conjecture in the previous version of the paper and thank Herv\'e Jacquet and Stephen Kudla for their questions and comments on the subject matter of the paper.

\section{Godement-Jacquet Theory and Reformulation}\label{sec-GJTR}

\subsection{Mellin transforms}\label{ssec-LMT}
We recall the local theory of Mellin transforms from the book of Igusa (\cite[Chapter I]{Ig78}) and state them in a slightly more general situation in order to treat the case
that meromorphic functions may have poles that are not real numbers. Since the proofs are almost the same, we omit the details.

Let $F$ be a local field of characteristic zero. This means that it is either the complex field $\BC$, the real field $\BR$, or a finite extension of the $p$-adic field $\BQ_p$ for some prime $p$.

When $F$ is non-Archimedean, let $\Fo_F$ be the ring of integers with maximal ideal $\Fp_F$ and fix a uniformizer $\vpi_F$ of $\Fp_F$. Let $\Fo_F/\Fp_F = \kappa_F\simeq \BF_q$. Fix the norm $|x|_F = q^{-\ord_F(x)}$ where $\ord_F:F\to \BZ$ is the valuation on $F$ such that $\ord_F(\vpi_F)=1$. Fix the Haar measure $\ud^+x$ on $F$ so that $\vol(\ud^+x,\Fo_F)= 1$. Let $\psi =\psi_F$ be an additive character of $F$ which is trivial on $\Fo_F$ but non-trivial on $\vpi^{-1}_F\cdot \Fo_F$. In particular the standard Fourier transform defined via $\psi_F$ is self-dual w.r.t $\ud^+x$. Similarly, fix a multiplicative Haar measure $\ud^\times x$ on $F^\times$, which is normalized so that $\vol(\ud^\times x,\Fo^\times_F) = 1$. In particular $\ud^\times x = \frac{1}{\zet_F(1)}\cdot\frac{\ud^+x}{|x|_F}$, where $\zet_F(s)$ is the local Dedekind zeta factor attached to $F$.

When $F$ is Archimedean, fix the following norm on $F$,
$$
|z|_F = \bigg\{
\begin{matrix}
\text{absolute value of $z$}, & F=\BR\\
z\wb{z}, & F=\BC
\end{matrix}.
$$
Take the Haar measure $\ud^+x$ on $F$ that is the usual Lebesgue measure on $F$, and set
$$
\ud^\times x =
\bigg\{
\begin{matrix}
\frac{\ud^+x}{2|x|_F}, & F=\BR,\\
\frac{\ud^+x}{2\pi |x|_F}, & F=\BC
\end{matrix}
$$
the multiplicative Haar measures on $F^\times$.
The additive character $\psi = \psi_F$ of $F$ is chosen as follows
$$
\psi_F(x) =
\begin{cases}
\exp(2\pi ix), & F=\BR,\\
\exp(2\pi i(x+\wb{x})), &F=\BC.
\end{cases}
$$
For convenience, set the following norm on $F$,
$$
|\cdot| =
\begin{cases}
|\cdot|_F, & F\neq \BC,\\
|\cdot|_F^{\frac{1}{2}}, & F=\BC.
\end{cases}
$$
We denote by $\FX(F^\times)$ the set of all quasi-characters of $F^\times$. Define the topological group $\Omega_F$ to be $\{\pm 1\}$ if $F=\BR$, $\BC^\times_1$ if $F=\BC$, and the unit group
$\Fo_F^\times$ if $F$ is non-archimedean. It is clear that any $\chi\in\FX(F^\times)$ can be written as
\begin{align}\label{chi}
\chi(x)=\chi_u(x)=\chi_{u,\ome}(x)=
|x|_F^u\omega(\ac(x)),
\end{align}
for any $x\in F^\times$, with $u\in\BC$ and $\ome\in\Ome_F^\wedge$, the Pontryagin dual of $\Ome_F$. Here $\ac(x)=\frac{x}{|x|_F}\in\Fo_F^\times$, if $F$ is non-Archimedean, and
\[
\ac(x) =
\begin{cases}
\frac{x}{|x|_F}\in \{\pm1\}, & F=\BR,\\
\frac{x}{|x|} = \frac{x}{|x|^{1/2}_F}\in \BC^\times_1, & F=\BC.
\end{cases}
\]
It is clear that the unitary character $\omega$ of $\Ome_F$ is uniquely determined by $\chi\in\FX(F^\times)$, in particular, we have
\begin{align}\label{ome}
\omega(\ac(x))=\ac(x)^p,
\end{align}
with $p\in\{0,1\}$ if $F=\BR$ and $p\in\BZ$ if $F=\BC$. Hence we may sometimes write $\chi=(u,\ome)$ and $\ome(x)=\ome(\ac(x))$ for $x\in F^\times$.

For any local field $F$ of characteristic zero, following \cite[Sections I.4 and I.5]{Ig78}, we define the following two spaces of functions associated to the local field $F$.

\begin{dfn}\label{dfn:CF}
Let $\CF(F^\times)$ be the space of complex-valued functions $\Ff$ such that
\begin{enumerate}
\item $\Ff\in \CC^\infty(F^\times)$, the space of all smooth functions on $F^\times$.
\item When $F$ is non-Archimedean, $\Ff(x)=0$ for $|x|_F$ sufficiently large. When $F$ is Archimedean, we denote $\Ff^{(n)}:=\frac{\ud^n\Ff}{\ud x^n}$ if $F=\BR$, and
$\Ff^{(n)}=\Ff^{(a+b)}:=\frac{\partial^{a+b}\Ff}{\partial^ax\partial^b\wb{x}}$ if $F=\BC$ and $n=a+b$. Then we have
\[
\Ff^{(n)}(x) = \mathrm{o}(|x|_F^\rho)
\]
as $|x|_F\to \infty$ for any $\rho$ and any $n=a+b\in \BZ_{\geq 0}$ with $a,b\in\BZ_{\geq 0}$.
\item When $F$ is Archimedean, there exists
\begin{itemize}
\item
a sequence $\{m_k\}_{k=0}^\infty$ of positive integers,
\item
a sequence of smooth functions $\{a_{k,m}\}$ on $\{\pm 1\}$ if $F=\BR$ and on $\BC^\times_1$ if $F=\BC$, parametrized by $m=1,2,...,m_k$ and $k\in \BZ_{\geq 0}$,
\item
a sequence $\{\lam_k\}_{k=0}^\infty$ of complex numbers with $\{\Re(\lam_k)\}_{k=0}^\infty$ a strictly increasing sequence of real numbers with no finite accumulation point and  $\Re(\lam_0)\geq \lam\in \BR$,
\end{itemize}
such that
$$
\lim_{|x|_F\to 0}
\bigg\{
\Ff(x)-\sum_{k=0}^\infty
\sum_{m=1}^{m_k}
a_{k,m}(\ac(x))|x|_F^{\lam_k}
(\ln(|x|_F))^{m-1}
\bigg\} =0.
$$
The limit is term-wise differentiable and uniform (even after term-wise differentiation) in $\ac(x)$.

When $F$ is non-Archimedean, one can take the sequence $\{\lam_k\}$ to be a finite set $\Lam$ and the sequence $\{m_k\}$ to be a finite subset of $\BZ_{\geq0}$. The smooth functions $\{a_{k,m}(\ac(x))\}$
are on the unit group $\Fo_F^\times$.
\end{enumerate}
\end{dfn}
Since the topological group $\Omega$ is compact and abelian,
we have the Fourier expansion for the smooth functions
$\{a_{k,m}(\ac(x))\}$ on $\Omega$:
\[
a_{k,m}(\ac(x)) = \sum_{\omega\in\Omega^\wedge}a_{k,m,\omega}\omega(\ac(x)).
\]
In the Archimedean case, we may write $a_{k,m,\omega}=a_{k,m,p}$ with $p\in\{0,1\}$ if $F=\BR$ and $p\in\BZ$ if $F=\BC$.

\begin{dfn}\label{dfn:CZ}
With the same notation as in Definition \ref{dfn:CF}, let $\CZ(\FX(F^\times))$ be the space of complex-valued functions $\Fz(\chi_{s,\ome})=\Fz(|\cdot|_F^s\ome(\ac(\cdot)))$ on $\FX(F^\times)$ such that
\begin{enumerate}
\item $\Fz(\chi_{s,\ome})$ is meromorphic on $\FX(F^\times)$ with poles at most for $s=-\lam_j$ with $\lam_j$ belonging to the given set  $\{\lam_k\}_{k=0}^\infty$
if $F$ is Archimedean; and belonging to the given finite set $\Lam$ if $F$ is non-Archimedean.
\item
For any $k\geq 0$, the difference
\[
\Fz(\chi_{s,\ome})-\sum_{m=1}^{m_k}\frac{b_{k,m,\ome}}{(s+\lam_k)^m}
\]
is holomorphic for $s$ in a small neighborhood of $-\lam_k$ if $F$ is Archimedean; and is
a polynomial in $\BC[q^s,q^{-s}]$ if $F$ is non-Archimedean.
\item
When $F$ is non-Archimedean, the function $\Fz(\chi_{s,\ome})$ is identically zero for almost all characters $\ome\in\Ome^\wedge$ with $\Ome=\Fo_F^\times$. When $F$ is Archimedean,
for every polynomial $P(s,p)$ in $s,p$ with coefficients in $\BC$, and every pair of real numbers $a<b$, the function $P(s,p)\Fz(\chi_{s,\ome})$ is bounded when $s$ belongs to the vertical strip
\begin{align}\label{eq:vst}
S_{a,b}=\{s\in\BC\mid a\leq\Re(s)\leq b\},
\end{align}
with neighborhoods of $-\lam_0,-\lam_1,...$ removed therefrom. More precisely, there exists a constant $c$ depending only on $P$, $\Fz$, $a,b$, but neither on $s$ nor on $p$, such that
$$
|P(s,p)\Fz(\chi_{s,\ome})|\leq c
$$
when $s$ runs in the vertical strip $S_{a,b}$ with small neighborhoods of $-\lam_0,-\lam_1,...$ removed.
\end{enumerate}
\end{dfn}

The main results on the local theory of Mellin transforms established in \cite[Chapter I]{Ig78} is as follows.

\begin{thm}[Mellin Transforms]\label{thm:MT}
There is a bijective linear correspondence $\CM = \CM_F$ between the space $\CF(F^\times)$ and the space $\CZ(\FX(F^\times))$. More precisely, for $\Ff\in \CF(F^\times)$,
$$
\CM(\Ff)(\chi_{s,\ome}) = \int_{F^\times}
\Ff(x)\chi_{s,\ome}(x)\ud^\times x
$$
defines a holomorphic function on
\[
\FX_{-\sig_0}(F^\times) = \{\chi_{s,\ome}(\cdot)=|\cdot|_F^s\ome(\ac(\cdot))\in \FX(F^\times)\mid \Re(s)>-\sig_0\}
\]
for some $\sig_0\in\BR$,
which has a meromorphic continuation to all characters $\chi_{s,\ome}\in\FX(F^\times)$
and belongs to $\CZ(\FX(F^\times))$ after meromorphic continuation.
Conversely, for $\Fz\in \CZ(\FX(F^\times))$ and $x\in F^\times$, the Mellin inverse transform $\CM_F^{-1}(\Fz)(x)$ belongs to the space $\CF(F^\times)$. Moreover, the following identities hold
\[
\CM(\CM^{-1}(\Fz))=\Fz, \qquad \CM^{-1}(\CM(\Ff))=\Ff
\]
for any $\Ff\in\CF(F^\times)$ and $\Fz\in\CZ(\FX(F^\times))$. Here the Mellin inverse transform is explicitly given as follows.

When $F$ is Archimedean, the Mellin inverse transform $\CM_F^{-1}(\Fz)(x)$ is given by
\begin{align}\label{eq:MIT-ar}
\CM^{-1}(\Fz)(x) :=
\sum_{\ome\in\Ome_F^\wedge}
\frac{1}{2\pi i}
\int^{\sig+i\infty}_{\sig-i\infty}
\Fz(\chi_{s,\ome})
\chi_{s,\ome}(x)^{-1}\ud s
\end{align}
with $\ome(\ac(x))=\ac(x)^p$, which defines a function $\Ff$ in $\CF(F^\times)$ independent of $\sig>-\sig_0$, and the coefficients $a_{k,m,p}$ and $b_{k,m,p}$ satisfy the following relations:
$$
b_{k,m,p} = (-1)^{m-1}(m-1)!\cdot a_{k,m,-p}
$$
for every $k\geq 0,m\geq 1$ with $p\in\{0,1\}$ if $F=\BR$ and $p\in\BZ$ if $F=\BC$. The coefficients $a_{k,m,p}$ and $b_{k,m,p}$ satisfy the following relations:
\[
b_{\lam,m,\ome}=\sum_{j=m}^{m_\lam}e_{j,m}(-\ln q)^{j-1}a_{\lam,j,\ome^{-1}}
\]
with $e_{j,m}$ defined by the following identity of polynomials in a formal unknown $t$
\[
t^{n-1}=\sum_{\ell=1}^ne_{n,\ell}\begin{pmatrix}t+\ell-1\\\ell-1\end{pmatrix}.
\]

When $F$ is non-Archimedean, the Mellin inverse transform $\CM_F^{-1}(\Fz)(x)$ is given by
\begin{align}\label{eq:MIT-na}
\CM_F^{-1}(\Fz)(x)
:=
\sum_{\ome\in\Ome^\wedge}\bigg(\Res_{z=0}(\Fz(\chi_{s,\ome})|x|_F^{-s}q^s)\bigg)\ome(\ac(x))^{-1},
\end{align}
which defines a function $\Ff$ in $\CF(F^\times)$. Here $z=q^{-s}$ for abbreviation.
\end{thm}

\subsection{Local theory of Godement-Jacquet}\label{ssec-GJ}

Let $\RG_n:=\GL_n$ be the general linear group defined over $F$. Fix the following maximal (open if $F$ is non-Archimedean) compact subgroup $K$ of $\RG_n(F)=\GL_n(F)$,
\begin{align}\label{K}
K=
\begin{cases}
\GL_n(\Fo_F), & F\text{ is non-Archimedean};\\
\RO(n), &  F=\BR;\\
\RU(n), & F=\BC.
\end{cases}
\end{align}
Fix the Haar measure $\ud g = \frac{\ud^+g}{|\det g|_F^n}$ on $\RG_n(F)$ where $\ud^+g$ is the measure induced from the standard additive measure on $\RM_n(F)$, the $F$-vector space of $n\times n$-matrices. In particular, $\RG_n(F)$ embeds into $\RM_n(F)$ in a standard way.

Let $\Pi_F(n)$ be the set of equivalence classes of irreducible smooth representations of $\RG_n(F)$ when $F$ is non-Archimedean;
and of irreducible Casselman-Wallach representations of $\RG_n(F)$ when $F$ is Archimedean.
Set $\CC(\pi)$ the space of smooth matrix coefficients attached to $\pi$.

Let $\CS(\RM_n(F))$ be the space of the standard Schwartz-Bruhat functions on $\RM_n(F)$. The standard Fourier transform $\CF_\psi$ acting on $\CS(\RM_n(F))$ is defined as follows,
\begin{equation}\label{eq:FTMAT}
\CF_\psi(f)(x) = \int_{\RM_n(F)}
\psi(\tr(xy))f(y)\ud^+y
\end{equation}
where $\psi$ is a non-trivial additive character of $F$.
Moreover,  the standard Fourier transform $\CF_\psi$ extends to a unitary operator on the space $L^2(\RM(F),\ud^+x)$ and satisfies the following identity:
\begin{equation}\label{eq:FTId}
\CF_\psi\circ \CF_{\psi^{-1}}  =\Id.
\end{equation}

For any $\pi\in \Pi_F(n)$ and any quasi-character $\chi\in\FX(F^\times)$, the local zeta integral of Godement and Jacquet is defined by
\begin{equation}\label{eq:zi-GJ}
\CZ(s,f,\vphi_\pi,\chi) =
\int_{\RG_n(F)}
f(g)\vphi_\pi(g)\chi(\det g)|\det g|_F^{s+\frac{n-1}{2}}\ud g,
\end{equation}
for any $f\in\CS(\RM_n(F))$ and $\varphi_\pi\in\CC(\pi)$.
The following theorem contains the main results of the local theory for the Godement and Jacquet zeta integrals (\cite[Chapter I]{GJ72}).

\begin{thm}\label{thm:GJ}
With the notation introduced above, the following statements hold for any $f\in\CS(\RM_n(F))$ and $\vphi_\pi\in\CC(\pi)$.
\begin{enumerate}
\item The zeta integral $\CZ(s,f,\vphi_\pi,\chi)$ defined in \eqref{eq:zi-GJ} is absolutely convergent for $\Re(s)$ sufficiently large and admits a meromorphic continuation to $s\in \BC$.
\item $\CZ(s,f,\vphi_\pi,\chi)$ is a holomorphic multiple of the Langlands local $L$-function $L(s,\pi\times \chi)$ associated to $(\pi,\chi)$ and the standard embedding
$$
{\std}\colon\GL_n(\BC)\times\GL_1(\BC)\rightarrow\GL_n(\BC).
$$
Moreover, when $F$ is non-Archimedean, the fractional ideal $\CI_{\pi,\chi}$ that is generated by the local zeta integrals $\CZ(s,f,\vphi_\pi,\chi)$ is of the form:
\begin{align*}
\CI_{\pi,\chi} = \{
\CZ(s,f,\vphi_\pi,\chi)\mid f\in \CS(\RM_n(F)),\vphi_\pi\in \CC(\pi)
 \}
=L(s,\pi\times \chi)\cdot \BC[q^s,q^{-s}];
\end{align*}
and when $F$ is Archimedean, the local zeta integrals $\CZ(s,f,\vphi_\pi,\chi)$, with unitary characters $\chi$, have the following property: Let $S_{a,b}$ be the vertical strip for any $a<b$, as defined in \eqref{eq:vst}.
If $P_\chi(s)$ is a polynomial in $s$ such that the product $P_\chi(s)L(s,\pi\times\chi)$ is bounded in the vertical strip $S_{a,b}$, then the product $P_\chi(s) \CZ(s,f,\vphi_\pi,\chi)$ must be bounded in the same
vertical strip $S_{a,b}$.
\item The local functional equation
$$
\CZ(1-s,\CF_\psi(f),\vphi^\vee_\pi,\chi^{-1}) =
\gam(s,\pi\times \chi,\psi)
\cdot \CZ(s,f,\vphi_\pi,\chi)
$$
holds after meromorphic continuation,
where the function $\vphi^\vee_\pi(g):= \vphi_\pi(g^{-1})\in \CC(\wt{\pi})$, and $\gam(s,\pi\times \chi,\psi)$ is the Langlands local gamma function associated to $(\pi,\chi)$ and $\std$.
\item When $F$ is non-Archimedean and $\pi$ is unramified, take $f^\circ(g)={\bf 1}_{\RM_n(\Fo_F)}(g)$ to
be the characteristic function of $\RM_n(\Fo_F)$ and $\varphi_\pi^\circ(g)$ to be the zonal spherical function associated to $\pi$, then the following identity:
\[
\CZ(s,f^\circ,\vphi_\pi^\circ,\chi)=L(s,\pi\times\chi)
\]
holds for any unramified characters $\chi$ and all $s\in\BC$ as meromorphic functions in $s$.
\end{enumerate}
\end{thm}
For the statements of the current version of Theorem \ref{thm:GJ}, we have some comments in order. When $F$ is non-Archimedean, the theorem is \cite[Theorem 3.3]{GJ72}.
When $F$ is Archimedean, the statements were established in \cite{GJ72} only for $K$-finite vectors $f$ in $\CS(\RM_n(F))$ and $\vphi_\pi$ in $\CC(\pi)$, and were extended to general smooth vectors in
\cite[Section 4.7]{J09} and also in \cite[Theorem 3.10]{Li19}. About the boundedness on vertical strips, we refer to \cite[Section 4]{J09}.

\subsection{Reformulation of Godement-Jacquet theory}\label{ssec-RGJT}

The local theory of Godement-Jacquet can be reformulated within harmonic analysis and $L^2$-theory.

For $f\in\CS(\RM_n(F))$, we define
\begin{align}\label{GJ-SF}
\xi_f(g):=|\det g|_F^{\frac{n}{2}}\cdot f(g)
\end{align}
for $g\in\RG_n(F)$. Then we define the Schwartz space on $\RG_n(F)$ to be
\begin{align}\label{GJ-SS}
\CS_{\std}(\RG_n(F)):=\{\xi\in\CC^\infty(\RG_n(F))\mid \ |\det g|^{-\frac{n}{2}}\cdot\xi(g)\in\CS(\RM_n(F))\}.
\end{align}

\begin{prp}\label{prp:GJ-L2}
The Schwartz space $\CS_{\std}(\RG_n(F))$ is a subspace of $L^2(\RG_n(F),\ud g)$, which is the space of
square-integrable functions on $\RG_n(F)$.
\end{prp}

\begin{proof}
For $\xi\in\CS_{\std}(\RG_n(F))$, write $\xi(g)=|\det g|_F^{\frac{n}{2}}\cdot f(g)$ for some
$f\in\CS(\RM_n(F))$. We deduce the square-integrability of $\xi$ by the following computation:
\[
\int_{\RG_n(F)}\xi(g)\ovl{\xi(g)}\ud g
=
\int_{\RG_n(F)}f(g)\ovl{f(g)}\ud^+ g=\int_{\RM_n(F)}f(g)\ovl{f(g)}\ud^+ g<\infty.
\]
\end{proof}

Define the distribution kernel in the local theory of Godement-Jacquet to be
\begin{align}\label{GJ-kernel}
\Phi_{\GJ}(g):=\psi(\tr g)\cdot|\det g|_F^{\frac{n}{2}}
\end{align}
where $\psi$ is a non-trivial additive character of $F$.
We compute the convolution $\Phi_{\GJ}*\xi^\vee$ for any $\xi\in\CS_{\std}(\RG_n(F))$
with $\xi(g)=|\det g|_F^{\frac{n}{2}}\cdot f(g)$ for some $f\in\CS(\RM_n(F))$:
\begin{align*}
\Phi_{\GJ}*\xi^\vee(g)
&=
\int_{\RG_n(F)}\Phi_{\GJ}(h)\xi(g^{-1}h)\ud h
=
\int_{\RG_n(F)}\psi(\tr h)\cdot|\det h|_F^{\frac{n}{2}}\cdot|\det g^{-1}h|_F^{\frac{n}{2}}\cdot f(g^{-1}h)\ud h\\
&=
\int_{\RG_n(F)}f(h)\psi(\tr gh)\cdot|\det gh|_F^{\frac{n}{2}}\cdot|\det h|_F^{\frac{n}{2}}\ud h
=
|\det g|_F^{\frac{n}{2}}\int_{\RM_n(F)}f(h)\psi(\tr gh)\ud^+ h\\
&=
|\det g|_F^{\frac{n}{2}}\cdot\CF_\psi(f)(g).
\end{align*}
Since $\CF_\psi(f)(g)$ belongs to $\CS(\RM_n(F))$, by definition, we must have that
$\Phi_{\GJ}*\xi^\vee(g)$ belongs to $\CS_{\std}(\RG_n(F))$.
We define the Fourier operator $\CF_{\GJ}$ in the Godement-Jacquet theory to be
\begin{align}\label{GJ-FO}
\CF_{\GJ}(\xi)(g):=\left(\Phi_{\GJ}*\xi^\vee\right)(g)
\end{align}
for any $\xi\in\CS_{\std}(\RG_n(F))$.

\begin{prp}\label{prp:GJ-FO}
For any $\xi\in\CS_{\std}(\RG_n(F))$
with $\xi(g)=|\det g|_F^{\frac{n}{2}}\cdot f(g)$ for some $f\in\CS(\RM_n(F))$, the Fourier operator
$\CF_{\GJ}$ on $\CS_{\std}(\RG_n(F))$ and the classical Fourier transform $\CF_\psi$ on
$\CS(\RM_n(F))$ are related by the following identity:
\[
\CF_{\GJ}(\xi)(g)=\left(\Phi_{\GJ}*\xi^\vee\right)(g)
=|\det g|_F^{\frac{n}{2}}\cdot\CF_\psi(f)(g)=
|\det g|_F^{\frac{n}{2}}\cdot\CF_\psi(|\det(\cdot)|^{-\frac{n}{2}}\xi)(g)
\]
\end{prp}

For any $\xi\in\CS_{\std}(\RG_n(F))$ with $\xi(g)=|\det g|_F^{\frac{n}{2}}\cdot f(g)$ for some $f\in\CS(\RM_n(F))$, the zeta integral can be renormalized as
\begin{align}\label{GJ-Zeta}
\CZ(s,f,\vphi_\pi,\chi) =
\int_{\RG_n(F)}
|\det g|_F^{\frac{n}{2}}f(g)\vphi_\pi(g)\chi(\det g)|\det g|_F^{s-\frac{1}{2}}\ud g
=
\CZ(s,\xi,\vphi_\pi,\chi).
\end{align}
We compute the other side of the functional equation of the Godement-Jacquet zeta integrals:
\begin{align*}
\CZ(1-s,\CF_\psi(f),\vphi_\pi^\vee,\chi^{-1})
&=
\int_{\RG_n(F)}
|\det g|_F^{\frac{n}{2}}\CF_\psi(f)(g)\vphi_\pi^\vee(g)\chi^{-1}(g)|\det g|_F^{\frac{1}{2}-s}\ud g\\
&=
\int_{\RG_n(F)}
\CF_{\GJ}(\xi)(g)
\vphi_\pi^\vee(g)\chi^{-1}(g)|\det g|_F^{\frac{1}{2}-s}\ud g\\
&=
\CZ(1-s, \CF_{\GJ}(\xi), \vphi_\pi^\vee,\chi^{-1}).
\end{align*}

\begin{prp}\label{prp:GJ-FE}
For any $\xi\in\CS_{\std}(\RG_n(F))$, $\vphi_\pi\in\CC(\pi)$, and $\chi\in\FX(F^\times)$,
the zeta integral defined by
\[
\CZ(s,\xi,\vphi_\pi,\chi)
=
\int_{\RG_n(F)}
\xi(g)\vphi_\pi(g)\chi(\det g)|\det g|_F^{s-\frac{1}{2}}\ud g
\]
satisfies the functional equation:
\[
\CZ(1-s, \CF_{\GJ}(\xi), \vphi_\pi^\vee,\chi^{-1})
=
\gam(s,\pi\times\chi,\psi)\cdot
\CZ(s,\xi,\vphi_\pi,\chi),
\]
which holds as meromorphic functions in $s$.
\end{prp}

We are going to understand the Godement-Jacquet distribution $\Phi_{\GJ}$ in terms of the
Bernstein center of $\RG_n(F)$, when $F$ is non-Archimedean.
Recall from \cite{Ber84} that the Bernstein center $\FZ(G(F))$ of a reductive group $G(F)$ over a non-Archimedean local field $F$ is defined to be the endomorphism ring of the identity functor on the category of smooth representations of $G(F)$.
It turns out that the Bernstein center $\FZ(G(F))$ can be identified with the space of invariant and essentially compactly supported distributions on $G(F)$,
where an invariant distribution $\Phi$ on $G(F)$ is called essentially compactly supported if $\Phi*\CC^\infty_c(G(F))\subset \CC^\infty_c(G(F))$. It was proved in \cite{Ber84} that
through Plancherel transform, the Bernstein center $\FZ(G(F))$ can also be identified with the space of regular functions on the Bernstein variety $\Ome(G(F))$ attached to $G(F)$, where $\Ome(G(F))$ is
an infinite disjoint union of finite dimensional complex algebraic varieties.

\begin{prp}\label{prp:BC-GJ}
Let $F$ be a non-Archimedean local field of characteristic zero. For  any $m\in \BZ$, define
$$
\RG_n(F)_m = \{g\in \RG_n(F)\mid \quad|\det g|_F = q^{-m}_F\}.
$$
Let ${\bf 1}_m:=\mathbf{1}_{\RG_n(F)_m}$ be the characteristic function of
$\RG_n(F)_m\subset \RG_n(F)$.
Then the following statements hold.
\begin{enumerate}
\item The invariant distribution
\begin{align}\label{GJ-Phim}
\Phi_{\GJ,m}(g): = \Phi_{\GJ}(g)\mathbf{1}_{\RG_n(F)_m}(g)=\Phi_{\GJ}(g){\bf 1}_m(g)
\end{align}
lies in the Bernstein center $\FZ(\RG_n(F))$ of $\RG_n(F)$.
\item
Let $f_{\GJ,m}$ be the regular function on $\Ome(\RG_n(F))$ attached to $\Phi_{\GJ,m}\in \FZ(\RG_n(F))$. For every $\pi\in\Pi_F(n)$, $\chi\in\FX(F^\times)$, and $s\in\BC$,
define
\[
\pi_{\chi_s}:=\pi\otimes\chi_s=\pi\otimes\chi(\det) |\det|^s_F.
\]
Then the following Laurent series
$$
f_{\GJ}(\pi_{\chi_s})= \sum_{m\in \BZ}f_{\GJ,m}(\pi_{\chi_s})
$$
is convergent for $\Re(s)$ sufficiently large, with a meromorphic continuation to $s\in \BC$, and
$$
f_{\GJ}(\pi_{\chi_s})=\gam(\frac{1}{2},\wt{\pi_{\chi_s}},\psi)
=
\gam(\frac{1}{2}-s,\wt{\pi}\times\chi^{-1},\psi).
$$
\end{enumerate}
\end{prp}

\begin{proof}
For Part (1), we have to show that the invariant distribution $\Phi_{\GJ,m}(g)$ is essentially compact
on $\RG_n(F)$. By a simple reduction, it suffices to show that, for any open compact subgroup $\CK$ of $\RG_n(\Fo_F)$, we have
\[
\Phi_{\GJ,m}*{\bf 1}_\CK\in\CC^\infty_c(\RG_n(F)).
\]
Since ${\bf 1}_\CK(g)={\bf 1}_\CK(g^{-1})={\bf 1}_\CK^\vee(g)$,
the convolution $\Phi_{\GJ,m}*{\bf 1}_\CK=\Phi_{\GJ,m}*{\bf 1}_\CK^\vee$ can be written as
\begin{align*}
\Phi_{\GJ,m}*{\bf 1}_\CK^\vee(g)
&=
\int_{\RG_n(F)}
\Phi_{\GJ,m}(h){\bf 1}_\CK(g^{-1}h)\ud h
=
\int_{\RG_n(F)}
\Phi_{\GJ,m}(gh){\bf 1}_\CK(h)\ud h\\
&=
\int_{\RG_n(F)}
\psi(\tr gh)|\det gh|^{\frac{n}{2}}{\bf 1}_m(gh){\bf 1}_\CK(h)\ud h.
\end{align*}
By definition, ${\bf 1}_\CK(h)\neq 0$ if and only if $|\det h|_F=1$, and ${\bf 1}_m(gh)\neq 0$ if and only if
$|\det g|_F=q_F^{-m}$, i.e. $g\in\RG_n(F)_m$. This implies that ${\bf 1}_m(gh)={\bf 1}_m(g)$.
The last integral can be written as
\[
q^{-\frac{mn}{2}}_F\mathbf{1}_m(g)
\int_{\RG_n(F)}
\psi(\tr(gh))
\mathbf{1}_\CK(h)\ud h
\]
which can be written as
\begin{align*}
q_F^{-\frac{mn}{2}}\mathbf{1}_m(g)
\int_{\RM_n(F)}
\psi(\tr(gh))
\mathbf{1}_\CK(h)\ud^+ h
=q_F^{-\frac{mn}{2}}\mathbf{1}_m(g)\CF_\psi({\bf 1}_\CK)(g).
\end{align*}
Hence we obtain that
\[
\Phi_{\GJ,m}*{\bf 1}_\CK(g)=|\det g|_F^{\frac{n}{2}}\mathbf{1}_m(g)\CF_\psi({\bf 1}_\CK)(g).
\]
Since $\CF_\psi({\bf 1}_\CK)(g)\in\CS(\RM_n(F))$ and $|\det g|_F^{\frac{n}{2}}\mathbf{1}_m(g)$ is smooth
on $\RM_n(F)$, we obtain that the convolution
$\Phi_{m,\psi}*\mathbf{1}_\CK(g)$ belongs to $\CC_c^\infty(\RG_n(F))$ and
the invariant distribution $\Phi_{\GJ,m}(g)$ is essentially compact on $\RG_n(F)$.

For Part (2), recall from \cite{Ber84} that the regular function $f_{\GJ,m}$ attached to $\Phi_{\GJ,m}$
is defined as follows. For any $\pi\in \Pi_F(n)$ and $v\in \pi$, there exists an open compact subgroup
$\CK$ of $\RG_n(F)$, such that $v\in \pi^\CK$, the subspace of $\CK$-fixed vectors in $\pi$.
We may define an action of $\Phi_{\GJ,m}$ on $\pi$ via
\begin{align}\label{Phi-action-v}
\pi(\Phi_{\GJ,m})(v ):=
\pi(\Phi_{\GJ,m}*\Fc_\CK)(v),
\end{align}
where $\Fc_\CK:=\vol(\CK)^{-1}{\bf 1}_\CK$ is the normalized characteristic function of $\CK$.
Since $\Phi_{\GJ,m}*\Fc_\CK$ lies in $\CC^\infty_c(\RG_n(F))$, the right-hand side is well-defined, so is the left-hand side. It is clear that the action defined in \eqref{Phi-action-v} does not depend on the choice of
such an open compact subgroup $\CK$.
By Schur's lemma, there exists a constant $f_{\GJ,m}(\pi)$, depending on $\pi$, such that
\begin{align}\label{Phi-action-pi}
\pi(\Phi_{\GJ,m})= f_{\GJ,m}(\pi)\cdot \Id_\pi.
\end{align}
For each $m\in \BZ$, we define that for any $\xi\in \CC^\infty_c(\RG_n(F))$,
\begin{align}\label{GJ-FOm}
\CF_{\GJ,m}(\xi)(g) :=
(\Phi_{\GJ,m}*\xi^\vee)(g)
=
\int_{\RG_n(F)}
\Phi_{\GJ,m}(h)\xi(g^{-1}h)\ud h.
\end{align}

In order to include the quasi-characters $\chi\in\FX(F^\times)$ in the gamma function, we write
\begin{align}\label{chi-coeff}
\vphi_{\pi[\chi]}(g):=\vphi_\pi[\chi](g):=\vphi_\pi(g)\chi(\det g)=(\chi(g)\pi(g)v,\wt{v}),
\end{align}
with $v\in V_\pi$ and $\wt{v}\in V_{\wt{\pi}}$,
which is a matrix coefficient of $\pi$ twisted by $\chi$. We may denote the space of such twisted matrix coefficients of $\pi$ by $\CC(\pi[\chi])$. It is clear that we have
\[
\CZ(s,\xi,\vphi_{\pi[\chi]})=\CZ(s,\xi,\vphi_{\pi},\chi).
\]
For each $m\in \BZ$, $\vphi_{\pi[\chi]}\in \CC(\pi[\chi])$, and $\chi\in\FX(F^\times)$, consider the zeta function of Godement-Jacquet with $\CF_{\GJ,m}(\xi)$ as defined in \eqref{GJ-FOm}
\begin{align}\label{38-1}
\CZ(1-s,\CF_{\GJ,m}(\xi),\vphi^\vee_{\pi[\chi]})
=\CZ(1-s,\Phi_{\GJ,m}*\xi^\vee,\vphi^\vee_{\pi[\chi]}).
\end{align}
By Part (1) as proved above, we obtain that $\Phi_{\GJ,m}*\xi^\vee\in \CC^\infty_c(\RG_n(F))$ for
any $\xi\in \CC^\infty_c(\RG_n(F))$. Hence the integral in \eqref{38-1} is absolutely
convergent for any $s\in \BC$ when $\xi\in \CC^\infty_c(\RG_n(F))$.

We write the right-hand side of \eqref{38-1} as
\begin{align}\label{38-2}
\int_{\RG_n(F)}\Phi_{\GJ,m}*\xi^\vee(g)\vphi_\pi(g^{-1})\chi^{-1}(\det g)|\det g|_F^{\frac{1}{2}-s}
\ud g,
\end{align}
which is equal to
\begin{align}\label{38-3}
\int_{\RG_n(F)}\Phi_{\GJ,m}*\xi^\vee(g)
(v,\wt{\pi}(g)\wt{v})\chi_{s-\frac{1}{2}}(\det g)^{-1}
\ud g
=
(v,\int_{\RG_n(F)}\Phi_{\GJ,m}*\xi^\vee(g)\wt{\pi_{\chi_{s-\frac{1}{2}}}}(g)\wt{v}\ud g).
\end{align}
It is clear that
\[
\int_{\RG_n(F)}\Phi_{\GJ,m}*\xi^\vee(g)\wt{\pi_{\chi_{s-\frac{1}{2}}}}(g)\wt{v}\ud g
=
\wt{\pi_{\chi_{s-\frac{1}{2}}}}(\Phi_{\GJ,m}*\xi^\vee)\wt{v}
=
\wt{\pi_{\chi_{s-\frac{1}{2}}}}(\Phi_{\GJ,m})(\wt{\pi_{\chi_{s-\frac{1}{2}}}}(\xi^\vee)\wt{v}).
\]
Since $\xi^\vee$ belongs to $\CC_c^\infty(\RG_n(F))$, the vector
$\wt{\pi_{\chi_{s-\frac{1}{2}}}}(\xi^\vee)\wt{v}$ belongs to the space of
$\wt{\pi_{\chi_{s-\frac{1}{2}}}}$.
By definition, we have
\[
\wt{\pi_{\chi_{s-\frac{1}{2}}}}(\Phi_{\GJ,m})
=
f_{\GJ,m}(\wt{\pi_{\chi_{s-\frac{1}{2}}}})\cdot\RI_{\wt{\pi_{\chi_{s-\frac{1}{2}}}}}.
\]
Hence we can write the right-hand side of \eqref{38-3} as
\begin{align}\label{38-4}
(v,\int_{\RG_n(F)}\Phi_{\GJ,m}*\xi^\vee(g)\wt{\pi_{\chi_{s-\frac{1}{2}}}}(g)\wt{v}\ud g)
=
f_{\GJ,m}(\wt{\pi_{\chi_{s-\frac{1}{2}}}})\cdot(v,\wt{\pi_{\chi_{s-\frac{1}{2}}}}(\xi^\vee)\wt{v}).
\end{align}

Next we compute the twisted coefficient $(v,\wt{\pi_{\chi_{s-\frac{1}{2}}}}(\xi^\vee)\wt{v})$ on the
right-hand side of \eqref{38-4} as follows:
\begin{align*}
(v,\wt{\pi_{\chi_{s-\frac{1}{2}}}}(\xi^\vee)\wt{v})
&=
\int_{\RG_n(F)}\xi^\vee(h)(v,\wt{\pi_{\chi_{s-\frac{1}{2}}}}(h)\wt{v})\ud h
=
\int_{\RG_n(F)}\xi(h^{-1})(\pi_{\chi_{s-\frac{1}{2}}}(h^{-1})v,\wt{v})\ud h\\
&=
\int_{\RG_n(F)}\xi(h)(\pi_{\chi_{s-\frac{1}{2}}}(h)v,\wt{v})\ud h
=
\int_{\RG_n(F)}\xi(h)\vphi_{\pi[\chi]}(h)|\det h|^{s-\frac{1}{2}}\ud h\\
&=
\CZ(s,\xi,\vphi_{\pi[\chi]}).
\end{align*}
Hence we obtain the following functional equation
\begin{align}\label{38-5}
\CZ(1-s,\CF_{\GJ,m}(\xi),\vphi^\vee_{\pi[\chi]})
=
f_{\GJ,m}(\wt{\pi_{\chi_{s-\frac{1}{2}}}})\cdot
\CZ(s,\xi,\vphi_{\pi[\chi]})
\end{align}
for any $\xi\in\CC_c^\infty(\RG_n(F))$, $\vphi_\pi\in\CC(\pi)$ and $\chi\in\FX(F^\times)$.

From Theorem \ref{thm:GJ}, when $\Re(s)$ is sufficiently small, the zeta integral
$\CZ(1-s,\CF_{\GJ}(\xi),\vphi^\vee_{\pi[\chi]})$ converges absolutely for
any $\xi\in\CC_c^\infty(\RG_n(F))$, any $\vphi_\pi\in\CC(\pi)$ and any unitary character
$\chi\in\FX(F^\times)$. We write it as
\begin{align*}
\CZ(1-s,\CF_{\GJ}(\xi),\vphi^\vee_{\pi[\chi]})
&=
\sum_{m\in \BZ}
\CZ(1-s,\CF_{\GJ,m}(\xi),\vphi^\vee_{\pi[\chi]})
\\
&=
\CZ(s,\xi,\vphi_{\pi[\chi]})\cdot\sum_{m\in \BZ}f_{\GJ,m}(\wt{\pi_{\chi_{s-\frac{1}{2}}}}).
\end{align*}
By comparing with the right-hand side of the functional equation in Theorem \ref{thm:GJ}, we obtain that whenever $\Re(s)$ is sufficiently small,
\begin{align}\label{f-gamma}
f_{\GJ}(\wt{\pi_{\chi_{s-\frac{1}{2}}}})
=
\sum_{m\in \BZ}f_{\GJ,m}(\wt{\pi_{\chi_{s-\frac{1}{2}}}})
=
\gam(s,\pi\otimes\chi,\psi)=\gam(s,\pi_\chi,\psi)
\end{align}
By changing $s\to s+\frac{1}{2}$, we get
\[
f_{\GJ}(\wt{\pi_{\chi_s}})=\gam(s+\frac{1}{2},\pi_\chi,\psi)
=\gam(\frac{1}{2},\pi_{\chi_s},\psi).
\]
By taking the contragredient of $\pi_{\chi_s}$, we obtain that
\[
f_{\GJ}(\pi_{\chi_s})=\gam(\frac{1}{2},\wt{\pi_{\chi_s}},\psi)
=
\gam(\frac{1}{2}-s,\wt{\pi}\times\chi^{-1},\psi).
\]
This finishes the proof of Part (2).
\end{proof}

\section{$\pi$-Schwartz Functions and Fourier Operators}\label{sec-piSF-FO}

\subsection{Two spaces associated to $\pi$}\label{ssec-2S-pi}

For any $\pi\in \Pi_F(n)$, we are going to define two spaces associated to $\pi$:
$\CL_\pi(\FX(F^\times))$ and $\CS_\pi(F^\times)$.

The space $\CL_\pi=\CL_\pi(\FX(F^\times))$ consists of $\BC$-valued meromorphic functions $\Fz(\chi)$ on $\FX(F^\times)$ that
satisfy the following conditions.
\begin{enumerate}
\item $\Fz(\chi_{s,\ome})$ is a holomorphic multiple of the standard local $L$-function $L(s,\pi\times\omega)$ with $\chi_{s,\ome}(x)=|x|_F^s\ome(\ac(x))$.
\item If $F$ is non-Archimedean, $\Fz(\chi_{s,\ome})$ is nonzero for finitely many $\ome\in \Ome^\wedge$, and for each $\ome\in \Ome^\wedge$,
$\Fz(\chi_{s,\ome})\in L(s,\pi\times \omega)\cdot \BC[q^s,q^{-s}]$.
\item If $F$ is Archimedean, for any polynomial $P(\chi_{s,\ome})=P_{\omega}(s)$, if $P(\chi_{s,\ome}) L(s,\pi\times\omega)$ is holomorphic in any vertical strip $S_{a,b}$ as in \eqref{eq:vst},
with small neighborhoods at the possible poles of the $L$-function $L(s,\pi\times \ome)$ removed,
then for any $\Fz(\chi_{s,\ome})\in\CL_\pi$, the product $P(\chi_{s,\ome})\Fz(\chi_{s,\ome})$ is bounded in the same strip $S_{a,b}$, with small neighborhoods at the possible poles of
the $L$-function $L(s,\pi\times \ome)$ removed.
\end{enumerate}
From Part (3), we define a semi-norm to be
\[
\mu_{a,b:P}(\Fz):=\sup_{a\leq\Re(s)\leq b}|P(\chi_{s,\ome})\cdot \Fz(\chi_{s,\ome})|.
\]
Then the space $\CL_\pi$ is complete under the topology that is defined by the family of semi-norms $\mu_{a,b:P}$ for all possible choice of data $a,b; P$ as in Part (3) (\cite[Section 4]{J09}).

\begin{prp}\label{prp:Lpi-Z}
For any $\pi\in \Pi_F(n)$, the space $\CL_\pi$ is a subspace of $\CZ(\FX(F^\times))$ as defined in Definition \ref{dfn:CZ}.
\end{prp}

\begin{proof}
When $F$ is non-Archimedean, the statement is a consequence of Theorem \ref{thm:GJ}. We would like to focus on the case when $F$ is Archimedean.
In this case,  it suffices to estimate the boundedness condition. To do so, recall the following classical Stirling formula (see \cite[p.81]{J09} for instance)
\begin{equation}\label{eq:stirling}
\Gam(x+iy)\sim (2\pi)^{1/2}|y|^{x-\frac{1}{2}}e^{-\frac{\pi}{2}|y|}
\end{equation}
for $x$ fixed and $|y|\to \infty$.

Consider the Archimedean local $L$-functions $L(s,\pi\times \ome)=L(s,\pi\times \ac(\cdot)^p)$,
which can be explicitly expressed in terms of classical $\Gam$-functions with the local Langlands parameter of $\pi$. For instance, from
\cite[Section 8]{GJ72}, there exists a finite family of pairs $\{(l_i,u_i)\}_{i=1}^{t}$ with
$$
u_i\in \BC,\quad
l_i\in \begin{cases}
\BZ/2\BZ\simeq \{0,1\} , & F=\BR\\
\BZ, & F=\BC
\end{cases}
$$
such that in the fixed bounded vertical strip
$$
S_{a,b} = \{s\in \BC\mid a\leq \Re(s)\leq b\},
$$
up to a bounded factor in $S_{a,b}$, we have
$$
L(s,\pi\times \ac(\cdot)^p) \sim
\begin{cases}
\prod_{i=1}^t \Gam(\frac{s+u_i+l_i+p}{2}), & F=\BR\\
\prod_{i=1}^t \Gam(s+u_i+\frac{|l_i+p|}{2}), & F=\BC
\end{cases}
$$
with $p\in\BZ/2\BZ\simeq \{0,1\}$ if $F=\BR$; and $p\in\BZ$ if $F=\BC$.
Here $l_i+p$ is understood to be zero if both $l_i$ and $p$ are equal to $1$ when $F=\BR$.

It follows from the classical Stirling formula in \eqref{eq:stirling}, in particular the exponential decay of $\Gam(x+iy)$ along the imaginary axis, that for any polynomial
$P_\ome(s)=P(s)\in \BC[s]$ when $F=\BR$, and
$P_\ome(s)=P(s,p)\in \BC[s,p]$, the product $P(s,p) L(s,\pi\times \ac(\cdot)^p)$
is bounded in vertical strip $S_{a,b}$ with small neighborhoods at the possible poles removed. Hence from the definition of the space $\CL_\pi(\FX(F^\times))$, for any $\Fz(\chi_{s,\ome})\in \CL_\pi(\FX(F^\times))$,
the product $P(s,p)\Fz(\chi)$, with $\chi(x)=|x|_F^s\ac(x)^p$ is bounded in vertical strip $S_{a,b}$ with small neighborhoods at the possible poles of the $L$-function $L(s,\pi\times \ac(\cdot)^p)$ removed.
Therefore we obtain that the space $\CL_\pi=\CL_\pi(\FX(F^\times))$ is contained in the space $\CZ(\FX(F^\times))$, as defined in Definition \ref{dfn:CZ}.
\end{proof}

For any $\pi\in \Pi_F(n)$,  we define (Definition \ref{def:piSS}) the $\pi$-Schwartz space $\CS_\pi(F^\times)\subset \CC^\infty(F^\times)$ attached to $\pi$, by using the theory of local zeta integrals of Godement-Jacquet,
and prove that
\begin{align}\label{CSpi-CLpi}
\CS_\pi(F^\times)=\CM^{-1}(\CL_\pi)\subset\CC^\infty(F^\times)
\end{align}
by Theorem \ref{thm:GJ} and Theorem \ref{thm:MT}.

Consider the following determinant map:
\begin{align}\label{det}
\det={\det}_F\colon\RG_n(F)=\GL_n(F)\to F^\times.
\end{align}
It is clear that the kernel $\ker(\det)=\SL_n(F)$. For each $x\in F^\times$,
the fiber of the determinant map $\det$ is
\begin{align}\label{fiber}
\RG_n(F)_x:=
\{g\in \RG_n(F)\mid \det g=x\}.
\end{align}
It is clear that each fiber $\RG_n(F)_x$ is an $\SL_n(F)$-torsor. Hence one has the $\SL_n(F)$-invariant measure $\ud_x g$ that is induced from the (normalized) Haar measure $\ud_1 g$
on $\SL_n(F)$.

For $\xi\in\CS_{\std}(\RG_n(F))$ as defined in \eqref{GJ-SS}, $\varphi_\pi\in\CC(\pi)$, and $\chi\in\FX(F^\times)$,
the local zeta integral of Godement and Jacquet, as normalized in \eqref{GJ-Zeta}, can be written as
\begin{align}\label{eq:FI-1}
\CZ(s,\xi,\vphi_\pi,\chi)
=
\int_{F^\times}
\left(\int_{\RG_n(F)_x}
\xi(g)\vphi_\pi(g)\ud_x g \right)\chi(x)|x|_F^{s-\frac{1}{2}}\ud^\times x.
\end{align}
By Part (1) of Theorem \ref{thm:GJ}, the local zeta integral converges absolutely for $\Re(s)$ large. Hence the inner integral of \eqref{eq:FI-1}
\begin{align}\label{fibration}
\phi_{\xi,\vphi_\pi}(x) := \int_{\RG_n(F)_x}
\xi(g)\vphi_\pi(g)\ud_x g
=
|x|_F^{\frac{n}{2}}
\int_{\RG_n(F)_x}
f(g)\vphi_\pi(g)\ud_x g,
\end{align}
if $\xi(g)=|\det g|^{\frac{n}{2}}\cdot f(g)$ for some $f\in\CS(\RM_n(F))$,
is absolutely convergent for almost all $x\in F^\times$ and defines the fiber integration along the fibration in \eqref{det}.
Moreover, we have

\begin{prp}\label{prp:smooth}
For $\xi\in\CS_{\std}(\RG_n(F))$ and $\varphi_\pi\in\CC(\pi)$, the fiber integration in \eqref{fibration} that defines the function $\phi_{\xi,\vphi_\pi}(x)$ is absolutely convergent for all $x\in F^\times$, and
the function $\phi_{\xi,\vphi_\pi}(x)$ is smooth over $F^\times$.
\end{prp}

\begin{proof}
It is enough to show the proposition for the integral
\begin{align}\label{42-1}
\int_{\RG_n(F)_x}
f(g)\vphi_\pi(g)\ud_x g
\end{align}
with any $f\in\CS(\RM_n(F))$ and $\varphi_\pi\in\CC(\pi)$. In this case,
the product $f\cdot\vphi_\pi$ is smooth on $\RG_n(F)$.
Since the fiber $\RG_n(F)_x$ for any $x\in F^\times$ is closed in $\RG_n(F)$ and in $\RM_n(F)$, the restriction of $f$ to the fiber $\RG_n(F)_x$ is a Schwartz function on $\RG_n(F)_x$
(\cite{BerZ76} for $F$ non-Archimedean, and \cite[Theorem 4.6.1]{AG08} for $F$ Archimedean).

When $F$ is non-Archimedean, any $\vphi_\pi(g)\in\CC(\pi)$ is locally constant (smooth)
on $\RG_n(F)$ and hence is smooth on the fiber $\RG_n(F)_x$. This implies that the restriction of $f\cdot\vphi_\pi$ is locally constant and compactly supported on the fiber $\RG_n(F)_x$.
Hence the integral in \eqref{42-1} is absolutely convergent for all $x\in F^\times$, and defines a
smooth function in $x$ over $F^\times$.

When $F$ is Archimedean, since $\pi$ is a Casselman-Wallach representation of $\RG_n(F)$, the matrix coefficient $\vphi_\pi$ has at most of polynomial growth on $\RG_n(F)$ (\cite[Theorem~4.3.5]{Wal88}),
and so does on the fiber $\RG_n(F)_x$. This implies that the restriction of $f\cdot\vphi_\pi$ is a Schwartz function on the fiber $\RG_n(F)_x$ (\cite[Definition~4.1.1]{AG08}).
Thus the integral in \eqref{42-1} is absolutely convergent for all $x\in F^\times$.
Now we write the integral in \eqref{42-1} as
\begin{align}\label{42-2}
\int_{\RG_n(F)_x}
f(g)\vphi_\pi(g)\ud_x g
=
\int_{\SL_n(F)}
f(t_1(x)g)\vphi_\pi(t_1(x)g)\ud_1g
\end{align}
where $t_1(x)=\diag(x,1,\cdots,1)\in\RG_n(F)$ and $\ud_1g$ is the Haar measure of $\SL_n(F)$. It is clear that the absolute convergence of the integral in \eqref{42-2} is uniform when $x$ runs in any
compact subset of $F^\times$. Hence the integral in \eqref{42-1} defines a smooth function in $x$
over $F^\times$.

This finishes the proof.
\end{proof}

For $\xi\in\CS_{\std}(\RG_n(F))$ and $\varphi_\pi\in\CC(\pi)$, the function $\phi_{\xi,\vphi_\pi}(x)$ as given in Proposition \ref{prp:smooth} via the fiber integration \eqref{fibration} is called
a {\sl $\pi$-Schwartz function} on $F^\times$ associated to the pair $(\xi,\vphi_\pi)$. Here is the definition of $\pi$-Schwartz space.

\begin{dfn}[$\pi$-Schwartz Space]\label{def:piSS}
For any $\pi\in\Pi_F(n)$, the space of $\pi$-Schwartz functions is defined by
$$
\CS_\pi(F^\times)
=
\Span\{\phi_{\xi,\vphi_\pi}\in\CC^\infty(F^\times)\mid \xi\in \CS_{\std}(\RG_n(F)),\vphi_\pi\in \CC(\pi)\},
$$
where the $\pi$-Schwartz function $\phi_{\xi,\vphi_\pi}$ associated to a pair $(\xi,\vphi_\pi)$ is defined in \eqref{fibration}.
\end{dfn}

For any $\phi\in \CS_\pi(F^\times)$ and a quasi-character $\chi\in\FX(F^\times)$, define a $\GL_1$ zeta integral $\CZ(s,\phi,\chi)$ associated to the pair $(\phi,\chi)$ to be
\begin{equation}\label{eq:1-zeta}
\CZ(s,\phi,\chi) =
\int_{F^\times}
\phi(x)\chi(x)|x|_F^{s-\frac{1}{2}}\ud^\times x.
\end{equation}
When $\phi=\phi_{\xi,\vphi_\pi}$ for some $\xi\in\CS_{\std}(\RG_n(F))$ and $\varphi_\pi\in\CC(\pi)$, from Theorem \ref{thm:GJ}, we have the following identity of local zeta integrals:
\begin{align}\label{eq:zetas}
\CZ(s,\phi,\chi)=\CZ(s,\xi,\vphi_\pi,\chi),
\end{align}
which holds for $\Re(s)$ sufficiently large and then for all $s\in\BC$ by meromorphic continuation. Therefore, Theorem \ref{thm:GJ} can be restated for the $\GL_1$ zeta integrals
$\CZ(s,\phi,\chi)$.

\begin{thm}[$\GL_1$ Zeta Integrals]\label{thm:1-zeta}
The $\GL_1$ zeta integral $\CZ(s,\phi,\chi)$ as defined in \eqref{eq:1-zeta} for any $\phi\in \CS_\pi(F^\times)$ and any quasi-character $\chi\in\FX(F^\times)$ enjoy the following properties:
\begin{enumerate}
\item The zeta integral $\CZ(s,\phi,\chi)$ is absolutely convergent for $\Re(s)$ sufficiently large, and admits a meromorphic continuation to $s\in \BC$.
\item $\CZ(s,\phi,\chi)$ is a holomorphic multiple of the Langlands local $L$-function $L(s,\pi\times \chi)$ associated to $(\pi,\chi)$ and the standard embedding
$$
{\std}\colon\GL_n(\BC)\times\GL_1(\BC)\rightarrow\GL_n(\BC).
$$
Moreover, when $F$ is non-Archimedean, the fractional ideal generated by the local zeta integrals $\CZ(s,\phi,\chi)$ is of the form:
\begin{align*}
\{\CZ(s,\phi,\chi)\mid \phi\in \CS_\pi(F^\times) \}
=L(s,\pi\times \chi)\cdot \BC[q^s,q^{-s}];
\end{align*}
and when $F$ is Archimedean, the $\GL_1$ zeta integrals $\CZ(s,\phi,\chi)$, with unitary characters $\chi$, have the following property: Let $S_{a,b}$ be the vertical strip for any $a<b$, as defined in \eqref{eq:vst}.
If $P_\chi(s)$ is a polynomial in $s$ such that the product $P_\chi(s)L(s,\pi\times\chi)$ is bounded in the vertical strip $S_{a,b}$, with small neighborhoods at the possible poles of the $L$-function
$L(s,\pi\times \chi)$ removed, then the product $P_\chi(s) \CZ(s,\phi,\chi)$ must be bounded in the same vertical strip $S_{a,b}$, with small neighborhoods at the possible poles of the $L$-function
$L(s,\pi\times \chi)$ removed.
\item When $F$ is non-Archimedean, and $\pi$ is unramified, define
\[
\BL_\pi(x):=\phi_{\xi^\circ,\varphi_\pi^\circ}(x)
\]
where $\xi^\circ(g)=|\det g|^{\frac{n}{2}}{\bf 1}_{\RM_n(\Fo_F)}(g)$, with
${\bf 1}_{\RM_n(\Fo_F)}(g)$ being the characteristic function of $\RM_n(\Fo_F)$,
and $\varphi_\pi^\circ(g)$ is the zonal spherical function associated to $\pi$. Then the following identity
\[
\CZ(s,\BL_\pi,\chi)=L(s,\pi\times\chi)
\]
holds for any unramified characters $\chi$ and all $s\in\BC$ as meromorphic functions in $s$.
\end{enumerate}
\end{thm}

We are going to discuss the relation between
the $\pi$-Schwartz functions and the square-integrable functions in $L^2(F^\times,\ud^\times x)$.

\begin{prp}\label{prp:phiL2}
For any $\pi\in\Pi_F(n)$, there exists a real number $\alp_\pi$ such that for any $\phi\in\CS_\pi(F^\times)$ and for any $\kappa\geq\alp_\pi+\frac{n}{2}$, the function $|x|_F^\kappa\phi(x)$ belongs to the space
$L^2(F^\times,\ud^\times x)$ of square-integrable functions on $F^\times$.
\end{prp}

\begin{proof}
For any $\alp_0\in\BR$, we consider the following inner product of the function $|x|^{\frac{\alp_0}{2}}\phi(x)$ for any $\phi(x)\in\CS_\pi(F^\times)$. We write $\phi=\phi_{\xi,\vphi_\pi}$ for some
$\xi\in\CS_{\std}(\RG_n(F))$ and $\vphi_\pi\in\CC(\pi)$ and write
$\xi(g)=|\det g|_F^{\frac{n}{2}}f(g)$ with $f\in\CS(\RM_n(F))$. Then
\begin{align}\label{phiL2-1}
\int_{F^\times}
\phi(x)\wb{\phi(x)}
|x|_F^{\alp_0}
\ud^\times x
=&
\int_{F^\times}
|x|_F^{\alp_0+n}
\ud^\times x
\int_{\det g_1= \det g_2 = x}
f(g_1)\vphi_\pi(g_1)
\wb{f(g_2)}\wb{\vphi_\pi(g_2)}
\ud_x g_1 \ud_x g_2\nonumber
\\
=&
\int_{(\RG_n(F)\times \RG_n(F))^\circ}
f(g_1)\vphi_\pi(g_1)
\wb{f(g_2)}\wb{\vphi_\pi(g_2)}
|\det g_1|_F^{\alp_0+n}\ud(g_1,g_2)^\circ.
\end{align}
where $(\RG_n(F)\times \RG_n(F))^\circ:=\{(g_1,g_2)\in\RG_n(F)\times\RG_n(F)\mid \det g_1=\det g_2\}$ and $\ud(g_1,g_2)^\circ$ is a Haar measure on $(\RG_n(F)\times \RG_n(F))^\circ$, which
makes the above fiber integration factorization hold.

We consider the following natural embedding
\[
(\RG_n(F)\times \RG_n(F))^\circ\hookrightarrow(\RM_n(F)\times \RM_n(F))^\circ
\]
with an open dense image,
where
\[
(\RM_n(F)\times \RM_n(F))^\circ:=\{(X,Y)\in\RM_n(F)\times\RM_n(F)\mid \det X=\det Y\},
\]
which is the fiber product with respect to the determinant map: $X\mapsto \det X$, and is a closed
subvariety of the affine space $\RM_n(F)\times\RM_n(F)$. The natural group
action of $\RG_n\times\RG_n$ on $\RM_n\times\RM_n$ via
\[
(g,h)((X,Y))=(gX,hY)
\]
for $(g,h)\in\RG_n\times\RG_n$ and $(X,Y)\in\RM_n\times\RM_n$ yields the action of
$(\RG_n(F)\times \RG_n(F))^\circ$ on $(\RM_n(F)\times \RM_n(F))^\circ$ by restriction. Take
$\ud^+X\wedge\ud^+Y$ to be an additive Haar measure on $\RM_n(F)\times\RM_n(F)$ with
$|\det gh|_F^n$ the modulus function of the action of $\RG_n\times\RG_n$ on
$\RM_n\times\RM_n$. Take the measure $\ud^+ (X,Y)^\circ$ on $(\RM_n(F)\times \RM_n(F))^\circ$,
which is the pullback of the measure $\ud^+X\wedge\ud^+Y$ through the fiber product embedding.
Then the modulus function of the action of $(\RG_n(F)\times \RG_n(F))^\circ$ on $(\RM_n(F)\times \RM_n(F))^\circ$ is
\[
|\det gh|_F^n=|\det g|_F^{2n}=|\det h|_F^{2n}
\]
for any $(g,h)\in(\RG_n(F)\times \RG_n(F))^\circ$. It is easy to check that
\[
\frac{\ud^+ (g,h)^\circ}{|\det gh|_F^n}
\]
is a Haar measure on $(\RG_n(F)\times \RG_n(F))^\circ$. Hence there is a constant $c>0$, such
that
\[
\ud(g,h)^\circ=c\cdot\frac{\ud^+ (g,h)^\circ}{|\det gh|_F^n}.
\]
The integral in \eqref{phiL2-1} can be written as
\begin{align}\label{phiL2-2}
\int_{(\RM_n(F)\times \RM_n(F))^\circ}
f(X)\vphi_\pi(X)
\wb{f(Y)}\wb{\vphi_\pi(Y)}
|\det X|_F^{\alp_0-n}\ud^+(X,Y)^\circ
\end{align}
Here we assume that $\alp_0\geq n$ and both $\vphi_\pi(g_1)$ and $\vphi_\pi(g_2)$ are viewed as measurable functions on $\RM_n(F)$ that extends by zero to the boundary $\RM_n(F)\bs \GL_n(F)$.

Since the $F$-analytical manifold $(\RM_n(F)\times \RM_n(F))^\circ$
is closed in $\RM_n(F)\times \RM_n(F)$, the restriction of the Schwartz function $f(g_1)\times \wb{f(g_2)}$ to $(\RM_n(F)\times \RM_n(F))^\circ$ is still a Schwartz function, which is smooth and compactly supported when $F$ is non-Archimedean, and is in the sense of \cite{AG08} when $F$ is Archimedean. By Theorem \ref{thm:GJ}, the zeta integral of Godement-Jacquet $\CZ(s,f,\vphi_\pi,\chi)$ converges absolutely for
$\Re(s)$ sufficiently large. It follows that for any $\pi\in\Pi_F(n)$, there exists a real number $\alp_\pi$ such that for any $\vphi_\pi\in\CC(\pi)$ and any $\Re(s)\geq\alp_\pi$, the product
$|\det(g)|_F^{s}\vphi_\pi(g)$ is bounded when $\det g$ tends to zero.

We write  the $F$-analytical closed submanifold $(\RM_n(F)\times \RM_n(F))^\circ$ as a union of two closed submanifolds:
\[
(\RM_n(F)\times \RM_n(F))^\circ=(\RM_n(F)\times \RM_n(F))^\circ_{\geq\veps}\cup(\RM_n(F)\times \RM_n(F))^\circ_{\leq\veps},
\]
where
\[
(\RM_n(F)\times \RM_n(F))^\circ_{\geq\veps} = \{(g_1,g_2)\in \RM_n(F)^{\Del\det}\mid \
|\det g_1|_F\geq \veps\}
\]
and
\[
(\RM_n(F)\times \RM_n(F))^\circ_{\leq\veps} = \{(g_1,g_2)\in \RM_n(F)^{\Del\det}\mid \
|\det g_1|_F\leq \veps\}.
\]
For any $\pi\in\Pi_F(n)$, the restriction of the product
$\vphi_\pi(g_1) \wb{\vphi_\pi(g_2)}\cdot|\det g_1|_F^{s-n}$ to the closed submanifold
$(\RM_n(F)\times \RM_n(F))^\circ_{\geq\veps}$
is of moderate growth and its restriction to the closed submanifold $(\RM_n(F)\times \RM_n(F))^\circ_{\leq\veps}$ is bounded when $\Re(s)\geq2\alp_\pi+n$.
It is also clear the Schwartz function $f(g_1)\times \wb{f(g_2)}$ on $(\RM_n(F)\times\RM_n(F))^\circ$
is still a Schwartz function when restricted to either the closed submanifold
$(\RM_n(F)\times \RM_n(F))^\circ_{\geq\veps}$ or  the closed submanifold
$(\RM_n(F)\times \RM_n(F))^\circ_{\leq\veps}$.
Hence for any $\alp_0\in\BR$ with $\alp_0\geq2\alp_\pi+n$, the following integral
$$
\int_{(\RM_n(F)\times \RM_n(F))^\circ}
f(X)\vphi_\pi(X)
\wb{f(Y)}\wb{\vphi_\pi(Y)}
|\det X|_F^{\alp_0-n}\ud^+(X,Y)^\circ
$$
converges absolutely, and so does the following integral
\[
\int_{F^\times}
\phi(x)\wb{\phi(x)}
|x|_F^{\alp_0}
\ud^\times x.
\]
It follows that the product $\phi(x) |x|_F^\kappa$ is square integrable on $F^\times$ for
$\kappa=\frac{\alp_0}{2}\geq\alp_\pi+\frac{n}{2}$.
\end{proof}

\begin{cor}\label{L2}
When $\pi\in\Pi_F(n)$ is unitarizable, for any $\phi\in\CS_\pi(F^\times)$, the function
$|x|_F^{\frac{n}{2}}\cdot\phi(x)$ belongs to the space $L^2(F^\times,\ud^\times x)$ of square-integrable functions on $F^\times$.
\end{cor}

\begin{proof}
If $\pi\in\Pi_F(n)$ is unitarizable, then the matrix coefficient $\varphi_\pi(g)$ is bounded above over
$G_n(F)$. For $\phi\in\CS_\pi(F^\times)$, we write $\phi=\phi_{\xi,\vphi_\pi}$ with
$\xi\in\CS_{\std}(\RG_n(F))$ and $\vphi_\pi\in\CC(\pi)$, and write
$\xi(g)=|\det g|_F^{\frac{n}{2}}\cdot f(g)$ with $f\in\CS(\RM_n(F))$. We compute the inner product
of $|x|_F^{\frac{n}{2}}\cdot\phi(x)$ as follows:
\begin{align}\label{L2-1}
\int_{F^\times}
\phi(x)\wb{\phi(x)}|x|_F^n
\ud^\times x
&\leq
\int_{F^\times}|x|_F^{2n}
\int_{\RG_n(F)_x}
|f(g_1)\vphi_\pi(g_1)|\ud_x g_1
\int_{\RG_n(F)_x}
|f(g_2)\vphi_\pi(g_2)|
\ud_x g_2
\ud^\times x\nonumber\\
&\leq
c(\vphi_\pi)\cdot
\int_{F^\times}|x|_F^{2n}
\int_{\RG_n(F)_x}
|f(g_1)|\ud_x g_1
\int_{\RG_n(F)_x}
|f(g_2)|
\ud_x g_2
\ud^\times x
\end{align}
for some positive constant $c(\vphi_\pi)$ depending on $\vphi_\pi$.
Following the proof of Proposition \ref{prp:phiL2}, we obtain that
\begin{align}\label{L2-2}
\int_{F^\times}
\phi(x)\wb{\phi(x)}|x|_F^n
\ud^\times x
\leq c\cdot c(\vphi_\pi)
\int_{(\RM_n(F)\times\RM_n(F))^\circ}
|f(X)|\cdot|\wb{f(Y)}|\ud(X,Y)^\circ.
\end{align}
The integral on the right-hand side of \eqref{L2-2} comes from the integral in \eqref{phiL2-2} with
$\alp_0=n$. As explained in the proof of Proposition \ref{prp:phiL2}, the product
$f(X)\times\wb{f(Y)}$ is a Schwartz function on $(\RM_n(F)\times\RM_n(F))^\circ$. Hence
the integral on the right-hand side of \eqref{L2-2} converges.
\end{proof}

By using Proposition \ref{prp:phiL2} and Theorem \ref{thm:1-zeta}, together with Theorem \ref{thm:MT}, we are able to understand the $\pi$-Schwartz space $\CS_\pi(F^\times)$ by means of the $L$-functions $L(s,\pi\times\chi)$ for
any $\pi\in\Pi_F(n)$.

\begin{prp}\label{prp:Spi-F}
For any $\pi\in\Pi_F(n)$, the $\pi$-Schwartz space $\CS_\pi(F^\times)$ is contained in the space $\CF(F^\times)$ as defined in Definition \ref{dfn:CF}
\end{prp}

\begin{proof}
Note first that the $\GL_1$ zeta integral attached to $\phi\in\CS_\pi(F^\times)$ is the same of the Mellin transform of $\phi$ up to a shift in $s$ by the unramified part of $\chi$.
By Theorem \ref{thm:1-zeta} and Proposition \ref{prp:Lpi-Z}, the image of $\CS_\pi(F^\times)$ under Mellin transform is contained the space $\CL_\pi(\FX(F^\times))$ and hence in the space $\CZ(\FX(F^\times))$.
By Theorem \ref{thm:MT}, for any $\phi\in \CS_\pi(F^\times)$, there exists $\phi_0\in \CF(F^\times)$, such that
\begin{align}\label{CM=0}
\CM(\phi-\phi_0)(\chi) = 0
\end{align}
holds identically for any quasi-character $\chi\in\FX(F^\times)$.
It remains to show that $\phi-\phi_0 = 0$ holds identically.
By smoothness of $\phi$ and $\phi_0$, it suffices to show that after unramified twist, both $\phi$ and $\phi_0$ are square-integrable on $F^\times$.

For $\phi_0\in \CF(F^\times)$, there exists $s_0\in \BR$ such that for any $\Re(s)>s_0$,
$$
\lim_{x\to 0}
\phi_0(x)|x|_F^{s+1} = 0
$$
and the limit is preserved by differentiation on both sides when $F$ is Archimedean. It follows that
$\phi_0(x)|x|_F^{s}$ is indeed square integrable on $F^\times$ for $\Re(s)>s_0$,
via the asymptotic formula appearing in the definition of $\CF(F^\times)$.

For any $\phi\in \CS_\pi(F^\times)$, by Proposition \ref{prp:phiL2}, there exists $\alp_\pi\in\BR_{>0}$ such that the function $|x|_F^s\phi(x)$ is square-integrable if $\Re(s)\geq\alp_\pi+\frac{n}{2}$.
By taking $\kappa>\max\{s_0,\alp_\pi+\frac{n}{2}\}$, we obtain that
both $\phi_0(x)|x|_F^\kappa$ and $\phi(x)|x|_F^\kappa$ are square-integrable over $F^\times$. From \eqref{CM=0}, we obtain that the Mellin transform
\[
\CM(\phi(x)|x|_F^\kappa-\phi_0(x)|x|_F^\kappa)(\chi)=0
\]
for all quasi-characters $\chi\in\FX(F^\times)$, in particular, for all unitary characters $\chi$ of $F^\times$. Therefore, by the Mellin inversion formula (Theorem \ref{thm:MT}),
we obtain that
\[
\phi(x)|x|_F^\kappa-\phi_0(x)|x|_F^\kappa=0
\]
as functions in the space $L^2(F^\times,\ud^\times x)$. Since both $\phi(x)$ and $\phi_0(x)$ are smooth, we must have that $\phi(x)=\phi_0(x)\in\CF(F^\times)$.
\end{proof}

Finally we are ready to characterize the Mellin inversion $\CM^{-1}(\CL_\pi)$ in terms of the $\pi$-Schwartz space $\CS_\pi(F^\times)$ as in \eqref{CSpi-CLpi}.
\begin{cor}\label{cor:SpiFpi}
For any $\pi\in\Pi_F(n)$, the Mellin inversion $\CM^{-1}(\CL_\pi)$ coincides with the space $\CS_\pi(F^\times)$:
\[
\CS_\pi(F^\times)=\CM^{-1}(\CL_\pi)\subset\CC^\infty(F^\times).
\]
In particular, the space $\CC_c^\infty(F^\times)$ of smooth compactly supported functions on $F^\times$ is contained in the $\pi$-Schwartz space $\CS_\pi(F^\times)$.
\end{cor}

\begin{proof}
By Proposition \ref{prp:Spi-F}, we have that the space $\CS_\pi(F^\times)$ is contained in the space $\CF(F^\times)$. By Theorem \ref{thm:1-zeta}, the Mellin transform ($\GL_1$ zeta integral) of the space $\CS_\pi(F^\times)$ is equal to the space $\CL_\pi=\CL_\pi(\FX(F^\times))$. Hence we obtain that $\CS_\pi(F^\times)=\CM^{-1}(\CL_\pi)$, because the Mellin transform is a bijective
correspondence between the space $\CF(F^\times)$ and the space $\CZ(\FX(F^\times))$ (Theorem \ref{thm:MT}). Finally, since the space $\CL_\pi$ contains the space of holomorphic functions on $\FX(F^\times)$ that are of Paley-Wiener type along the vertical strips,
it is clear from Theorem \ref{thm:MT} again that $\CC_c^\infty(F^\times)$ is contained in the $\pi$-Schwartz space $\CS_\pi(F^\times)$.
\end{proof}

The relevant functional equation for $\GL_1$ zeta integrals will be discussed in the next section.

\subsection{Fourier operators}\label{ssec:FO-GL1}
We define a Fourier operator $\CF_{\pi,\psi}$ from the $\pi$-Schwartz space $\CS_\pi(F^\times)$ to the $\wt{\pi}$-Schwartz space $\CS_{\wt{\pi}}(F^\times)$ for any $\pi\in \Pi_F(n)$ with smooth contragredient
$\wt{\pi}$ and prove the functional equation for $\GL_1$ zeta integrals $\CZ(s,\phi,\chi)$.

For $\phi\in\CS_\pi(F^\times)$, the Fourier operator $\CF_{\pi,\psi}(\phi)$ is defined by the following diagram:
\begin{align}\label{diag:F}
\xymatrix{
\CS_{\std}(\RG_n(F))\otimes \CC(\pi)\ar[d]\ar[rrr]^{(\CF_{\GJ},(\cdot)^{\vee})}&&& \CS_{\std}(\RG_n(F))\otimes \CC(\wt{\pi})\ar[d]\\
\CS_\pi(F^\times) \ar[rrr]^{\CF_{\pi,\psi}} &&& \CS_{\wt{\pi}}(F^\times)
}
\end{align}
where $\psi$ is a non-trivial additive character of $F$.
More precisely, for $\phi=\phi_{\xi,\vphi_\pi}\in\CS_\pi(F^\times)$ with a $\xi\in\CS_{\std}(\RG_n(F))$ and
a $\vphi_\pi\in\CC(\pi)$, we define
\begin{align}\label{eq:1-FO}
\CF_{\pi,\psi}(\phi)=\CF_{\pi,\psi}(\phi_{\xi,\vphi_\pi}):=\phi_{\CF_{\GJ}(\xi),\vphi_\pi^\vee},
\end{align}
where $\vphi_\pi^\vee(g)=\vphi_\pi(g^{-1})\in\CC(\wt{\pi})$. Hence we obtain that
\begin{align}\label{eq:1-FO-1}
\CF_{\pi,\psi}(\phi)=\CF_{\pi,\psi}(\phi_{\xi,\vphi_\pi})\in\CS_{\wt{\pi}}(F^\times).
\end{align}
It remians to check that the definition of the Fourier operator in \eqref{eq:1-FO} is independent of
the choice of $\xi\in\CS_{\std}(\RG_n(F))$ and $\vphi_\pi\in\CC(\pi)$.

\begin{prp}\label{prp:1-CF}
The Fourier operator $\CF_{\pi,\psi}$ as in \eqref{eq:1-FO} is independent of the choice of
$\xi\in\CS_{\std}(\RG_n(F))$ and $\vphi_\pi\in\CC(\pi)$.
\end{prp}

\begin{proof}
Assume that $\phi_{\xi_1,\vphi_{\pi,1}} = \phi_{\xi_2,\vphi_{\pi,2}}$ for some
$\xi_1, \xi_2\in \CS_{\std}(\RG_n(F))$ and $\vphi_{\pi,1}, \vphi_{\pi,2}\in\CC(\pi)$. We want to show that
$\CF_{\pi,\psi}(\phi_{\xi_1,\vphi_{\pi,1}})=\CF_{\pi,\psi}(\phi_{\xi_2,\vphi_{\pi,2}})$.

From \eqref{eq:zetas}, we must have that
\[
\CZ(s,\xi_1,\vphi_{\pi,1},\chi)=\CZ(s,\xi_2,\vphi_{\pi,2},\chi)
\]
for all quasi-character $\chi\in\FX(F^\times)$ and all $s\in\BC$. Of course, the identity holds for $\Re(s)$ large and then for all $s\in\BC$ by meromorphic continuation.
By the functional equation in Proposition \ref{prp:GJ-FE}, we obtain the following identity
\[
\CZ(1-s,\CF_{\GJ}(\xi_1),\vphi^\vee_{\pi,1},\chi^{-1})
=
\CZ(1-s,\CF_{\GJ}(\xi_2),\vphi^\vee_{\pi,2},\chi^{-1})
\]
for all $\chi\in\FX(F^\times)$ with $\Re(s)$ sufficiently small first and then all $s\in\BC$ by meromorphic continuation. It follows by the identity in \eqref{eq:zetas} again
that for all $\chi\in\FX(F^\times)$ and for $\Re(s)+\Re(\chi)$ sufficiently large, the following integral
\[
\int_{F^\times}\left(\phi_{\CF_{\GJ}(\xi_1),\vphi_{\pi,1}^\vee}(x)-\phi_{\CF_{\GJ}(\xi_2),\vphi_{\pi,2}^\vee}(x)\right)
\chi(x)|x|_F^{s-\frac{1}{2}}\ud^\times x=0
\]
holds. By Proposition \ref{prp:Spi-F}, we have that
$\phi_{\CF_{\GJ}(\xi_1),\vphi_{\pi,1}^\vee}(x)-\phi_{\CF_{\GJ}(\xi_2),\vphi_{\pi,2}^\vee}(x)$
belongs to $\CF(F^\times)$. Finally, by Theorem \ref{thm:MT}, we must have
that
\[
\phi_{\CF_{\GJ}(\xi_1),\vphi_{\pi,1}^\vee}(x)-\phi_{\CF_{\GJ}(\xi_2),\vphi_{\pi,2}^\vee}(x)=0
\]
as functions on $F^\times$. Therefore, we prove that
\[
\phi_{\CF_{\GJ}(\xi_1),\vphi_{\pi,1}^\vee}(x)=\phi_{\CF_{\GJ}(\xi_2),\vphi_{\pi,2}^\vee}(x)
\]
as functions on $F^\times$, and
$\CF_{\pi,\psi}(\phi_{\xi_1,\vphi_{\pi,1}})=\CF_{\pi,\psi}(\phi_{\xi_2,\vphi_{\pi,2}})$.
\end{proof}

The following theorem on the local functional equation for the $\GL_1$ zeta integrals $\CZ(s,\phi,\chi)$ is a direct consequence of Theorem \ref{thm:GJ} and Proposition \ref{prp:1-CF}.

\begin{thm}[$\GL_1$ Functional Equation]\label{thm:1-FE}
For any $\pi\in\Pi_F(n)$ and its contragredient $\wt{\pi}\in\Pi_F(n)$, there exists a Fourier operator $\CF_{\pi,\psi}$, which takes $\phi\in\CS_\pi(F^\times)$ to $\CF_{\pi,\psi}(\phi)\in\CS_{\wt{\pi}}(F^\times)$,
such that the following functional equation holds after meromorphic continuation,
\[
\CZ(1-s,\CF_{\pi,\psi}(\phi),\chi^{-1}) =
\gam(s,\pi\times \chi,\psi)
\cdot \CZ(s,\phi,\chi),
\]
for any $\phi\in \CS_\pi(F^\times)$. Moreover, the following identities
$$
\CF_{\wt{\pi},\psi^{-1}}\circ \CF_{\pi,\psi} = \Id,\qquad \CF_{{\pi},\psi}\circ \CF_{\wt{\pi},\psi^{-1}}=\Id
$$
hold. When $F$ is non-Archimedean, and $\pi$ is unramified, the Fourier operator $\CF_{\pi,\psi}$
takes the basic function $\BL_\pi\in\CS_\pi(F^\times)$ to the basic function $\BL_{\wt{\pi}}\in\CS_{\wt{\pi}}(F^\times)$:
\[
\CF_{\pi,\psi}(\BL_\pi)=\BL_{\wt{\pi}},
\]
where the basic function $\BL_\pi$ is defined in Theorem \ref{thm:1-zeta}.
\end{thm}

\section{$\pi$-Poisson Summation Formula on $\GL_1$}\label{sec-PSF}

Let $k$ be a number field and $\BA$ be the ring of adeles of $k$. Denote by $|k|$ the set of all local places of $k$ and by $|k|_\infty$ the set of all Archimedean local places of $k$. We may write
\[
|k|=|k|_\infty\cup|k|_f
\]
where $|k|_f$ is the set of non-Archimedean local places of $k$. For each $\nu\in|k|$, we write $F=k_\nu$.
Let $\Pi_\BA(n)$ be the set of equivalence classes of irreducible admissible representations
of $\RG_n(\BA)$. If we write $\pi=\otimes_{\nu\in|k|}\pi_\nu$, then we assume that
$\pi_\nu\in\Pi_{k_\nu}(n)$, where at almost all finite local places $\nu$, the local representations $\pi_\nu$ are unramified. Moreover, when $\nu$ is a finite local place, $\pi_\nu$ is an irreducible admissible representation of $\RG_n(k_\nu)$, and when $\nu$ is an infinite local place, we assume that $\pi_\nu$ is of Casselman-Wallach type as representation of $\RG_n(k_\nu)$.
Let $\CA(\RG_n)\subset\Pi_\BA(n)$ be the subset consisting of equivalence classes of irreducible admissible automorphic representations
of $\GL_n(\BA)$, and $\CA_\cusp(\RG_n)$ be the subset of cuspidal members of $\CA(\RG_n)$.

\subsection{$\pi$-Schwartz space and Fourier operator}\label{ssec-SS-FO-BA}
Take any $\pi=\otimes_{\nu\in|k|}\pi_\nu\in\Pi_\BA(n)$.
For each local palce $\nu\in|k|$, the $\pi_\nu$-Schwartz space $\CS_{\pi_\nu}(k_\nu^\times)$ is defined in Definition \ref{def:piSS}.
Recall from Theorems \ref{thm:1-zeta} and \ref{thm:1-FE} the basic function
$\BL_{\pi_\nu}\in\CS_{\pi_\nu}(k_\nu^\times)$ of $\pi_\nu$
when the local component $\pi_\nu$ of $\pi$ is unramified. It is clear from the definition that
$\BL_{\pi_\nu}(1)=1$ (We have to normalize various local measures in the computations. Actually it follows from the fact that the Laurent expansion of the unramified local $L$-factor has constant term $1$).

For the given $\pi=\otimes_\nu\pi_\nu\in\Pi_\BA(n)$, we define the $\pi$-Schwartz space on
$\BA^\times$ to be
\begin{align}\label{piSS-BA}
\CS_\pi(\BA^\times):=\otimes_{\nu\in|k|}\CS_{\pi_\nu}(k_\nu^\times),
\end{align}
which is the restricted tensor product of the local $\pi_\nu$-Schwartz space $\CS_{\pi_\nu}(k_\nu^\times)$ with respect to the family of the basic functions $\BL_{\pi_\nu}$ for all the local places $\nu$ at which
$\pi_\nu$ are unramified. The factorizable vectors $\phi=\otimes_\nu\phi_\nu$ in $\CS_\pi(\BA^\times)$ can be written as
\begin{align}\label{SF-factorize}
\phi(x)=\prod_{\nu\in|k|}\phi_\nu(x_\nu).
\end{align}
Here for almost all finite local places $\nu$, $\phi_\nu(x_\nu)=\BL_{\pi_\nu}(x_\nu)$. According to
our normalization, we have $\BL_{\pi_\nu}(x_\nu)=1$ when $x_\nu\in\Fo_\nu^\times$, the unit group of
the ring $\Fo_\nu$ of integers at $\nu$. Hence for any given $x\in\BA^\times$, the product in \eqref{SF-factorize} is a finite product over Archimedean local places and finitely many non-Archimedean local places containing all ramified local places.

For any factorizable vectors $\phi=\otimes_\nu\phi_\nu$ in $\CS_\pi(\BA^\times)$, we define the
$\pi$-Fourier operator:
\begin{align}\label{FO-BA}
\CF_{\pi,\psi}(\phi):=\otimes_{\nu\in|k|}\CF_{\pi_\nu,\psi_\nu}(\phi_\nu),
\end{align}
where for each $\nu\in|k|$, $\CF_{\pi_\nu,\psi_\nu}$ is the local Fourier operator as defined in
\eqref{diag:F} and \eqref{eq:1-FO}, which takes the $\pi_\nu$-Schwartz space
$\CS_{\pi_\nu}(k_\nu^\times)$ to the $\wt{\pi}_\nu$-Schwartz space $\CS_{\wt{\pi}_\nu}(k_\nu^\times)$,
and
\[
\CF_{\pi_\nu,\psi}(\BL_{\pi_\nu})=\BL_{\wt{\pi}_\nu}
\]
when the data are unramified at $\nu$. Hence the Fourier operator $\CF_{\pi,\psi}$ as defined in
\eqref{FO-BA} maps $\pi$-Schwartz space $\CS_\pi(\BA^\times)$ to $\wt{\pi}$-Schwartz space
$\CS_{\wt{\pi}}(\BA^\times)$.

\subsection{Global zeta integral}\label{ssec-GZIs}
For any $\pi=\otimes_\nu\pi_\nu\in\CA(\RG_n)$, we define the ($\GL_1$) global zeta integrals to be
\begin{align}\label{1-gzi}
\CZ(s,\phi,\chi):=
\int_{\BA^\times}
\phi(x)\chi(x)|x|_\BA^{s-\frac{1}{2}}\ud^\times x
\end{align}
for any $\phi\in\CS_\pi(\BA^\times)$ and characters $\chi$ of $k^\times\bs\BA^\times$. When
$\phi=\otimes_\nu\phi_\nu$, we have
\[
\CZ(s,\phi,\chi)=\Pi_{\nu\in|k|}\CZ(s,\phi_\nu,\chi_\nu).
\]
Let $S$ be a finite subset of $|k|$, which contains all Archimedean local places and all the finite local places
$\nu$ at which $\pi_\nu$ or $\chi_\nu$ is ramified. Then we write
\[
\CZ(s,\phi,\chi)=L^S(s,\pi\times\chi)\cdot\Pi_{\nu\in S}\CZ(s,\phi_\nu,\chi_\nu),
\]
according to Theorem \ref{thm:1-zeta}. If $\pi$ is unitarizable, the partial $L$-function
$L^S(s,\pi\times\chi)$ converges absolutely for $\Re(s)$ large. By Theorem \ref{thm:1-zeta} again, the
finite Euler product $\Pi_{\nu\in S}\CZ(s,\phi_\nu,\chi_\nu)$ converges absolutely for $\Re(s)$ large.

\begin{prp}\label{prp:gzi}
Let $\pi\in\CA(\RG_n)$ be unitarizable. Then for any $\phi\in\CS_\pi(\BA^\times)$ and any character
$\chi$ of $k^\times\bs\BA^\times$, the zeta integral $\CZ(s,\phi,\chi)$ as defined in \eqref{1-gzi}
converges absolutely for $\Re(s)$ sufficiently large.
\end{prp}

When $\pi$ belongs to $\CA_\cusp(\RG_n)$, which is unitary, the zeta integral $\CZ(s,\phi,\chi)$ can be
identified with the Godement-Jacquet global zeta integral. For any $f=\otimes_\nu f_\nu\in\CS(\RM_n(\BA))$
and any $\varphi_\pi\in\CC(\pi)$, the Godement-Jacquet global zeta integral is defined to be
\begin{align}\label{GJ-gzi}
\CZ(s,f,\vphi_\pi,\chi) :=
\int_{\GL_n(\BA)}
f(g)\vphi_\pi(g)\chi(\det g)|\det g|_F^{s+\frac{n-1}{2}}\ud g,
\end{align}

\begin{thm}[Theorem 13.8, \cite{GJ72}]\label{thm:GJ-gzi}
For $\pi\in\CA_\cusp(\RG_n)$ and unitary automorphic character $\chi$ of $k^\times\bs\BA^\times$,
the global zeta integral $\CZ(s,f,\vphi_\pi,\chi)$ converges absolutely for
$\Re(s)>\frac{n+1}{2}$, admits analytic continuation to an entire function in $s\in\BC$, and
satisfies the global functional equation:
\begin{align}\label{GJ-gfe}
\CZ(s,f,\varphi_\pi,\chi)=\CZ(1-s,\CF_\psi(f),\varphi_\pi^\vee,\chi^{-1}),
\end{align}
where $\CF_\psi$ is the global Fourier transform from $\CS(\RM_n(\BA))$ to $\CS(\RM_n(\BA))$ associated to
the additive character $\psi$ of $k\bs\BA$.
\end{thm}

For $\Re(s)>\frac{n+1}{2}$, we write
\begin{align}\label{GJ-gzi-1}
\CZ(s,f,\vphi_\pi,\chi) =
\int_{\BA^\times}
\left(|x|_\BA^{\frac{n}{2}}
\int_{\RG_n(\BA)_x}
f(g)\vphi_\pi(g)\ud_x g\right)\chi(x)|x|_\BA^{s-\frac{1}{2}}\ud^\times x,
\end{align}
where $\RG_n(\BA)_x:=\{g\in\RG_n(\BA)\mid \det g=x\}$ is an $\SL_n(\BA)$-torsor, and the measure
$\ud_x g$ is $\SL_n(\BA)$-invariant.
As in the local situations, we define, for any $x\in\BA^\times$,
\begin{align}\label{fiber}
\phi_{\xi,\varphi_\pi}(x):=
\int_{\RG_n(\BA)_x}
\xi(g)\vphi_\pi(g)\ud_x g
=
|x|_\BA^{\frac{n}{2}}
\int_{\RG_n(\BA)_x}
f(g)\vphi_\pi(g)\ud_x g
\end{align}
where $\xi(g):=|\det g|_{\BA}^{\frac{n}{2}}\cdot f(g)$ belongs to the space
\begin{align}\label{CSstd}
\CS_{\std}(\RG_n(\BA))=\{\xi\in\CC^\infty(\RG_n(\BA))\mid \xi(g)\cdot|\det g|_\BA^{-\frac{n}{2}}\in
\CS(\RM_n(\BA))\}.
\end{align}
It is clear that
\begin{align}\label{CSstd-otimes}
\CS_{\std}(\RG_n(\BA))=\otimes_{\nu\in|k|}\CS_{\std}(\RG_n(k_\nu)).
\end{align}
Write $\RG_n(\BA)$ as a direct product decomposition:
\begin{align}\label{dpd}
\RG_n(\BA)=A_n(\BR)^+\cdot\RG_n(\BA)^1,
\end{align}
where $\RG_n(\BA)^1:=\{g\in\RG_n(\BA)\mid\ |\det g|_\BA=1\}$ and $A_n(\BR)^+$ is the identity
connected component of the center $Z_{\RG_n}(\BR)$ of $\RG_n(\BR)$.
As in \cite[Section 13]{GJ72}, any matrix coefficient $\varphi_\pi$ of the cuspidal
$\pi\in\CA_\cusp(\RG_n)$ can be written as
\begin{align}\label{gcoeff}
\varphi_\pi(g)=\int_{A_n(\BR)^+\RG_n(k)\bs\RG_n(\BA)}\alpha_\pi(hg)\alpha_{\wt{\pi}}(h)\ud h
=
\int_{\RG_n(k)\bs\RG_n(\BA)^1}\alpha_\pi(hg)\alpha_{\wt{\pi}}(h)\ud h
\end{align}
for some $\alpha_\pi\in V_\pi$ and $\alpha_{\wt{\pi}}\in V_{\wt{\pi}}$, where $V_\pi$ is the cuspidal
automorphic realization of $\pi$ in $L^2(\RG_n(k)\bs\RG_n(\BA),\omega)$ with central character
$\omega_\pi=\omega$.
In this case, we have $\omega_{\wt{\pi}}=\omega^{-1}$. In the integral in \eqref{fiber}, the coefficient
$\varphi_\pi(g)$ is bounded over $\RG_n(\BA)$. Since $f\in\CS(\RM_n(\BA))$ and $\RG_n(\BA)_x$ is a closed submanifold
in $\RM_n(\BA)$, the restriction to $\RG_n(\BA)_x$ of the Schwartz function $f$ is still a Schwartz
function on $\RG_n(\BA)_x$. Hence the integral in \eqref{fiber} converges absolutely for any
$x\in\BA^\times$, and the convergence is uniform when $x$ runs in any given compact neighborhood
of $\BA^\times$.

\begin{prp}\label{prp:fiber-1}
For $\pi\in\CA_\cusp(\RG_n)$,
the function $\phi_{\xi,\varphi_\pi}(x)$ as defined in \eqref{fiber} is smooth on $\BA^\times$.
Moreover, if $\xi(g)=\otimes_\nu\xi_\nu=|\det g|^{\frac{n}{2}}\cdot f(g) \in\CS_{\std}(\RG_n(\BA))$ with $f=\otimes_\nu f_\nu\in\CS(\RM_n(\BA))$ and $\varphi_\pi=\otimes_\nu\varphi_{\pi_\nu}$, then
the function defined by
\[
\phi_{\xi,\varphi_\pi}(x)=\prod_{\nu\in|k|}\phi_{\xi_\nu,\varphi_{\pi_\nu}}(x_\nu)
\]
for any $x\in\BA^\times$ belongs to $\CS_\pi(\BA^\times)$.
\end{prp}

\begin{proof}
Since the integral in \eqref{fiber} converges absolutely for any $x\in\BA^\times$, and the convergence is uniform when $x$ runs in any given compact neighborhood of $\BA^\times$, the function
$\phi_{\xi,\varphi_\pi}(x)$ is smooth on $\BA^\times$.

To prove the second statement, we take $f=\otimes_\nu f_\nu\in\CS(\RM_n(\BA))$.
Since $\CC(\pi)=\otimes_\nu\CC(\pi_\nu)$, we take
$\varphi_\pi=\otimes_\nu\varphi_{\pi_\nu}$ with $\varphi_{\pi_\nu}\in\CC(\pi_\nu)$.
Then there exists a finite subset $S_0$ which contains all Archimedean local places of $k$, such that for any
finite local place $\nu$ of $k$, if $\nu\not\in S_0$, then $f_\nu=f^\circ_\nu={\bf 1}_{\RM_n(\Fo_\nu)}$,
$\pi_\nu$ is unramified and $\varphi_{\pi_\nu}=\varphi_{\pi_\nu}^\circ$, which is the zonal spherical
function on $\RG_n(k_\nu)$ associated to $\pi_\nu$. For any $x\in\BA^\times$, and for any finite subset
$S$ of $|k|$ that contains $S_0$ and $x_\nu\in\Fo_\nu^\times$ if $\nu\not\in S$, we have
\begin{align}\label{fiber-3}
\phi_{\xi,\varphi_\pi}(x)=
\int_{\det g=x}
\xi(g)\vphi_\pi(g)\ud_x g
=
\lim_S\prod_{\nu\in S}
\int_{\det g_\nu=x_\nu}
\xi_\nu(g_\nu)\vphi_{\pi_\nu}(g_\nu)\ud_{x_\nu} g_\nu
\end{align}
with $\xi(g)=|\det g|_\BA^{\frac{n}{2}}\cdot f(g)$ and $\xi=\otimes_\nu\xi_\nu$, where
$\xi_\nu(g)=|\det g|_\nu^{\frac{n}{2}}\cdot f_\nu(g)$.
At $\nu\not\in S$, we have $|x_\nu|_\nu=1$ and the local integral
\[
\int_{\det g_\nu=x_\nu}
\xi_\nu(g_\nu)\vphi_{\pi_\nu}(g_\nu)\ud_{x_\nu} g_\nu
=
\int_{\det g_\nu=x_\nu}
{\bf 1}_{\RM_n(\Fo_\nu)}(g_\nu)\vphi_{\pi_\nu}^\circ(g_\nu)\ud_{x_\nu} g_\nu
=
\vol(\RG_n(\Fo_\nu)_{x_\nu})=1.
\]
Hence we obtain that $\phi_{\xi,\varphi_\pi}(x)=\prod_\nu\phi_{\xi_\nu,\varphi_{\pi_\nu}}(x_\nu)$.
\end{proof}

\begin{cor}\label{zeta-zeta}
Assume that $\pi\in\CA_\cusp(\RG_n)$ is unitary. Then for any $\phi=\phi_{\xi,\varphi_\pi}\in\CS_\pi(\BA^\times)$ with
$\xi(g)=|\det g|_\BA^{\frac{n}{2}}\cdot f(g)\in\CS_{\std}(\RG_n(\BA))$ for
some $f\in\CS(\RM_n(\BA))$ and
$\varphi_\pi\in\CC(\pi)$, the following identity
\[
\CZ(s,\phi,\chi)=\CZ(s,f,\varphi_\pi,\chi)
\]
holds for any character $\chi$ of $k^\times\bs\BA^\times$ and $\Re(s)$ sufficiently large.
\end{cor}

\begin{prp}\label{FO-FOnu}
If $\pi\in\CA_\cusp(\RG_n)$, then for any $\phi=\phi_{\xi,\varphi_\pi}\in\CS_\pi(\BA^\times)$ with
$\xi(g)=|\det g|_\BA^{\frac{n}{2}}\cdot f(g)\in\CS_{\std}(\RG_n(\BA))$ for
some $f\in\CS(\RM_n(\BA))$ and
$\varphi_\pi\in\CC(\pi)$, the following identity
\[
\CF_{\pi,\psi}(\phi_{\xi,\vphi_\pi})(x)
=
\phi_{\CF_{\GJ}(\xi),\vphi_\pi^\vee}(x)
\]
for any $x\in\BA^\times$. Moreover, for any $x\in\BA^\times$, the following $\BA^\times$-equivariant
property
\[
\CF_{\pi,\psi}(\phi^x)(y)=\CF_{\pi,\psi}(\phi)(yx^{-1})
\]
holds, where $\phi^x(y):=\phi(yx)$.
\end{prp}

\begin{proof}
Assume that $\phi=\phi_{\xi,\varphi_\pi}\in\CS_\pi(\BA^\times)$ with
$\xi(g)=|\det g|_\BA^{\frac{n}{2}}\cdot f(g)\in\CS_{\std}(\RG_n(\BA))$ for
some $f\in\CS(\RM_n(\BA))$ and
$\varphi_\pi\in\CC(\pi)$ is factorizable: $\phi=\otimes_\nu\phi_\nu$. By definition in \eqref{FO-BA}, we have
\[
\CF_{\pi,\psi}(\phi)(x)=\prod_{\nu\in|k|}\CF_{\pi_\nu,\psi_\nu}(\phi_\nu)(x_\nu).
\]
Write $\phi_\nu(x_\nu)=\phi_{\xi_\nu,\vphi_{\pi_\nu}}(x_\nu)$. Then we have
\[
\CF_{\pi_\nu,\psi_\nu}(\phi_\nu)(x_\nu)
=
\phi_{\CF_{\GJ,\nu}(\xi_\nu),\vphi_{\pi_\nu}^\vee}(x_\nu).
\]
When the data involved are unramified, we have from the simple calculation below \eqref{fiber-3} that
$\CF_{\pi_\nu,\psi_\nu}(\phi_\nu)(x_\nu)=1$. Hence we obtain
\[
\CF_{\pi,\psi}(\phi)(x)=\prod_{\nu}\CF_{\pi_\nu,\psi_\nu}(\phi_\nu)(x_\nu)
=
\prod_\nu\phi_{\CF_{\GJ,\nu}(\xi_\nu),\vphi_{\pi_\nu}^\vee}(x_\nu)
=\phi_{\CF_{\GJ}(\xi),\vphi_\pi^\vee}(x)
\]
as in \eqref{fiber-3}.

In order to verify the $\BA^\times$-equivariant
property: $\CF_{\pi,\psi}(\phi^x)(y)=\CF_{\pi,\psi}(\phi)(yx^{-1})$ for any $x,y\in\BA^\times$, it is
enough to the local Fourier operators $\CF_{\pi_\nu,\psi_\nu}$ for all local place $\nu\in|k|$ enjoy the
same equivariant property. This local equivariant property for the Fourier operators $\CF_{\pi_\nu,\psi_\nu}$ can be deduced from the local functional equation for zeta integral $\CZ(s,\phi,\chi)$ in
Theorem \ref{thm:1-FE} through a simple computation.
\end{proof}

We can deduce the following result from Theorem \ref{thm:GJ-gzi}.

\begin{thm}\label{thm:gzeta}
Let $\pi$ be an irreducible unitary cuspidal automorphic representation of $\RG_n(\BA)$ with
the local component $\pi_\nu$ being of Casselman-Wallach type at all $\nu\in|k|_\infty$.
For any $\phi\in\CS_\pi(\BA^\times)$ and any unitary character $\chi$ of $k^\times\bs\BA^\times$,
the global zeta integral $\CZ(s,\phi,\chi)$ converges absolutely for $\Re(s)>\frac{n+1}{2}$, admits analytic
continuation to an entire function in $s\in\BC$, and satisfies the functional equation
\[
\CZ(s,\phi,\chi)=\CZ(1-s,\CF_{\pi,\psi}(\phi),\chi^{-1}),
\]
where $\CF_{\pi,\psi}$ is the Fourier operator as defined in \eqref{FO-BA} that takes $\CS_\pi(\BA^\times)$
to $\CS_{\wt{\pi}}(\BA^\times)$.
\end{thm}

\subsection{$\pi$-Poisson summation formula}\label{ssec-PSF}
We establish here the Poisson summation formula on $\GL_1$ for the Fourier operator
$\CF_{\pi,\psi}$, which is associated to any $\pi\in\CA_\cusp(\RG_n)$, and takes $\CS_\pi(\BA^\times)$ to
$\CS_{\wt{\pi}}(\BA^\times)$. Technically, it is possible to establish such a summation formula from the global functional equation in Theorem \ref{thm:gzeta}. However, we are going to take a slightly different way
below.

\begin{thm}[$\pi$-Poisson Summation Formula]\label{thm:PSF}
For any  $\pi\in\CA_\cusp(\RG_n)$, take $\wt{\pi}$ to be the contragredient of $\pi$.
For any $\phi\in\CS_\pi(\BA^\times)$, the $\pi$-theta function
\[
\Theta_\pi(x,\phi):=\sum_{\alpha\in k^\times}\phi(\alpha x)
\]
converges absolutely for any $x\in\BA^\times$, and the following identity holds
\[
\Theta_\pi(x,\phi)
=
\Theta_{\wt{\pi}}(x^{-1},\CF_{\pi,\psi}(\phi)),
\]
as functions in $x\in\BA^\times$, where $\CF_{\pi,\psi}$ is the Fourier operator as defined in
\eqref{FO-BA} that takes $\CS_\pi(\BA^\times)$ to $\CS_{\wt{\pi}}(\BA^\times)$.
\end{thm}

\begin{proof}
It is clear that $\Theta_\pi(x,\phi)=\Theta_\pi(1,\phi^x)$ with $\phi^x(y)=\phi(xy)$. By Proposition \ref{FO-FOnu}, we have
$\Theta_{\wt{\pi}}(x^{-1},\CF_{\pi,\psi}(\phi))=\Theta_{\wt{\pi}}(1,\CF_{\pi,\psi}(\phi^x))$. Since
$\phi\in\CS_\pi(\BA^\times)$
is arbitrary, it is enough to show that
\[
\Theta_\pi(1,\phi):=\sum_{\alpha\in k^\times}\phi(\alpha)
\]
converges absolutely and the following identity
\[
\Theta_\pi(1,\phi)
=
\Theta_{\wt{\pi}}(1,\CF_{\pi,\psi}(\phi))
\]
holds.

In order to prove that the summation $\Theta_\pi(1,\phi)$ is absolutely convergent,
we write $\phi=\phi_{\xi,\varphi_\pi}\in\CS_\pi(\BA^\times)$ with
$\xi(g)=|\det g|_\BA^{\frac{n}{2}}\cdot f(g)\in\CS_{\std}(\RG_n(\BA))$ for
some $f\in\CS(\RM_n(\BA))$ and
$\varphi_\pi\in\CC(\pi)$.  From \eqref{gcoeff} we have
\begin{align}\label{ps-1}
\varphi_\pi(g)=\int_{A_n(\BR)^+\RG_n(k)\bs\RG_n(\BA)}\beta_1(hg)\beta_2(h)\ud h
=
\int_{\RG_n(k)\bs\RG_n(\BA)^1}\beta_1(hg)\beta_2(h)\ud h
\end{align}
for some $\beta_1\in V_\pi$ and $\beta_2\in V_{\wt{\pi}}$, where $V_\pi$ is the cuspidal
automorphic realization of $\pi$ in $L^2(\RG_n(k)\bs\RG_n(\BA),\omega)$ and so is $V_{\wt{\pi}}$.

First, we have that
\begin{align}\label{ps-2}
\Theta_\pi(1,\phi)
=
\sum_{\alp\in k^\times}
\phi_{\xi,\varphi_\pi}(\alp)
=
\sum_{\alp\in k^\times}
\int_{\RG_n(\BA)_\alp}
\xi(g)
\int_{\RG_n(k)\bs \RG_n(\BA)^1}
\bet_1(hg)
\bet_2(h)\ud h \ud_\alp g.
\end{align}
From changing variable $g\to h^{-1}g$, we have that $\det g=\alp\cdot\det h$ and \eqref{ps-2} becomes
\begin{align}\label{ps-3}
\int_{\RG_n(k)\bs \RG_n(\BA)^1}
\sum_{\alp\in k^\times}
\int_{\RG_n(\BA)_{\alp\cdot \det h}}
\xi(h^{-1}g)
\bet_1(g)
\bet_2(h)\ud_{\alp\cdot \det h} g\ud h.
\end{align}
For $g\in\RG_n(\BA)_{\alp\cdot \det h}$, we change
$g$ to $t_1(\alp)\cdot y$ with $\det y=\det h$, where $t_1(\alp)={\rm diag}(\alp,\RI_{n-1})\in\RG_n(k)$. Then \eqref{ps-3} can be written as
\begin{align}\label{ps-4}
\int_{\RG_n(k)\bs \RG_n(\BA)^1}
\sum_{\alp\in k^\times}
\int_{\GL_n(\BA)_{\det h}}
\xi(h^{-1}t_1(\alp)g)
\bet_1(g)
\bet_2(h)\ud_{\det h} g\ud h,
\end{align}
since $\bet_1$ is automorhic. For any $h\in\RG_n(\BA)^1$, we have $|\det h|_\BA=1$.
Hence we must have that $\RG_n(\BA)_{\det h}\subset\RG_n(\BA)^1$. It is clear that
$\RG_n(\BA)_{\det h}$ is an $\SL_n(\BA)$-torsor and the measure $\ud_{\det h} g$ is
left $\SL_n(k)$-invariant.
Hence \eqref{ps-4} can be written as
\begin{align}\label{ps-5}
\int_{\RG_n(k)\bs \RG_n(\BA)^1}
\sum_{\alp\in k^\times}
\sum_{\epsilon\in \SL_n(k)}
\int_{\SL_n(k)\bs \RG_n(\BA)_{\det h}}
\xi(h^{-1}t_1(\alp)\epsilon g)
\bet_1(g)
\bet_2(h)\ud_{\det h} g\ud h.
\end{align}
Since any element $\gam\in\RG_n(k)$ can be written as a product of $t_1(\alp)$ and $\epsilon$ in a unique way,
we obtain that \eqref{ps-5} is equal to
\begin{align}\label{ps-6}
\int_{\RG_n(k)\bs \RG_n(\BA)^1}
\int_{\SL_n(k)\bs \RG_n(\BA)_{\det h}}
\left(\sum_{\gam\in \RG_n(k)}
\xi(h^{-1}\gam g)\right)
\bet_1(g)
\bet_2(h)\ud_{\det h} g\ud h.
\end{align}
Since $\xi(g)=|\det g|_\BA^{\frac{n}{2}}\cdot f(g)\in\CS_{\std}(\RG_n(\BA))$ for
some $f\in\CS(\RM_n(\BA))$, and $h\in\RG_n(\BA)^1$ and $g\in\RG_n(\BA)_{\det h}$, we must have
that
\[
\xi(h^{-1}\gam g)=|\det (h^{-1}\gam g)|_\BA^{\frac{n}{2}}\cdot f(h^{-1}\gam g)=f(h^{-1}\gam g).
\]
Hence we obtain that
\begin{align}\label{ps+1}
\sum_{\gam\in \RG_n(k)}
\xi(h^{-1}\gam g)=\sum_{\gam\in \RG_n(k)}
f(h^{-1}\gam g).
\end{align}
By \cite[Lemma 11.7]{GJ72}, for any $f\in\CS(\RM_n(\BA))$, the summation
$
\sum_{\gam\in \RG_n(k)}
f(h^{-1}\gam g)
$
is of moderate growth in $g,h\in\RG_n(k)\bs\RG_n(\BA)$ as an automorphic function on
$\RG_n(k)\bs\RG_n(\BA)\times\RG_n(k)\bs\RG_n(\BA)$, and so is the summation
$
\sum_{\gam\in \RG_n(k)}
\xi(h^{-1}\gam g)
$
as an automorphic function in $g,h\in\RG_n(k)\bs\RG_n(\BA)^1$. Since both $\beta_1(g)$ and $\beta_2(h)$
are cuspidal, we obtain that the integral in \eqref{ps-6} converges absolutely, and so does the
$\pi$-theta function $\Theta_\pi(1,\phi)$ at $x=1$.

Now we continue with the integral in \eqref{ps-6} to prove the identity
\[
\Theta_\pi(1,\phi)
=
\Theta_{\wt{\pi}}(1,\CF_{\pi,\psi}(\phi)).
\]
Recall from \cite[Section 11]{GJ72} and also \cite[Theorem 4.0.1]{Luo} the classical Poisson summation formula:
\begin{align}\label{ps-7}
\sum_{\gam\in \RM_n(k)}f(h^{-1}\gam g) =
\sum_{\gam\in \RM_n(k)}
|\det gh^{-1}|^{-n}_\BA \CF_\psi(f)(g^{-1}\gam h)
\end{align}
holds for any $f\in\CS(\RM_n(\BA))$ and $h,g\in\RG_n(\BA)$.  When $g,h\in\RG_n(\BA)^1$, it can be
re-written according to the rank of $\gam\in\RM_n(k)$ as follows:
\begin{align}\label{ps-8}
\sum_{\gam\in \RG_n(k)}f(h^{-1}\gam g)
&=
\sum_{\gam\in \RG_n(k)}\CF_\psi(f)(g^{-1}\gam h)\nonumber
\\
&+
\sum_{\gam\in \RM_n(k), \rank(\gam)<n}
\CF_\psi(f)(g^{-1}\gam h)
-
\sum_{\gam\in \RM_n(k), \rank(\gam)<n}
f(h^{-1}\gam g).
\end{align}
We denote the boundary terms by
\begin{align}\label{ps-9}
\RB_{f}(h,g):=
\sum_{\gam\in \RM_n(k), \rank(\gam)<n}
\CF_\psi(f)(g^{-1}\gam h)
-
\sum_{\gam\in \RM_n(k), \rank(\gam)<n}
f(h^{-1}\gam g).
\end{align}
Then \eqref{ps-6} can be written as a sum of the following two terms:
\begin{align}\label{ps-10}
\int_{\RG_n(k)\bs \RG_n(\BA)^1}
\int_{\SL_n(k)\bs \RG_n(\BA)_{\det h}}
\left(\sum_{\gam\in \RG_n(k)}
\CF_\psi(f)(g^{-1}\gam h)
\right)
\bet_1(g)
\bet_2(h)\ud_{\det h} g\ud h,
\end{align}
and
\begin{align}\label{ps-11}
\int_{\RG_n(k)\bs \RG_n(\BA)^1}
\int_{\SL_n(k)\bs \RG_n(\BA)_{\det h}}
\RB_{f}(h,g)
\bet_1(g)
\bet_2(h)\ud_{\det h} g\ud h.
\end{align}
From the proof of \cite[Lemma 12.13]{GJ72} and \cite[Lemma 4.1.4]{Luo}, we must have that
the term in \eqref{ps-11} is zero, because of the cuspidality of both $\beta_1(g)$ and $\beta_2(h)$. Hence we obtain that
$\Theta_\pi(1,\phi)=\Theta_\pi(1,\phi_{\xi,\varphi_\pi})$
is equal to the term in \eqref{ps-10}.

Now we write \eqref{ps-10} as
\begin{align}\label{ps-12}
\int_{\RG_n(k)\bs \RG_n(\BA)^1}
\int_{\SL_n(k)\bs \RG_n(\BA)_{\det h}}
\left(\sum_{\gam\in \RG_n(k)}
\CF_\psi(f)((\gam g)^{-1} h)
\right)
\bet_1(g)
\bet_2(h)\ud_{\det h} g\ud h.
\end{align}
By writing back that $\gam=t_1(\alp)\cdot\epsilon$ with $\alp\in k^\times$ and $\epsilon\in\SL_n(k)$, we obtain that
\eqref{ps-12} is equal to
\begin{align}\label{ps-13}
\int_{\RG_n(k)\bs \RG_n(\BA)^1}
\int_{\RG_n(\BA)_{\det h}}
\left(\sum_{\alp\in k^\times}
\CF_\psi(f)((t_1(\alp)g)^{-1} h)
\right)
\bet_1(g)
\bet_2(h)\ud_{\det h} g\ud h.
\end{align}
By changing $t_1(\alp)g$ to $g$, we write \eqref{ps-13} as
\begin{align}\label{ps-14}
\sum_{\alp\in k^\times}
\int_{\RG_n(k)\bs \RG_n(\BA)^1}
\int_{\RG_n(\BA)_{\alp\cdot\det h}}
\CF_\psi(f)(g^{-1} h)
\bet_1(g)
\bet_2(h)\ud_{\alp\cdot\det h} g\ud h.
\end{align}
After changing variable $g\to hg$, \eqref{ps-14} can be written as
\begin{align}\label{ps-15}
\sum_{\alp\in k^\times}
\int_{\RG_n(\BA)_{\alp}}
\CF_\psi(f)(g^{-1})
\int_{\RG_n(k)\bs \RG_n(\BA)^1}
\bet_1(hg)
\bet_2(h)\ud h\ud_{\alp} g,
\end{align}
which is equal to
\begin{align}\label{ps-16}
\sum_{\alp\in k^\times}
\int_{\RG_n(\BA)_{\alp}}
\CF_\psi(f)(g^{-1})
\varphi_\pi(g)\ud_\alp g.
\end{align}
Finally, by changing $g$ to $g^{-1}$, we obtain that \eqref{ps-6} is equal to
\begin{align}\label{ps-17}
\sum_{\alp\in k^\times}
\int_{\RG_n(\BA)_{\alp}}
\CF_\psi(f)(g)
\varphi_\pi(g^{-1})\ud_\alp g
\end{align}
By Proposition \ref{prp:GJ-FO}, when $\det g=\alp\in k^\times$, we have
\[
\CF_\psi(f)(g)=
\CF_{\GJ}(\xi)(g)
\]
for $\xi(g)=|\det g|^{\frac{n}{2}}\cdot f(g)$. Hence the summation in \eqref{ps-17} is equal to
\[
\sum_{\alp\in k^\times}\CF_{\pi,\psi}(\phi_{\xi,\varphi_\pi})(1)
=
\Theta_{\wt{\pi}}(1,\CF_{\pi,\psi}(\phi_{\xi,\varphi_\pi})).
\]
This proves the $\pi$-Poisson summation formula
\[
\Theta_\pi(1,\phi)=\Theta_{\wt{\pi}}(1,\CF_{\pi,\psi}(\phi))
\]
for all $\phi\in\CS_\pi(\BA^\times)$. We are done.
\end{proof}

The proof of the $\pi$-Poisson summation formula in Theorem \ref{thm:PSF} uses the Poisson summation formula associated to the classical Fourier transform $\CF_\psi$ over the affine space $\RM_n(\BA)$, without
using the global functional equation for the global zeta integrals $\CZ(s,\phi,\chi)$ in Theorem \ref{thm:gzeta}.
Hence we are able to obtain the global functional equation for the global zeta integrals
$\CZ(s,\phi,\chi)$ as in Theorem \ref{thm:gzeta} by using the $\pi$-Poisson summation formula
in Theorem \ref{thm:PSF}. Of course, this is essentially the same proof as the one
that  uses the global functional equation of Godement-Jacquet zeta functions in Theorem \ref{thm:GJ-gzi}.
However, it seems still meaningful to point out the contribution of the $\pi$-Poisson summation formulae on
$\GL_1$ in the theory of the global functional equation for the standard automorphic $L$-function
$L(s,\pi\times\chi)$ for any automorphic characters $\chi$ of $\BA^\times$ and any
irreducible cuspidal automorphic representations $\pi$ of $\GL_n(\BA)$, as an extension in a
different prospective of the Tate's thesis to the study of higher degree automorphic $L$-functions.

\section{Absolute Convergence of Generalized Theta Functions}\label{sec-ACGTF}

We explore in this section the possibilities to establish the $\pi$-Poisson summation formulae
(Theorem \ref{thm:PSF}) when the automorphic representation $\pi$ of $\GL_n(\BA)$ may not be cuspidal.

\subsection{Absolute convergence of $\pi$-theta functions}

Recall that we denote by $\Pi_\BA(n)$ the set of all equivalence classes of irreducible admissible representations of
$\RG_n(\BA)$.
For any $\pi=\otimes_\nu\pi_\nu\in\Pi_\BA(n)$, we have, as in \eqref{piSS-BA},
\[
\CS_\pi(\BA^\times)=\otimes_{\nu\in|k|}\CS_{\pi_\nu}(k_\nu^\times).
\]
For $\phi\in\CS_\pi(\BA^\times)$, we are going to show that the $\pi$-theta function
\begin{align}\label{pi-theta-ac}
\Theta_\pi(x,\phi)=\sum_{\alp\in k^\times}\phi(\alp x)
\end{align}
converges absolutely under an assumption (Assumption \ref{ass-FH})
on the unramified local components $\pi_\nu$ of $\pi$.

For any $\pi=\otimes_\nu\pi_\nu\in\Pi_\BA(n)$, let $S_\pi$ be a finite subset of local places of $k$ containing $|k|_\infty$ such that for any finite local place $\nu\not\in S_\pi$, the local component $\pi_\nu$
is unramified. For any $\pi_\nu$ with $\nu\not\in S_\pi$, via the Satake isomorphism, one has the
Frobenius-Hecke conjugacy class $c(\pi_\nu)$ in $\RG_n(\BC)$ associated to $\pi_\nu$. We write
\begin{align}\label{Frob-Hecke}
c(\pi_\nu):={\rm diag}(q_\nu^{s_1(\pi_\nu)}, \cdots, q_\nu^{s_n(\pi_\nu)})\in\GL_n(\BC)=\RG_n(\BC)
\end{align}
up to the adjoint action of $\RG_n(\BC)$, with $s_j(\pi_\nu)\in\BC$ for $j=1,2,\cdots,n$,
where $q_\nu$ is the cardinality of the residue field $\kappa_\nu = \Fo_\nu/\Fp_\nu$.
The following is the assumption we take on the unramified local components $\pi_\nu$ of $\pi$.

\begin{ass}\label{ass-FH}
Let $\pi=\otimes_\nu\pi_\nu\in\Pi_\BA(n)$ be an irreducible admissible representation of $\RG_n(\BA)$.
There exists a positive real number $\kappa_{\pi}$, which depends only on $\pi$, such that
\[
\max_{1\leq j\leq n}\{\Re(s_j(\pi_\nu))\}
<\kappa_\pi
\]
for every $\nu\not\in S_\pi$.
\end{ass}

Then we need to prove some technical local results.

\begin{lem}\label{ext-0}
For any $\pi=\otimes_\nu\pi_\nu\in\Pi_\BA(n)$ with Assumption \ref{ass-FH},
there exists a positive real number $s_\pi\geq \kappa_\pi$ such that
for any real number $a_0> s_\pi$, the following limit
\[
\lim_{|x|_\nu\to 0}\phi_\nu(x)|x|_\nu^{a_0}=0
\]
holds for any $\phi_\nu\in\CS_{\pi_\nu}(k_\nu^\times)$ and any local place $\nu\in|k|$.
In particular, $\phi_\nu(x) |x|_\nu^{a_0}$ extends to a continuous function with compact support on $k_\nu$.
\end{lem}

\begin{proof}
When $\nu\in |k|_\infty$, the asymptotic of $\phi_\nu\in \CS_{\pi_\nu}(k_\nu^\times)$ near $x=0$
is characterized in Definition \ref{dfn:CF}. In particular, following the notation in Definition \ref{dfn:CF}, the fixed sequence $\{\lam_k\}_{k=0}^\infty$ has strictly increasing real part $\{\Re(\lam_k)\}_{k=0}^\infty$. Hence for any positive real number $s_0\in \BR$ satisfying the inequality
$$
s_0+\Re(\lam_0)>0,
$$
the limit
\begin{align}\label{e-0-0}
\lim_{|x|_\nu\to 0}
\phi_\nu(x)\cdot |x|^{s_0}_\nu = 0
\end{align}
holds as the limit formula in Definition \ref{dfn:CF} is term-wise differentiable and uniform (even after term-wise differentiation). Hence the function $\phi_\nu(x)\cdot |x|^{s_0}_\nu$ is continuous with
compact support on $k_\nu$ for any positive real number $s_0$ satisfying $s_0+\Re(\lam_0)>0$.
Since the set $|k|_\infty$ is finite, it is possible to choose a sufficiently positive $s_\infty\in \BR$ such
that the prescribed property holds for all functions $\phi_\nu(x)\cdot |x|^{a_0}_\nu$ with
$\phi_\nu\in \CS_{\pi_\nu}(k_\nu^\times)$ at all $\nu\in |k|_\infty$, as long as $a_0\geq s_\infty$.

It remains to treat the case when $\nu\in |k|_f$, the finite local places of $k$. We consider the local
zeta integrals $\CZ(s,\phi_\nu,\ome_\nu)$ for any $\phi_\nu\in \CS_{\pi_\nu}(k^\times_\nu)$,  and any unitary character $\ome_\nu\in\Ome_\nu^\wedge$. By Theorem \ref{thm:1-zeta}, it converges absolutely for $\Re(s)$ sufficiently positive and admits a meromorphic continuation to $s\in\BC$.
For each $\nu\in|k|_f$, we take $c_{\pi_\nu}$ to be a sufficiently positive real number, such that
$\CZ(s,\phi_\nu,\ome_\nu)$ converges absolutely for $\Re(s)>c_{\pi_\nu}$.
If $\nu\not\in S_\pi$, then $\pi_\nu$ is unramified. In this case,  the zeta integral
$\CZ(s,\phi_\nu,\ome_\nu)$ converges absolutely for $\Re(s)>\kappa_\pi$, where the positive real number $\kappa_\pi$ depends on $\pi$ only, according to Assumption \ref{ass-FH}. Hence if we
take a positive real number $c_\pi$ with
\begin{align}\label{cpi}
c_\pi:=\max\{\kappa_\pi,c_{\pi_\nu}\mid \nu\in S_\pi\cap|k|_f\},
\end{align}
then for any $\phi_\nu\in \CS_{\pi_\nu}(k^\times_\nu)$,  and any unitary character
$\ome_\nu\in\Ome_\nu^\wedge$, the local zeta integrals $\CZ(s,\phi_\nu,\ome_\nu)$
converges absolutely for $\Re(s)>c_{\pi}$ at all finite local places $\nu\in|k|_f$.

By the Mellin inversion formula as displayed in \eqref{eq:MIT-na}, we have
\begin{align}\label{e-0-1}
\phi_\nu(x)\cdot|x|_\nu^{d} =
\sum_{\ome_\nu\in \Ome_\nu^{\wedge}}
\left(
\Res_{z=0}\left(\CZ(s+d,\phi_\nu,\ome_\nu)|x|^{-s}_\nu q_\nu^s\right)
\right)
\ome_\nu(\ac(x))^{-1},
\end{align}
where $z=q_\nu^{-s}$ and $d>c_\pi$. Since the summation on the right-hand side is finite, it suffices to show that the following limit formula
\begin{align}\label{e-0-2}
\lim_{|x|_\nu\to 0}
\Res_{z=0}\left(\CZ(s+d,\phi_\nu,\ome_\nu)|x|^{-s}_\nu q_\nu^s\right) = 0
\end{align}
holds for each $\ome_\nu\in \Ome_\nu^\wedge$.

It is clear that $\CZ(s+d,\phi_\nu,\ome_\nu)$ is holomorphic for $\Re(s)>-(d-c_\pi)$.
By Theorem \ref{thm:1-zeta}, we have
\[
\CZ(s+d,\phi_\nu,\ome_\nu)
=
p_\nu(s)\cdot L(s+d,\pi_\nu\times\ome_\nu),
\]
where $p_\nu(s)\in \BC[q_\nu^s,q_\nu^{-s}]$, depending on $\phi_\nu$.
By the supercuspidal support of
$\pi_\nu\otimes\ome_\nu$, we obtain that the representation $\pi_\nu\otimes\ome_\nu$ can be
embedded, as an irreducible subrepresentation, into the following induced representation
\[
\pi_\nu\otimes\ome_\nu\hookrightarrow\Pi_\nu:=\Ind^{\RG_n(k_\nu)}_{P(k_\nu)}\tau_{\nu,1}\otimes\cdots\otimes\tau_{\nu,t_\nu}
\]
where $\tau_{\nu,j}$ is an irreducible supercuspidal representation of $\RG_{a_{\nu,j}}(k_\nu)$ with
$n=a_{\nu,1}+\cdots+a_{\nu,t_\nu}$ (c.f. \cite{J71}).
By \cite[Theorem 3.4]{GJ72}, we have
\[
L(s,\Pi_\nu)=L(s,\tau_{\nu,1})\cdots L(s,\tau_{\nu,t_\nu}).
\]
By \cite[Corollary 3.6]{GJ72}, we have
\[
\frac{L(s,\pi_\nu\times\ome_\nu)}{L(s,\Pi_\nu)}
\]
is a polynomial in $q_\nu^{-s}$. Hence we obtain that for the given
$\phi_\nu\in\CS_{\pi_\nu}(k_\nu^\times)$, there exists a polynomial $\CP_v(s)$ in $q_\nu^s$ and $q_\nu^{-s}$, depending on $\pi_\nu\otimes\ome_\nu$ and $\phi_\nu$,
such that
\begin{align}\label{e-0-3}
\CZ(s+d,\phi_\nu,\ome_\nu)
=
\CP_\nu(s)L(s+d,\Pi_\nu).
\end{align}
By applying \cite[Proposition 5.11]{GJ72} to the local $L$-functions $L(s,\tau_{\nu,j})$, we obtain that
$L(s,\tau_{\nu,j})=1$ when $\tau_{\nu,j}$ is either supercuspidal ($a_{\nu,j}\geq 2$) or
a ramified character ($a_{\nu,j}=1$). Hence there exists an integer $1\leq r_\nu\leq t_\nu\leq n$, such that
\begin{align}\label{e-0-4}
\CZ(s+d,\phi_\nu,\ome_\nu)
=
\CP_\nu(s)
\prod_{j=1}^{r_\nu}\frac{1}{1-q_\nu^{-s-d+s_{\nu,j}}}
=
\prod_{j=1}^{r_\nu}
\left(\sum_{\ell_j=0}^\infty q_\nu^{-(s+d-s_{\nu,j})\ell_j}\right)
\end{align}
for some $s_{\nu,j}\in\BC$, with $j=1,2,\cdots,r_\nu$.

Now we are ready to discuss the limit in \eqref{e-0-2}. For $z=q_\nu^{-s}$, we have
\begin{align}\label{e-0-5}
\CZ(s+d,\phi_\nu,\ome_\nu)|x|^{-s}_\nu q_\nu^s
=
\CP_\nu(z)\cdot
\frac{\prod_{j=1}^{r_\nu}
\left(\sum_{\ell_j=0}^\infty q_\nu^{-\ell_j (d-s_{\nu,j})}\cdot z^{\ell_j}\right)}{z^{\ord_\nu(x)+1}},
\end{align}
where $\CP_\nu(z)$ is a polynomial function in $z,z^{-1}$.
By taking the residue at $z=0$, we obtain that
\begin{align}\label{e-0-6}
\Res_{z=0}\left(\CZ(s+d,\phi_\nu,\ome_\nu)|x|^{-s}_\nu q_\nu^s\right)
=
\FC_0(x)
\end{align}
where $\FC_0(x)$ is the coefficient of the constant term of
\begin{align}\label{e-0-7}
\CP_\nu(z)\cdot
\frac{\prod_{j=1}^{r_\nu}
\left(\sum_{\ell_j=0}^\infty q_\nu^{-\ell_j (d-s_{\nu,j})}\cdot z^{\ell_j}\right)}{z^{\ord_\nu(x)}}.
\end{align}
Since $\CP_\nu(z)$ is a polynomial function in $z,z^{-1}$ with degree depending on $\pi$,
without loss of generality, we may assume that $\CP_\nu(z)\equiv1$ when we compute $\FC_0(x)$.
In this case, the constant term of \eqref{e-0-7} with $\CP_\nu(z)\equiv1$ is equal to
\begin{align}
\sum_{\substack{\ell_1+\cdots+\ell_{r_\nu}=\ord_\nu(x)\\ \ell_1,\cdots,\ell_{r_\nu}\geq 0}}
q_\nu^{-\ell_1 (d-s_{\nu,j})-\cdots-\ell_{t_0} (d-s_{\nu,j})}.
\end{align}
When $\nu\notin S_\pi$, $\pi_\nu$ is unramified,
\[
{\rm diag}(q_\nu^{s_{\nu,1}},\cdots,q_\nu^{s_{\nu,n}})=c(\pi_\nu)
\]
is the Frobenius-Hecke conjugacy class associated to $\pi_\nu$ in $\RG_n(\BC)$ with
$s_{\nu,j}=s_j(\pi_\nu)$ for $j=1,2,\cdots,n$.
By Assumption \ref{ass-FH} and the definition of the positive real number $c_\pi$ as in \eqref{cpi}, we take $d_0=0$ and have
\begin{align}\label{cpi-1}
d-\Re(s_{\nu,j})>c_\pi-\Re(s_{\nu,j})\geq 0
\end{align}
for all $j=1,2,\cdots,n$. For the remaining finite local places $\nu$, we may choose a positive
real number $d_0$ such that
\begin{align}\label{cpi-2}
d+d_0-\Re(s_{\nu,j})>c_\pi+d_0-\Re(s_{\nu,j})\geq 0
\end{align}
for all $j=1,2,\cdots,r_\nu$ and all $\nu\in S_\pi\cap|k|_f$.
Hence with the choice of $d_0$, we have
\begin{align}\label{e-0-8}
\left|\Res_{z=0}\left(\CZ(s+d+d_0,\phi_\nu,\ome_\nu)|x|^{-s}_\nu q_\nu^s\right)\right|
&\leq
\sum_{\substack{\ell_1+\cdots+\ell_{r_\nu}=\ord_\nu(x)\\ \ell_1,\cdots,\ell_{r_\nu}\geq 0}}
q_\nu^{-\sum_{j=1}^{r_\nu}\ell_j (d+d_0-\Re(s_{\nu,j}))}\nonumber\\
&\leq
\sum_{\substack{\ell_1+\cdots+\ell_{r_\nu}=\ord_\nu(x)\\ \ell_1,\cdots,\ell_{r_\nu}\geq 0}}
q_\nu^{-\ord_\nu(x)(d+d_0-\max_j\{\Re(s_{\nu,j})\})}\nonumber\\
&=
\binom{\ord_\nu(x)+r_\nu-1}{r_\nu-1}
\cdot
q_\nu^{-\ord_\nu(x)(d+d_0-\max_j\{\Re(s_{\nu,j})\})}.
\end{align}
Since $d+d_0-\max_j\{\Re(s_{\nu,j})\}>0$, and the function
$\binom{\ord_\nu(x)+r_\nu-1}{r_\nu-1}$ is a polynomial in $\ord_\nu(x)$, we must have that
\[
\lim_{\ord_\nu(x)\to+\infty}\binom{\ord_\nu(x)+r_\nu-1}{r_\nu-1}
\cdot
q_\nu^{-\ord_\nu(x)(d+d_0-\max_j\{\Re(s_{\nu,j})\})}=0.
\]
Therefore, by taking a positive real number $s_\pi=\max\{s_\infty, c_\pi+d_0\}$,
we obtain that for any $a_0>s_\pi$, the function $\phi_\nu(x)|x|_\nu^{a_0}$ is continuous over $k_\nu$ and
has the limit
\[
\lim_{|x|_\nu\to 0}\phi_\nu(x)|x|_\nu^{a_0}=0,
\]
for any $\phi_\nu\in\CS_{\pi_\nu}(k_\nu^\times)$ and at any local place $\nu\in|k|$.
We are done.
\end{proof}

\begin{lem}\label{basic}
Let $\pi=\otimes_\nu\pi_\nu\in\Pi_\BA(n)$ satisfy Assumption \ref{ass-FH}.
For any $\nu\notin S_\pi$, the basic function $\BL_{\pi_\nu}\in\CS_{\pi_\nu}(k_\nu^\times)$
is supported on $\Fo_\nu-\{0\}$ with
$$
\BL_{\pi_\nu}(\Fo^\times_\nu) = 1.
$$
Moreover, there exists a positive real number $b_\pi\geq s_\pi$, which is independent of $\nu$, such that for any $b_0> b_\pi$,
$$
\left|\BL_{\pi_\nu}(x)\cdot
|x|_\nu^{b_0}\right|\leq 1
$$
holds for all $\nu\notin S_\pi$.
\end{lem}

\begin{proof}
We continue with the proof of Lemma \ref{ext-0} for the non-Archimedean case, and specialize it
to the unramified situation.
Note that the basic function $\BL_{\pi_\nu}\in\CS_{\pi_\nu}(k_\nu^\times)$ is the Mellin inversion of the local unramified $L$-factor
\[
\CZ(s,\BL_{\pi_\nu})=L(s,\pi_\nu),
\]
whose Mellin inversion can be calculated by \eqref{e-0-1} after setting $\CP_\nu(s) = 1$.
In other words, taking the constant $s_\pi$ as in Lemma \ref{ext-0}, we have, for any $a_0>s_\pi$,
$$
\BL_{\pi_\nu}(x)\cdot|x|^{a_0} =
\Res_{z=0}\left(\CZ(s+a_0,\BL_{\pi_\nu})
|x|^{-s}_\nu q_\nu^s\right).
$$
As in \eqref{e-0-5}, we write
\begin{align}\label{e-0-9}
\CZ(s+a_0,\BL_{\pi_\nu})
=
\frac{1}{\prod_{j=1}^n(1-q_\nu^{-s-a_0+s_j(\pi_\nu)})}
=
\prod_{j=1}^n
\left(
\sum_{\ell_j\geq 0}
q_\nu^{\ell_j(s_j(\pi_\nu)-a_0)}z^{\ell_j}
\right).
\end{align}
where we write $z=q_\nu^{-s}$ and $c_j(\pi_\nu)=q_\nu^{s_j(\pi_\nu)}$.
From the Laurent expansion on the right-hand side, we obtain that the function
\[
\CZ(s+a_0,\BL_{\pi_\nu})
|x|^{-s}_\nu q_\nu^{s}
\]
is holomorphic in $z=q_\nu^{-s}$ whenever $x\notin \Fo_\nu$. By taking the residue at $z=0$, we
obtain that
\[
\BL_{\pi_\nu}(x) \cdot|x|^{a_0}= 0\quad \text{for}\quad x\notin \Fo_\nu.
\]
Hence the basic function $\BL_{\pi_\nu}(x)$ has support included in $\Fo_\nu$.
Similarly, when $x\in \Fo^\times_\nu$, $|x|_\nu= 1$ and
$$
\Res_{z=0}\left(\CZ(s+a_0,\BL_{\pi_\nu}) q_\nu^s\right)= 1
$$
because of the constant term on the right-hand side of \eqref{e-0-9}.
Therefore, we obtain
$$
\BL_{\pi_\nu}(\Fo^\times_\nu)  =1.
$$

Finally, whenever $x\in \Fo_F\smallsetminus\{0\}$, we apply \eqref{e-0-8} to the
unramified case, and obtain that
\[
\left|\BL_{\pi_\nu}(x)\cdot |x|^{b}\right|
\leq
\binom{\ord_\nu(x)+n-1}{n-1}
\cdot
q_\nu^{-\ord_\nu(x)\cdot \min_{j}\{b-\Re(s_j(\pi_\nu))\}},
\]
as long as $b> s_\pi$.
By Assumption \ref{ass-FH}, we have
\[
{\min}_{1\leq j\leq n}\{b-s_j(\pi_\nu)\}>{\min}_{1\leq j\leq n}\{\kappa_\pi-s_j(\pi_\nu)\}>0.
\]
Therefore whenever $\ord_\nu(x)\geq 1$,
\[
\binom{\ord_\nu(x)+n-1}{n-1}
\cdot
q_\nu^{-\ord_\nu(x)\cdot \min_{j}\{b-\Re(s_j(\pi_\nu))\}}
\leq \binom{\ord_\nu(x)+n-1}{n-1}
\cdot
2^{-\ord_\nu(x)\cdot \min_{j}\{b-\Re(s_j(\pi_\nu))\}}
\]
since $q_\nu\geq 2$ for any $\nu\notin S_\pi$. It turns out that we only need to find a positive integer
$b_\pi\geq s_\pi\in \BR$ such that for any $b> b_\pi$, the following
\[
\binom{\ord_\nu(x)+n-1}{n-1}
\cdot 2^{-\ord_\nu(x)\cdot \min_{j}\{b-\Re(s_j(\pi_\nu))\}}
\leq 1
\]
holds for any $\nu\notin S_\pi$ and $\ord_\nu(x)\geq 1$. Equivalently, after applying the function $\log_2$ on both sides, the above inequality becomes
\[
\log_2\binom{\ord_\nu(x)+n-1}{n-1}
-\ord_\nu(x)\cdot \min_{j}\{b-\Re(s_j(\pi_\nu))\}
\leq0.
\]
Hence it suffices to show the existence of $b_\pi\in \BR$ so that
\begin{align*}
\min_{j}\{b-\Re(s_j(\pi_\nu))\}
&=b-\max_j\{\Re(s_j(\pi_\nu))\}
\\
&>b_\pi-\kappa_\pi
\geq
\frac{\log_2\binom{\ord_\nu(x)+n-1}{n-1}}{\ord_\nu(x)}
\end{align*}
for any $\ord_\nu(x)\geq 1$, i.e.
\begin{equation}\label{eq:boundspi}
b_\pi\geq\kappa_\pi+\frac{\log_2\binom{\ord_\nu(x)+n-1}{n-1}}{\ord_\nu(x)}
\end{equation}
for any $\ord_\nu(x)\geq 1$. As a function of $t\geq 1$,
$$
\log_2\binom{t+n-1}{n-1}
=\log_2\frac{\prod_{k=1}^{n-1}(t+k)}{(n-1)!}
\geq
\log_2\frac{\prod_{k=1}^{n-1}(1+k)}{(n-1)!}\geq \log_2(n)\geq 0,
$$
Thus we obtain that
$$
\frac{\log_2\binom{t+n-1}{n-1}}{t}\geq 0
$$
for any $t\geq 1$. On the other hand, by L'Hôspital's rule, one must have that
$$
\lim_{t\to \infty}
\frac{\log_2\binom{t+n-1}{n-1}}{t}
=0.
$$
It follows that as a continuous function in $t\geq 1$, there exists a constant $c_0\in \BR$ such that
$$
\frac{\log_2\binom{t+n-1}{n-1}}{t}
< c_0
$$
for any $t\geq 1$. It is clear now that the inequality in \eqref{eq:boundspi} holds for any
$$
b_\pi\geq \kappa_\pi+c_0.
$$
Therefore it suffices to take $b_\pi = \max\{s_\pi,\kappa_\pi+c_0\}$. We are done.
\end{proof}

\begin{thm}[Absolute Convergence of $\pi$-Theta Functions]\label{thm:AC}
Fix any $\pi = \otimes_\nu \pi_\nu\in\Pi_\BA(n)$ with Assumption \ref{ass-FH}.
Then for any $\phi\in \CS_\pi(\BA^\times)$, the $\pi$-theta function
$$
\Theta_\pi(x,\phi):=\sum_{\alp\in k^\times}\phi(\alp x)
$$
is absolutely convergent for any $x\in \BA^\times$.
\end{thm}

\begin{proof}
For any $\pi=\otimes_\nu\pi_\nu\in\Pi_\BA(n)$, let $S_\pi$ be a finite subset of local places of $k$ containing $|k|_\infty$ and for any finite local place $\nu\not\in S_\pi$, the local component
$\pi_\nu$ is unramified.
We may assume that $\phi\in \CS_\pi(\BA^\times)$ is a pure restricted tensor of the form
$$
\phi = \left(\otimes_{\nu\notin S_\pi}\BL_{\pi_\nu}\right)\otimes
\left(\otimes_{\nu\in S_\pi} \phi_{\nu}\right)
$$
with $\phi_\nu\in\CS_{\pi_\nu}(k_\nu^\times)$ for all $\nu\in S_\pi$.

Fix a positive real number $s_0>b_\pi\geq s_\pi\geq\kappa_\pi$  where the constants $\kappa_\pi$, $s_\pi$, and $b_\pi$ are as given in Assumption \ref{ass-FH},  Lemma \ref{ext-0}, and Lemma \ref{basic}, respectively.
By Lemma \ref{basic}, for any $\nu\notin S_\pi$, we have the function
$\BL_{\pi_\nu}(x)|x|^{s_0}_\nu$ is continuous on $k_\nu$ and supported on $\Fo_\nu$. Moreover,
we have
$$
\big|\BL_{\pi_\nu}(x)|x|^{s_0}_\nu\big|\leq 1
$$
for every $\nu\notin S_\pi$. Similarly, for any finite $\nu\in S_\pi$, the function
$\phi_{\nu}(x)|x|^{s_0}_\nu$
is continuous on $k_\nu$ with compact support. We may assume that the support of
$\phi_{\nu}(x)|x|^{s_0}_\nu$ is contained in a fractional ideal $\Fp_\nu^{m_\nu}$ for some integer $m_\nu\in \BZ$.

For any $\alp\in k^\times$, the Artin product formula shows that $|\alp|_\BA=1$ (\cite{W73}).
Hence we obtain
that
$$
\Theta_\pi(1,\phi)
=\sum_{\alp\in k^\times}
\phi(\alp)
=
\sum_{\alp\in k^\times}
\phi(\alp)
\cdot
|\alp|_{\BA}^{s_0}.
$$
Write $\Fo_\phi:=\prod_{\nu\notin S_\pi}\Fo_\nu$ and $\Fm_\phi:=\prod_{\nu\in S_\pi\cap|k|_f}
\Fp^{m_\nu}$. Then by the weak approximation theorem (\cite{W73}), the product
$\Fo_\phi\cdot\Fm_\phi=\Fm(\phi)$ is a fractional ideal of $\Fo=\Fo_k$, the ring
of integers in $k$.
From the support of the functions $\phi_\nu\cdot|\cdot|^{s_0}$ for all $\nu\in|k|_f$, we write
\[
\Theta_\pi(1,\phi)
=
\sum_{\alp\in k^\times\cap \Fm(\phi)}
\left(\prod_{\nu\in |k|_\infty}
\phi_\nu(\alp)\cdot |\alp|_\nu^{s_0}\right)
\cdot
\left(\prod_{\nu\in |k|_f}
\phi_\nu(\alp)\cdot |\alp|_\nu^{s_0}\right),
\]
which, up to a positive constant, is essentially bounded by
$$
\sum_{\alp\in k^\times\cap \Fm(\phi)}
\prod_{\nu\in |k|_\infty}
\left|\phi_\nu(\alp)\right|\cdot |\alp|_\nu^{s_0}.
$$
Hence it suffices to show that the following summation
\begin{align}\label{ac-1}
\sum_{\alp\in \Fm(\phi)}
\prod_{\nu\in |k|_\infty}
\left|\phi_{\nu}(\alp)\right|
\cdot |\alp|^{s_0}_\nu
\end{align}
is absolutely convergent.

To do so, we consider the following compact set
$$
\CB_\infty(1) := \{(\alp_\nu)\in \prod_{\nu\in |k|_\infty}k_\nu\mid\ |\alp_\nu|_\nu\leq 1,
\forall\nu\in |k|_\infty\}.
$$
We write \eqref{ac-1} as
\begin{align}\label{ac-2}
\sum_{\alp\in \Fm(\phi)\cap\CB_\infty(1)}
\prod_{\nu\in |k|_\infty}
\left|\phi_{\nu}(\alp)\right|
\cdot |\alp|^{s_0}_\nu
+
\sum_{\alp\in\Fm(\phi)\smallsetminus\left(\Fm(\phi)\cap\CB_\infty(1)\right)}
\prod_{\nu\in |k|_\infty}
\left|\phi_{\nu}(\alp)\right|
\cdot |\alp|^{s_0}_\nu
\end{align}
Since the fractional ideal $\Fm_\phi$ is a lattice in the affine space $\prod_{\nu\in |k|_\infty}k_\nu$,
its intersection with $\CB_\infty(1)$ is a finite set.  By Lemma \ref{ext-0}, the function
$\prod_{\nu\in |k|_\infty}\phi_{\nu}(x)|x|^{s_0}_\nu$ is continuous over
$\prod_{\nu\in |k|_\infty}k_\nu$, and hence is bounded over $\CB_\infty(1)$. Thus, the first
summation in \eqref{ac-2} is bounded.
For the second summation in \eqref{ac-2}, the function
$\prod_{\nu\in |k|_\infty}\phi_{\nu}(x)\cdot |x|^{s_0}_\nu$ is of Schwartz type over
$(\prod_{\nu\in |k|_\infty}k_\nu)\smallsetminus\CB_\infty(1)$. Hence the second summation in \eqref{ac-2}
is bounded by the same proof for the absolute convergence of the classical Poisson summation formula
(\cite{W65} and also \cite[Chapter 4]{Ig78}).
This proves the absolute convergence of $\Theta_\pi(x,\phi)$ for any
$\phi\in\CS_\pi(\BA^\times)$.

For any $x\in\BA^\times$, we have $\Theta_\pi(x,\phi)=\Theta_\pi(1,\phi^x)$ with $\phi^x(y)=\phi(yx)$.
Hence $\Theta_\pi(x,\phi)$ converges absolutely for any $\phi\in\CS_\pi(\BA^\times)$. We are done.
\end{proof}

\subsection{Justification of Assumption \ref{ass-FH}}\label{ssec-JAss}

We prove Assumption \ref{ass-FH} when $\pi\in\CA(\RG_n)$ is any irreducible admissible automorphic
representation of $\RG_n(\BA)$.
\begin{prp}\label{prp:Ass}
For any $\pi\in \CA(\RG_n)$, Assumption \ref{ass-FH} holds.
\end{prp}

\begin{proof}
A cuspidal datum $(\RP,\veps)$ of $\RG_n$ consists of a standard parabolic subgroup $\RP$ of $\RG_n$ with Levi decomposition
$\RP= \RM\cdot \RN$ with the Levi subgroup $\RM$ and the unipotent radical $\RN$,
and an irreducible cuspidal automorphic representation $\veps$ of $\RM(\BA)$, which is
square-integrable up to a twist of automorphic character of of $\RM(\BA)$.
For any $\pi=\otimes_{\nu\in|k|}\pi_\nu\in \CA(\RG_n)$, by \cite{L79}, there exists a cuspidal datum $(\RP,\veps)$ of $\RG_n$, such that $\pi$ is a subquotient of the induced representation
$\Ind^{\RG_n(\BA)}_{\RP(\BA)}(\veps)$ of $\RG_n(\BA)$. It follows that
for any $\nu\in |k|$, the $\nu$-component $\pi_\nu$ is a subquotient of the induced representation
$\Ind^{\RG_n(k_\nu)}_{\RP(k_\nu)}(\veps_\nu)$ of $\RG_n(k_\nu)$, where $\veps_\nu$ is the
$\nu$-component of $\veps=\otimes_\nu\veps_\nu$.

Let $\RT$ be the maximal torus of $\RG_n$, consisting of
all diagonal matrices, and $\RB=\RT\cdot\RU$ be the Borel subgroup of $\RG_n$, consisting of all
upper-triangular matrices.
Take $S$ to be a finite subset of $|k|$, such that $S$ contains $|k|_\infty$ and for any $\nu\notin S$,
$\pi_\nu$ and $\veps_\nu$ are unramified.
It is well-known (see \cite{C79} for instance) that there exists an unramified character $\eta_\nu$ of
the maximal torus $\RT(k_\nu)$, such that $\veps_\nu$ embeds as a subrepresentation into
the unramified induced representation $\Ind^{\RM(k_nu)}_{(\RM\cap \RB)(k_\nu)}(\eta_\nu)$.
By induction in stages, we obtain that
$\Ind^{\RG_n(k_\nu)}_{\RP(k_\nu)}(\veps_\nu)$ embeds as a subrepresentation into the spherical
induced representation $\Ind^{\RG_n(k_\nu)}_{\RB(k_\nu)}(\eta_\nu)$ of $\RG_n(k_\nu)$.
Hence the irreduicble spherical representation $\pi_\nu$ is the unique spherical subquotient of
$\Ind^{\RG_n(k_\nu)}_{\RB(k_\nu)}(\eta_\nu)$. Via the Satake isomorphism, the Frobenius-Hecke
conjugacy class of $\pi_\nu$ in $\RG_n(\BC)$ is
\[
c(\pi_\nu)={\diag}(\eta_\nu^1(\vpi_\nu), \cdots,\eta_\nu^n(\vpi_\nu)).
\]
Here $\vpi_\nu$ is the uniformizer of the prime ideal $\Fp_\nu$, and
for any $t={\rm diag}(t_1,\cdots,t_n)\in\RT(k_\nu)$,
the unramified character $\eta_\nu$ is given by
\[
\eta_\nu(t)=\eta_\nu^1(t_1)\cdots\eta_\nu^n(t_n).
\]
It is clear that the conjugacy class of the semisimple element $c(\pi_\nu)$ in the complex dual group
$\RM(\BC)$ of the Levi subgroup $\RM$ is the Frobenius-Hecke conjugacy class $c(\veps_\nu)$ of
$\veps_\nu$. In other words, both $\pi_\nu$ and $\veps_\nu$ shares the same Satake parameter in
$\RT(\BC)^{\RW_n}$, where $\RW_n$ is the Weyl group of $\RG_n$.

Take $\del_\veps$ to be an automorphic character of $\RM(\BA)$ such that
$\del_\veps\otimes\veps$ is square-integrable modulo the center of $\RM$. Then for $\nu\notin S$,
the $\nu$-component $(\del_\veps\otimes\veps)_\nu$ is spherical and unitary. By the classification of
the spherical unitary dual of $\GL_n$ over a non-Archimedean local field $k_\nu$ (\cite{Td86}), we obtain
$$
\left|\log_{q_\nu}\left(\max_{1\leq j\leq n}\{\left|
(\del_\veps)_\nu^j(\vpi_\nu)\eta_\nu^j(\vpi_\nu)
\right|\}
\right)
\right|
\leq\frac{n-1}{2}.
$$
Since the unramified part of the automorphic character $\del_\veps$ is completely determined by
$\veps$ and the cuspidal datum $(\RP,\veps)$ of $\pi$ is uniquely determined by $\pi$, up to conjugation, we obtain that there exists a positive real number $\kappa_\pi$, depending only on
$\pi\in\CA(\RG_n)$, such that
\[
\left|\log_{q_\nu}\left(\max_{1\leq j\leq n}\{\left|
\eta_\nu^j(\vpi_\nu)
\right|\}
\right)
\right|
<\kappa_\pi.
\]
This justifies the assumption.
\end{proof}

By Theorem \ref{thm:AC} and Proposition \ref{prp:Ass},  we obtain the following absolute convergence.

\begin{cor}\label{cor:AC}
For any $\pi\in\CA(\RG_n)$, the $\pi$-theta function
\[
\Theta_\pi(x,\phi)
=\sum_{\alp\in k^\times}\phi(\alp x)
\]
converges absolutely for any $\phi\in\CS_\pi(\BA^\times)$ and any $x\in\BA^\times$.
\end{cor}

Another consequence of Proposition \ref{prp:Ass} is the absolute convergence of global zeta integral
of Godement-Jacquet type for any $\pi\in\CA(\RG_n)$.

\begin{cor}\label{cor:covgeneralautomorphic}
For any $\pi\in \CA(\RG_n)$, there exists a positive real number $r_\pi\in \BR$, such that the global zeta integral
$$
\CZ(s,f,\vphi_\pi) =
\int_{\GL_n(\BA)}
f(g)\vphi_\pi(g)|\det g|_\BA^{s+\frac{n-1}{2}}\ud g,\quad f\in \CS(\RM_n(\BA)),\vphi_\pi\in \CC(\pi)
$$
is absolutely convergent for any $\Re(s)>r_\pi$.
\end{cor}
\begin{proof}
There is no harm to assume that $f= \otimes_\nu f_\nu$ is a pure restricted tensor. Similarly, one can write $\vphi_\pi = \prod_\nu \vphi_{\pi_\nu}$. For the given $\pi\in\CA(\RG_n)$, take the finite
subset $S$ of $|k|$ as in the proof of Proposition \ref{prp:Ass}. Then for $\nu\notin S$,
the function $f_\nu$ is the charactersitic function of $\RM_n(\Fo_\nu)$, and $\vphi_{\pi_\nu}$ is the zonal spherical function attached to the unramified representation $\pi_\nu$.
From \cite[Chapter I, \S7]{GJ72}, we have
$$
\CZ(s,f_\nu,\vphi_{\pi_\nu})  =
\frac{1}{\det(\RI_n-\alp(\pi_\nu)q_\nu^{-s})}=L(s,\pi_\nu)
$$
where the left-hand side is absolutely convergent whenever $\Re(s)>\kappa_\pi$,
where $\kappa_\pi$ is determined in the proof of Proposition \ref{prp:Ass}. It follows that
$$
\prod_{\nu\notin S}
\CZ(s,f_\nu,\vphi_{\pi_\nu})
=
\prod_{\nu\notin S}
\frac{1}{\det(\RI_n-\alp(\pi_\nu)q_\nu^{-s})}=L^S(s,\pi)
$$
is absolutely convergent for $\Re(s)>\kappa_\pi+1$.
As $S$ is a finite set, it is clear that one can choose a
real number $r_\pi$ to be sufficiently positive (depending on $\pi$ only) such that the global zeta integral
\[
\CZ(s,f,\vphi_\pi)=L^S(s,\pi)\cdot\prod_{\nu\in S}\CZ(s,f_\nu,\vphi_{\pi_\nu})
\]
converges absolutely for $\Re(s)>r_\pi$. We are done.
\end{proof}

\section{$(\sigma,\rho)$-Theta Functions on $\GL_1$}\label{sec-sigma-rhoTheta}

For any $k$-split reductive group $G$ and $\rho\colon G^\vee\to\GL_n(\BC)$, we are going to introduce
$(\sigma,\rho)$-theta functions for any unitary $\sigma\in\Pi_\BA(G)$, the set of equivalence classes of
irreducible admissible representations of $G(\BA)$, by means of the existence of
the local Langlands reciprocity map as in the local Langlands conjecture for $G$.

\subsection{On the local Langlands conjecture}\label{ssec-LLC}
We briefly review the local Langlands conjecture for $G$ over any local field $F=k_\nu$ for any local place $\nu\in|k|$.

For any Archimedean local fields, the local Langlands conjecture for $G$ is a theorem of
Langlands, which follows from the Langlands classification theory (\cite{L89}). At any non-Archimedean local places, for unramified representations, their local Langlands parameters are uniquely determined by the Satake isomorphism (\cite{S63}) and also \cite{C79}). In the following we review the local Langlands
conjecture for $F$-split reductive group $G$ over non-Archimedean local field $F$ of characteristic zero.

Let $\CW_F$ be the Weil group attached to $F$. The set of local Langlands parameters is denoted by
$\Phi_F(G)$, which continuous, Frobenius semisimple homomorphisms
\begin{align}\label{L-p}
\varsigma\colon \CW_F\times\SL_2(\BC)\longrightarrow G^\vee,
\end{align}
up to conjugation by $G^\vee$. The local Langlands conjecture asserts that there exists a reciprocity map
\begin{align}\label{FR}
\FR_{F,G}\colon \Rep(G(F))\longrightarrow\Phi_F(G),
\end{align}
where $\Rep(G(F))$ is the set of equivalence classes of smooth representations of
$G(F)$ of finite length. $\FR_{F,G}$ is expected to be surjective with finite fibers, and to satisfy
a series of compatibility conditions. One of the key issues is to formulate and prove the uniqueness of such local Langlands reciprocity map.

When $G=\GL_n$, it is a theorem of Harris-Taylor (\cite{HT01}), of G. Henniart (\cite{H00}) and of
P. Scholze (\cite{Sc13}) that the local Langlands reciprocity map exists and is unique with compatibility of
local factors, plus other conditions. Note that in this case, the uniqueness of such a local Langlands
reciprocity map is proved by Henniart using the special case of the local converse theorem (\cite{H93}).
However, such a uniqueness is not known in general. When $G$ is an $F$-quasisplit classical group,
then such a local Langlands reciprocity map exists due to the endoscopic classification of J. Arthur
(\cite{Ar13}).

In their recent work (\cite{FS21}), L. Fargues and P. Scholze use the geometrization method to understand the local Langlands conjecture. In particular, they establish a local Langlands reciprocity
map for any $F$-split reductive groups considered in this paper. More precise, Theorem I.9.6 of
\cite{FS21} asserts that for any $F$-split reductive group $G$, there exists a local Langlands reciprocity
map $\FR_{F,G}$ from $\Rep(G(F))$ to $\Phi_F(G)$, satisfying nine compatibility conditions. In particular
when $G=\GL_n$, the reciprocity map of Fargues and Scholze coincides with the unique one for $\GL_n$.
When $G$ is an $F$-quasisplit classical group, the reciprocity map of Fargues and Scholze coincides with
the one by Arthur. Although it is still not known (as far as the authors know) if the reciprocity map of Fargues and Scholze is unique, it is the most promising one towards the local Langlands conjecture in
great generality.

From now on, we are going to take the following assumption.
\begin{ass}\label{FarS}
Over any non-Archimedean local field $F$ of characteristic zero, for any $F$-split reductive group $G$,  the reciprocity map $\FR_{F,G}$ exists for the local Langlands conjecture for $G$.
\end{ass}

We may simply take the reciprocity map $\FR_{F,G}$ as defined in \cite[Theorem I.9.6]{FS21} for the local Langlands conjecture. In fact, the relevant discussions in the rest of this paper make no essential difference
on which reciprocity map $\FR_{F,G}$ we are going to take. Of course, the difference may occur if one discuss
the definition of local $L$-functions $L(s,\sigma,\rho)$ or $\gam$-functions $\gam(s,\sigma,\rho,\psi)$.
but we are not going to discuss those objects in the rest of this paper.

\subsection{Absolute convergence of $(\sig,\rho)$-theta functions}\label{ssec-AC-sig-rho-TF}

Let $G$ be a $k$-split reductive group. Take $\rho\colon G^\vee(\BC)\to\GL_n(\BC)$ to be any
finite dimensional representation of the complex dual group $G^\vee(\BC)$.
For any unitary $\sigma\in\CA_\BA(G)$, we write $\sigma=\otimes_\nu\sigma_\nu$ with
$\sigma_\nu\in\Pi_{k_\nu}(G)$, the set of equivalence classes of irreducible admissible representations of
$G(k_\nu)$. Note that at any Archimedean local place $\nu$ of $k$, the local representation $\sigma_\nu$ is
assumed to be of Casselman-Wallach type.

By Assumption \ref{FarS}, for any local place $\nu\in|k|$, there exists a local $L$-parameter
$\varsigma_\nu=\varsigma_\nu(\sigma_\nu)$ such that the composition $\rho\circ\varsigma_\nu$ is a local
$L$-parameter for $\GL_n$. By the local Langlands conjecture for $\GL_n$ (\cite{L89}, \cite{H00}, \cite{HT01}, and \cite{Sc13}), there exists a unique irreducible admissible representation
\begin{align}\label{sigma-pi}
\pi_\nu=\pi_\nu(\sigma,\rho,\FR_{k_\nu,G})
\end{align}
belongs to $\Pi_F(n)$, which we may simply denote, if there is no confusion, by
\begin{align}\label{sigma-pi}
\pi_\nu=\pi_\nu(\sigma_\nu,\rho).
\end{align}
Define the {\bf $(\sigma_\nu,\rho)$-Schwartz space} on $k_\nu^\times$ to be
\begin{align}\label{localSS}
\CS_{\sigma_\nu,\rho}(k_\nu^\times):=\CS_{\pi_\nu}(k_\nu^\times).
\end{align}
At unramified local places, the $(\sigma_\nu,\rho)$-basic function $\BL_{\sigma_\nu,\rho}$ is taken to be
the $\pi_\nu$-basic function $\BL_{\pi_\nu}\in \CS_{\pi_\nu}(k_\nu^\times)$. Then we can define the {\bf $(\sigma,\rho)$-Schwartz space} on $\BA^\times$ to be the restricted tensor product:
\begin{align}\label{globalSS}
\CS_{\sigma,\rho}(\BA^\times):=\otimes_\nu\CS_{\sigma_\nu,\rho}(k_\nu^\times)
\end{align}
with respect to the basic function $\BL_{\sigma_\nu,\rho}$ at almost all finite local places. Note that
the definition of the $(\sigma,\rho)$-Schwartz space $\CS_{\sigma,\rho}(\BA^\times)$ may rely on
the assumption of the local Langlands reciprocity map (Assumption \ref{FarS}) at the ramified finite local
places of $\sigma$, when $G$ is a general $k$-split reductive group.

Let $\psi=\otimes_\nu\psi_\nu$ be a non-trivial additive character of $\BA$ with $\psi(a)=1$ for any $a\in k$.
Define the {\bf $(\sigma_\nu,\rho)$-Fourier operator} $\CF_{\sigma_\nu,\rho,\psi_\nu}$ on $k_\nu^\times$ to be
\begin{align}\label{localFO}
\CF_{\sigma_\nu,\rho,\psi_\nu}:=\CF_{\pi_\nu,\psi_\nu},
\end{align}
which is a linear transformation from the $(\sigma_\nu,\rho)$-Schwartz space
$\CS_{\sigma_\nu,\rho}(k_\nu^\times)$ to the $(\wt{\sigma_\nu},\rho)$-Schwartz space
$\CS_{\wt{\sigma_\nu},\rho}(k_\nu^\times)$. Then we define the {\bf $(\sigma,\rho)$-Fourier operator}
\begin{align}\label{globalFO}
\CF_{\sigma,\rho,\psi}
:=
\otimes_\nu\CF_{\sigma_\nu,\rho,\psi_\nu},
\end{align}
which is a linear transformation from the $(\sigma,\rho)$-Schwartz space
$\CS_{\sigma,\rho}(\BA^\times)$ to the $(\wt{\sigma},\rho)$-Schwartz space
$\CS_{\wt{\sigma},\rho}(\BA^\times)$. Again, the definition of the $(\sigma,\rho)$-Fourier operator
$\CF_{\sigma,\rho,\psi}$ may rely on
the assumption of the local Langlands reciprocity map (Assumption \ref{FarS}) at the ramified finite local
places of $\sigma$, when $G$ is a general $k$-split reductive group.

\begin{thm}[Absolute Convergence of $(\sig,\rho)$-Theta Functions]\label{thm:rho-AC}
Let $\rho\colon G^\vee(\BC)\to\GL_n(\BC)$ be any finite dimensional representation of the complex
dual group $G^\vee(\BC)$. Take Assumption \ref{FarS} for $G$.
Then for any given unitary $\sigma\in\CA_\BA(G)$, the $(\sig,\rho)$-theta function
\[
\Theta_{\sig,\rho}(x,\phi):=\sum_{\alp\in k^\times}\phi(\alp x)
\]
converges absolutely for any $\phi\in\CS_{\sig,\rho}(\BA^\times)$ and $x\in\BA^\times$.
\end{thm}

\begin{proof}
As discussed above, under Assumption \ref{FarS} for $G$, for any
$\sig =\otimes_\nu \sig_\nu\in \CA_\BA(G)$, we obtain $\pi_\nu= \pi_\nu(\sig_\nu,\rho)$ of
$\GL_n(k_\nu)$ for all $\nu\in|k|$. Note that at $\nu\in|k|_\infty$, $\pi_\nu$ is taken to be of Casselman-Wallach type. Hence $\pi:=\otimes_\nu\pi_\nu$ is an irreducible admissible representation of $\RG_n(\BA)$ and belongs to $\Pi_\BA(n)$. From \eqref{localSS} and \eqref{globalSS}, we have that
\[
\Theta_{\sig,\rho}(x,\phi)=\Theta_\pi(x,\phi)
\]
for any $\phi\in\CS_{\sig,\rho}(\BA^\times)=\CS_\pi(\BA^\times)$. By Theorem \ref{thm:AC}, it is
sufficient to verify Assumption \ref{ass-FH} for this $\pi$.

Since $\sig =\otimes_\nu \sig_\nu\in\CA_\cusp(G)$ is unitary, for any $\nu\in |k|$, $\sig_\nu$ is an irreducible
admissible unitary representation of $G(k_\nu)$, and is unramified for almost all $\nu\in |k|$.
Since $G$ is $k$-split, we can fix a Borel pair $(B,T)$ of $G$ defined over $k$, with a fixed maximal
$k$-split torus $T$ of $G$. Let $\varrho$ be the half-sum of positive roots with respect to the given pair
$(B,T)$ and let $\del_B$ be the modular character of $B(k_\nu)$. Then for any coweight
$\omega^\vee\in \Hom(\BG_m,T)$,
$$
\del_B(\omega^\vee(\vpi_\nu))^{1/2} = q_\nu^{\langle \varrho,\omega^\vee\rangle}
$$
where $\vpi_\nu$ is a fixed uniformizer in $\Fo_\nu$ and $\omega^\vee$ is viewed as a cocharacter from $k^\times_\nu$ to $T(k_\nu)$.

Let $S$ be a finite subset of $|k|$ containing $|k|_\infty$, such that for any $\nu\notin S$, both
$\sig_\nu$ and $\pi_\nu$ are unramified. For any $\nu\notin S$, $\sig_\nu$ is unitary and unramified. Then the zonal spherical function attached to $\sig_\nu$, which is the normalized matrix coefficient of $\sig_\nu$ attached to spherical vectors in $\sig_\nu$, is bounded by $1$ (see \cite[p.151, (40)]{C79} for instance). Now let
$$
c(\sig_\nu) = (q_\nu^{s_1(\sig_\nu)},...,q_\nu^{s_r(\sig_\nu)})
$$
be the Frobenius-Hecke conjugacy class of $\sig_\nu$ inside $T^\vee(\BC)\simeq (\BC^\times)^r$, where $r$ is the $k$-rational rank of $G$. Then by \cite[Theorem~4.7.1]{Mac71},
$$
\max_{1\leq j\leq r}\{|s_j(\sig_\nu)|\}
\leq
\max_{1\leq j\leq r}
\{|\langle \varrho,\omega^\vee_j\rangle|\}
$$
where $\{\omega^\vee_j\}_{j=1}^r$ is a fixed set of fundamental coweights.
Note that the result of \cite{Mac71} assumes $G$ to be simply-connected. But if we go over the proof
of \cite[Theorem~4.7.1]{Mac71}, the only result used is the explicit formula for zonal spherical functions
when the Frobenius-Hecke conjugacy class $c(\sig_\nu)$ of $\sig_\nu$ is non-singular.
Hence it suffices to apply the general formula appearing in \cite[Theorem~4.2]{Cas80} to the proof
of \cite[Theorem~4.7.1]{Mac71}.
Therefore $\max_{1\leq j\leq r}\{|s_j(\sig_\nu)|\}$ has an upper bound which is independent of the local places $\nu$.

At unramified local places,
we obtain the Frobenius-Hecke conjugacy class $c(\pi_\nu)$ of $\pi_\nu$ to be
\[
c(\pi_\nu)=\rho(c(\sig_\nu)).
\]
for all $\nu\notin S$. It is clear that for this $\pi=\otimes_\nu\pi_\nu\in\Pi_\BA(n)$, the family of
the Frobenius-Hecke conjugacy classes
$$
\{c(\pi_\nu)\mid \forall \nu\notin S\}
$$
associated to the irreducible admissible representation $\pi$ satisfies Assumption \ref{ass-FH}. We are done.
\end{proof}

Note that the definition of the $(\sigma,\rho)$-theta function $\Theta_{\sigma,\rho}(x,\phi)$ may depend
on the existence of the local Langlands reciprocity map $\FR_{F,G}$ for general $G$ (Assumption \ref{FarS}), However, the absolute convergence of $\Theta_{\sigma,\rho}(x,\phi)$ in Theorem \ref{thm:rho-AC} only depends on the unramified data, and hence is independent of Assumption \ref{FarS}.
As a consequence of Theorem \ref{thm:PSF}, we have

\begin{cor}\label{cor:CE}
Assume that the global Langlands functoriality is valid for $(G,\rho)$.
For $\sig\in\CA_\cusp(G)$, if its functorial transfer $\pi$ is cuspidal on $\RG_n(\BA)$, then
Conjecture \ref{cnj:main} holds with $\CE_{\sig,\rho}(\phi)=\Theta_{\sig,\rho}(1,\phi)$ for
any $\phi\in\CS_{\sig,\rho}(\BA^\times)$.
\end{cor}

\section{Variants of Conjecture \ref{cnj:main}}\label{sec-VCmain}

In Theorem \ref{thm:PSF}, we have established a $\pi$-Poisson summation formula
(Conjecture \ref{cnj:main}) for any $\pi\in\CA_\cusp(\RG_n)$ and $\rho=\std$. In this section, we explore the possibilities
when $\pi$ is not cuspidal.

\subsection{Certain special Schwartz functions}\label{ssec-SSF}
As before, we take $F$ to be any local field of characteristic zero. For any $\pi\in\Pi_F(n)$, recall from
Definition \ref{def:piSS} the space of $\pi$-Schwartz functions:
$$
\CS_\pi(F^\times)
=
\Span\{\phi_{\xi,\vphi_\pi}\in\CC^\infty(F^\times)\mid \xi\in \CS_{\std}(\RG_n(F)),\vphi_\pi\in \CC(\pi)\},
$$
where the $\pi$-Schwartz function $\phi_{\xi,\vphi_\pi}$ associated to a pair $(\xi,\vphi_\pi)$ is defined in \eqref{fibration}. We introduce here a subspace of $\CS_\pi(F^\times)$:
\begin{align}\label{piSS0}
\CS_{\pi}^\circ(F^\times)
:
=\Span \{
\phi_{\xi,\vphi_{\pi}}\mid
\xi\in \CC^\infty_c(\RG_n(F)),\vphi_{\pi}\in \CC(\pi)
\}.
\end{align}
The goal of this section is to prove the following result.

\begin{thm}\label{thm:testfunction}
Let $F$ be any local field of characteristic zero. For any $\pi\in\Pi_F(n)$, the subspace
$\CS_{\pi}^\circ(F^\times)$ of $\CS_\pi(F^\times)$ as defined in \eqref{piSS0} is equal to the space
$\CC_c^\infty(F^\times)$ of all smooth, compactly supported functions on $F^\times$.
\end{thm}

First of all, via the determinant morphism $\det\colon\RG_n\to\BG_m$, it is not hard to verify that the
fibre integration
\[
\xi\mapsto \int_{\det g=x}\xi(g)\ud_x g
\]
yields a surjective homomorphism from $\CC_c^\infty(\RG_n(F))$ to $\CC_c^\infty(F^\times)$.
For any $\xi\in\CC_c^\infty(\RG_n(F))$ and $\vphi_\pi\in\CC(\pi)$, the product $\xi(g)\vphi_\pi(g)$ belongs
to $\CC_c^\infty(\RG_n(F))$. With the fibre integration through $\det$, the function $\phi_{\xi,\vphi_\pi}(x)$
belongs to $\CC_c^\infty(F^\times)$. Hence we obtain that
\begin{align}\label{tf-1}
\CS_{\pi}^\circ(F^\times) \subset\CC_c^\infty(F^\times)
\end{align}
for any $\pi\in\Pi_F(n)$.
To prove the converse of \eqref{tf-1}, we are going to use different arguments for the non-Archimedean
case and the Archimedean case.

We first consider the {\sl non-Archimedean case.}
For any quasi-character $\chi\in\FX(F^\times)$, it can be written in a unique way as $\chi(x)=|x|_F^s\cdot\ome(x)$ with $s\in\BC$ and $\ome\in\Ome^\wedge$. We may write $\chi=\chi_{s,\ome}$ and
$\FX(F^\times)=\BC\times\Ome^\wedge$. Furthermore, we write the space $\CZ(\FX(F^\times))$ defined
in Definition \ref{dfn:CZ} as
$\CZ(\BC\times\Ome^\wedge)$. We denote by $\CL_\cpt$ the subspace of $\CZ(\BC\times\Ome^\wedge)$
consisting of functions $\Fz(\chi_{s,\ome})=\Fz(s,\ome)\in\CZ(\BC\times\Ome^\wedge)$ with the following two properties:
\begin{enumerate}
\item $\Fz({s,\ome})$ is nonzero at finitely many $\ome\in \Ome^\wedge$;
\item For any $\ome\in \Ome^\wedge$, $\Fz({s,\ome})\in \BC[q^s,q^{-s}]$.
\end{enumerate}
By Theorem \ref{thm:MT}, the subspace $\CL_\cpt$ is in one to one correspondence with
$\CC^\infty_c(F^\times)$ via the Mellin transform and its inversion. Denote by $\CL^\circ_{\pi}$ the
subspace of $\CL_\cpt$ that is in ono to one correspondence with the subspace $\CS^\circ_\pi(F^\times)$.
From the discussion right after \cite[Theorem~3.1.1]{Luo}, for any given $\ome\in\Ome^\wedge$,
the following subspace
$$
\CI^\circ_{\pi,\ome}
: = \{\CZ(s,\xi,\vphi_{\pi},\ome)\mid \xi\in \CC^\infty_c(\RG_n(F)), \vphi_{\pi}\in \CC(\pi)\}
$$
of the fractional ideal $\CI_{\pi,\ome}$ as in Theorem \ref{thm:GJ}
is equal to $\BC[q^s,q^{-s}]$. For the fixed $\ome\in \Ome^\wedge$, the space $\CI^\circ_{\pi,\ome}$
consists of the restriction of functions in $\CL_\cpt$ to the slice $\BC\times \{\ome\}$, according to the definition of the space $\CL_\cpt$.
In other words, for any fixed $\ome\in \Ome^\wedge$ and $\Fz({s,\ome})\in \CL_\cpt$,
there exists finitely many $\xi^j_\ome\in \CC^\infty_c(\RG_n(F))$ and
$\vphi^j_{\pi,\ome}\in \CC(\pi)$, such that
$$
\Fz({s,\ome})
=\sum_{j}\CZ(s,\xi^j_\ome, \vphi^j_{\pi,\ome},\ome)
=
\sum_{j}\CZ(s,\phi_{\xi^j_\ome,\vphi^j_{\pi,\ome}},\ome)
$$
for any $s\in \BC$. Hence with any fixed $\ome\in \Ome^\wedge$, for any $\Fz({s,\ome})\in \CL_\cpt$,
there exists $\Fz^\circ({s,\ome})\in \CL^\circ_{\pi}$ such that
\begin{align}\label{tf-2}
\Fz({s,\ome}) = \Fz^\circ({s,\ome})
\end{align}
as functions in $s\in\BC$.

Define, for each $\ome_0\in\Ome^\wedge$, a function $\Fz_{\ome_0}(s,\ome)$ with property:
\[
\Fz_{\ome_0}(s,\ome)
=
\begin{cases}
1,& {\rm if}\ \ome=\ome_0;\\
0,& {\rm if}\ \ome\neq\ome_0.
\end{cases}
\]
By definition, the function $\Fz_{\ome_0}(s,\ome)$ belongs to $\CL_\cpt$ for each $\ome_0\in\Ome^\wedge$.
Hence from \eqref{tf-2}, we have
\begin{align}\label{tf-3}
\Fz({s,\ome})
=
\sum_{\ome_0\in\Ome^\wedge}\Fz_{\ome_0}(s,\ome)\cdot\Fz({s,\ome})
=
\sum_{\ome_0\in\Ome^\wedge}
\Fz_{\ome_0}(s,\ome)\cdot
\Fz^\circ({s,\ome_0}),
\end{align}
for any $\Fz({s,\ome})\in\CL_\cpt$. Note here that the summations only takes finitely many
$\ome_0\in\Ome^\wedge$. Hence to prove the converse of \eqref{tf-1}, it is enough to show that
\begin{align}\label{tf-4}
\Fz_{\ome_0}(s,\ome)\cdot
\Fz^\circ({s,\ome_0})\in\CL^\circ_\pi
\end{align}
for every $\ome_0\in\Ome^\wedge$. It is clear that \eqref{tf-4} is an easy consequence of the following proposition.

\begin{prp}\label{pro:na-CS0}
The space $\CL_\cpt$ is an associative algebra without identity, and
the space $\CL^\circ_{\pi}$ is an $\CL_\cpt$-module under multiplication.
\end{prp}

\begin{proof}
From the definition of $\CL_\cpt$, it is clear that $\CL_\cpt$ is an associative algebra under the multiplication
and has no identity.

To prove that $\CL^\circ_{\pi}$ is an $\CL_\cpt$-module, we take $\Fz(s,\ome)\in\CL_\cpt$ and write $\phi$ as the Mellin inversion of $\Fz(s,\ome)$. Via the determinant morphism $\det:\RG_n(F)\to F^\times$, we write
$$
\phi(x) =
\int_{\det g=x}f(g)\ud_xg.
$$
for some $f\in \CC^\infty_c(\RG_n(F))$. For any $\xi\in \CC^\infty_c(\RG_n(F))$ and $\vphi_\pi\in \CC(\pi)$,
we write $\Fz^\circ(s,\ome)\in \CL^\circ_\pi$ to be the Mellin transform of the function
$\phi_{\xi,\vphi_\pi}\in\CS^\circ_\pi(F^\times)$. It is enough to show that
\begin{align}\label{tf-5}
\Fz(s,\ome)\cdot\Fz^\circ(s,\ome)\in\CL^\circ_\pi.
\end{align}
It is clear that
\begin{align}\label{tf-6}
\Fz(s,\ome)\cdot\Fz^\circ(s,\ome)
=
\CZ(s,\phi*\phi_{\xi,\vphi_\pi},\ome).
\end{align}
Now we compute the right-hand side of \eqref{tf-6} with a fixed $\ome\in\Ome^\wedge$.
\begin{align}\label{tf-7}
\CZ(s,\phi*\phi_{\xi,\vphi_\pi},\ome)
&=
\int_{x\in F^\times}
\ome(x)|x|_F^{s-\frac{1}{2}}\ud^\times x
\int_{y\in F^\times}
\phi(y)
\phi_{\xi,\vphi_\pi}(y^{-1}x)
\ud^\times y\nonumber
\\
&=
\int_{F^\times}
\ome(x)|x|_F^{s-\frac{1}{2}}\ud^\times x
\int_{F^\times}
\ud^\times y
\int_{\det g=y}
f(g)\ud_y g\\
&\qquad\qquad\qquad\qquad\qquad\qquad\qquad\cdot
\int_{\det h=y^{-1}x}
\xi(h)\vphi_\pi(h)\ud_{y^{-1}x}h.\nonumber
\end{align}
After changing variable $g\to gh^{-1}$, the right-hand side of \eqref{tf-7} is equal to
\begin{align}\label{tf-8}
\int_{F^\times}
\ome(x)|x|_F^{s-\frac{1}{2}}\ud^\times x
\int_{F^\times}
\ud^\times y
\int_{\det g=x}
f(gh^{-1})\ud_x g \int_{\det h=y^{-1}x}
\xi(h)\vphi_\pi(h)\ud_{y^{-1}x}h.
\end{align}
In \eqref{tf-8}, the integration in $y\in F^\times$ yields the following identity:
\begin{align}\label{tf-9}
\int_{y\in F^\times}
\ud^\times y
\int_{\det h = y^{-1}x}
f(g h^{-1})\xi(h)\vphi_\pi(h)\ud_{y^{-1}x}h
=\int_{\RG_n(F)}
f(gh^{-1})\xi(h)\vphi_\pi(h)\ud h.
\end{align}
By applying \eqref{tf-9} to \eqref{tf-8}, we can write \eqref{tf-8} as
\[
\int_{F^\times}
\ome(x)|x|_F^{s-\frac{1}{2}}\ud^\times x
\int_{\det g=x}
\int_{\RG_n(F)}
f(gh^{-1})\xi(h)\vphi_\pi(h)\ud h
\ud_x g,
\]
which is equal to
\begin{align}\label{tf-10}
\int_{g\in \RG_n(F)}
\int_{h\in \RG_n(F)}
f(gh^{-1})
\xi(h)\vphi_\pi(h)
\ome(\det g)|\det g|_F^{s-\frac{1}{2}}
\ud h\ud g.
\end{align}
By taking a change of variable: $h\to h^{-1}g$, \eqref{tf-10} can be written as
\begin{align}\label{tf-11}
\int_{g\in \RG_n(F)}
\int_{h\in \RG_n(F)}
f(h)
\xi(h^{-1}g)\vphi_\pi(h^{-1}g)
\ome(\det g)|\det g|_F^{s-\frac{1}{2}}
\ud h\ud g.
\end{align}
Since $\xi\in\CC_c^\infty(\RG_n(F))$, the function
\[
(g,h)\mapsto\xi(h^{-1}g)
\]
belongs to the space $\CC^\infty_c(\RG_n(k_\nu)\times \RG_n(k_\nu))$. By \cite[1.22]{BerZ76}, we have
$$
\CC^\infty_c(\RG_n(k_\nu)\times \RG_n(k_\nu))
\simeq \CC^\infty_c(\RG_n(k_\nu))\otimes \CC^\infty_c(\RG_n(k_\nu)).
$$
We may write
\[
\xi(h^{-1}g)
=
\sum_{j=1}^r\xi_j(g)\xi^j(h).
\]
for some $\xi_j(g)$ and $\xi^j(h)$ in $\CC_c^\infty(\RG_n(F))$. Meanwhile, we write
$$
\vphi_\pi(h^{-1}g) = \langle \pi(h^{-1}g) v,\wt{v}\rangle
=\langle \pi(g)v,\wt{\pi}(h)\wt{v}\rangle ,\quad v\in \pi,\wt{v}\in \wt{\pi}.
$$
By applying those explicit expressions to the integral in \eqref{tf-11}, we obtain that \eqref{tf-11} is
equal to
\begin{align}\label{ts-12}
&\sum_{j=1}^r
\int_{g\in \RG_n(F)}
\int_{h\in \RG_n(F)}
f(h)\xi_i(g)\xi^i(h)
\langle \pi(g)v,\wt{\pi}(h)\wt{v}\rangle
\ome(\det g)|\det g|^{s-\frac{1}{2}}_F \ud h\ud g\nonumber
\\
&\qquad=
\sum_{j=1}^r
\int_{g\in \RG_n(F)}
\xi_i(g)\ome(\det g)|\det g|^{s-\frac{1}{2}}_F
\ud g
\int_{h\in \RG_n(F)}
f(h)\xi^i(h)
\langle \pi(g)v,\wt{\pi}(h)\wt{v}\rangle
\ud h\nonumber
\\
&\qquad\qquad=
\sum_{j=1}^r
\int_{\RG_n(F)}
\xi_i(g)
\langle \pi(g)v,\wt{\pi}(f\cdot \xi^j)\wt{v}\rangle
\ome(\det g)|\det g|^{s-\frac{1}{2}}_F
\ud g.
\end{align}
By writing $\vphi_{\pi,j}(g) := \langle \pi(g)v,\wt{\pi}(f\cdot \xi^j)\wt{v}\rangle$, we obtain that
\begin{align}\label{tf-13}
\CZ(s,\phi*\phi_{\xi,\vphi_\pi},\ome)
=
\sum_{j=1}^r
\int_{\RG_n(F)}
\xi_i(g)
\vphi_{\pi,j}(g)
\ome(\det g)|\det g|^{s-\frac{1}{2}}_F
\ud g
=\sum_{j=1}^r
\CZ(s,\phi_{\xi_j,\vphi_{\pi,j}},\ome).
\end{align}
By definition of $\CL^\circ_\pi$, we obtain that the right-hand side of \eqref{tf-13} belongs to the space
$\CL^\circ_\pi$, so does the function $\CZ(s,\phi*\phi_{\xi,\vphi_\pi},\ome)$. Therefore we have established \eqref{tf-5}. We are done.
\end{proof}

We have finished the proof of Theorem \ref{thm:testfunction} for the non-Archimedean case.
Now we turn to the proof the converse of \eqref{tf-1} and hence Theorem \ref{thm:testfunction}
for the Archimedean case.

We first treat the case when $F\simeq \BC$.
It is clear that the multiplication map
\begin{align}\label{ar-tf-1}
\Fm:\BC^\times\times \SL_n(\BC) &\to \RG_n(\BC) \nonumber\\
(a,h)&\mapsto a\cdot h
\end{align}
provides a surjective group homomorphism with finite kernel, which in particular is a smooth (covering) map. The push-forward map along $\Fm$, which we denote by
\begin{align}\label{ar-tf-2}
\Fm_*:\CC^\infty_c(\BC^\times\times \SL_n(\BC)) \to \CC^\infty_c(\RG_n(\BC))
\end{align}
is surjective. In fact, the surjectivity can be easily verified as follows. For any $f\in \CC^\infty_c(\RG_n(\BC))$, let
$
\Fm^*(f)
$
be the pull-back of $f$ along $\Fm$, which is given by
$$
\Fm^*(f)(a,h) = f(a\cdot h),\quad (a,h)\in \BC^\times \times \SL_n(\BC).
$$
Then we obtain that
$$
\Fm_*(\Fm^*(f))(h) =
\sum_{(a,h) \in \BC^\times \times \SL_n(\BC), a\cdot h=g}
f(a\cdot h)
=
|\ker(\Fm)|
\cdot
f( g),\quad g\in \RG_n(\BC).
$$
It is clear now that the subspace $\CS_\pi^\circ(\BC^\times)$ of $\CS_\pi(\BC^\times)$ is equal to
the space spanned by the following functions
\begin{align}\label{ar-tf-3}
\phi_{\Fm_*(f),\vphi_{\pi}}(x)
=
\int_{\det g=x}
\Fm_*(f)(g)
\vphi_{\pi}(g)\ud_xg
=
\int_{\det g=x}
\sum_{(a,h)\in \BC^\times\times \SL_n(\BC),a\cdot h=g}
f(a,h)
\vphi_{\pi}(g)\ud_xg
\end{align}
with all $f\in \CC^\infty_c(\BC^\times\times \SL_n(\BC))$ and $\vphi_{\pi}\in \CC(\pi)$.
Thus, in order to show the converse of \eqref{tf-1}, it suffices to show that any function in
$\CC^\infty_c(\BC^\times)$ is of the form as in the right-hand side of \eqref{ar-tf-3}.

Let $\chi_{\pi}$ be the central character of $\pi$. By a change of variable, we write \eqref{ar-tf-3} as
\begin{align}\label{ar-tf-4}
\phi_{\Fm_*(f),\vphi_{\pi}}(x)
=
\int_{\SL_n(\BC)}
\sum_{a^n = x}
f(a,h)
\cdot
\chi_{\pi}(a)
\cdot
\vphi_{\pi}(h)\ud_1 h.
\end{align}
Assume that $f\in \CC^\infty_c(\BC^\times\times \SL_n(\BC))$ is given by a pure tensor
\[
f(a,h) = f_1(a)\cdot f_2(h)
\]
with $f_1\in \CC^\infty_c(\BC^\times)$ and $f_2\in \CC^\infty_c(\SL_n(\BC))$. Then \eqref{ar-tf-4}
can be written as
\begin{align}\label{ar-tf-5}
\phi_{\Fm_*(f),\vphi_{\pi}}(x) =
\left\{
\sum_{a^n = x}
f_1(a)\chi_{\pi}(a)
\right\}
\cdot
\int_{\SL_n(\BC)}
f_2(h)\vphi_{\pi}(h)\ud_1h.
\end{align}
It is clear that multiplying by the character $\chi_{\pi_\BC}$ stabilizes the space $\CC^\infty_c(\BC^\times)$, which means that $f_1(a)\chi_{\pi}(a)\in\CC_c^\infty(\BC^\times)$ for any $f_1\in\CC_c^\infty(\BC^\times)$.
The map
\begin{align*}
\CC^\infty_c(\BC^\times)&\to \CC^\infty_c(\BC^\times)
\\
f&\mapsto \left( x\mapsto \sum_{a^n=x}f(a) \right)
\end{align*}
is surjective, since it is the push-forward map along the surjective covering map
\begin{align*}
\BC^\times &\to \BC^\times
\\
a &\mapsto a^n.
\end{align*}
Therefore, any function in $\CC_c^\infty(\BC^\times)$ can be written as $\phi_{\Fm_*(f),\vphi_{\pi}}(x)$
for some $\vphi_{\pi}\in \CC(\pi)$ and $f\in \CC^\infty_c(\BC^\times\times \SL_n(\BC))$. This finishes
the proof of the converse of \eqref{tf-1}.

We now turn to the case when $F=\BR$ and treat the cases of $n$ odd and of $n$ even separately.

When $n$ is odd, the multiplication map
\begin{align*}
\Fm\colon\BR^\times \times \SL_n(\BR)&\to \RG_n(\BR)
\\
(a,g)&\mapsto a\cdot g
\end{align*}
is surjective, the proof in the complex case is applicable and yields a proof for this case.
We omit the details here.

When $n$ is even, we write $\RG_n(\BR)$ as a disjoint union two real smooth manifolds,
$$
\RG_n(\BR) = \RG_n^+(\BR)
\sqcup \RG_n^-(\BR)
$$
where $\RG_n^+(\BR)$ (resp. $\RG_n^-(\BR)$) consists of elements in $\RG_n(\BR)$ with positive (resp. negative) determinant.

By the surjectivity of the following map
\begin{align*}
\BR_{>0}\times \SL_n(\BR)&\to \RG^+_n(\BR)
\\
(a,g)&\to a\cdot g
\end{align*}
the proof in the complex case shows that the space $\CS^\circ_{\pi}(\BR^\times)$ contains the space
$\CC^\infty_c(\BR_{>0})$. Since $\BR^\times= \BR_{>0}\sqcup\BR_{<0}$, we have that
\[
\CC_c^\infty(\BR^\times)=\CC^\infty_c(\BR_{>0})\oplus\CC^\infty_c(\BR_{<0}).
\]
It remains to show that
\begin{align}\label{ar-tf-6}
\CC^\infty_c(\BR_{<0})
\subset \CS^\circ_{\pi}(\BR^\times).
\end{align}
Take $\theta = \diag(-1,1,...,1)\in \RG_n(\BR)$ and consider the following map
\begin{align*}
\Fm^-: \BR_{>0}\times \SL_n(\BR)&\to \RG^-_n(\BR)
\\
(a,h)&\mapsto a\cdot h\cdot \theta.
\end{align*}
As the complex situation, for any $f\in \CC^\infty_c(\BR_{>0}\times \SL_n(\BR))$, we set
\begin{align}\label{ar-tf-7}
\phi_{\Fm^-_*(f),\vphi_{\pi}}(x)
=\int_{\det g=x}
\sum_{(a,h)\in \BR_{>0}\times \SL_n(\BR), a\cdot h \cdot \theta = g}
f(a,h)\cdot
\vphi_{\pi}(g)
\ud_xg,
\end{align}
for $x\in\BR_{<0}$.
We only need to show that the space spanned by the functions of the following form:
\begin{align}\label{ar-tf-8}
\bigg\{x\mapsto \phi_{\Fm^-_*(f),\vphi_{\pi}}(x)\mid f\in \CC^\infty_c(\BR_{>0}\times \SL_n(\BR)), \vphi_{\pi}\in \CC(\pi)
\bigg\}
\end{align}
with $x\in\BR_{<0}$ contains (and hence is equal to) the space $\CC^\infty_c(\BR_{<0})$.

By a simple change of variable, we obtain that
\begin{align}\label{ar-tf-9}
\phi_{\Fm^-_*(f),\vphi_{\pi}}(x)
=
\int_{\SL_n(\BR)}
\sum_{a^n = -x}
f(a,h)
\cdot
\chi_{\pi}(a)
\vphi_\pi(h\cdot \theta)\ud_1h
\end{align}
where $\chi_\pi$ is the central character of $\pi\in\Pi_\BR(n)$.
Assume that $f(a,h) = f_1(a)\cdot f_2(h)$ is a pure tensor with $f_1\in\CC_c^\infty(\BR_{>0})$ and
$f_2\in\CC_c^\infty(\SL_n(\BR))$. Then \eqref{ar-tf-9} can be written as
\begin{align}\label{ar-tf-10}
\phi_{\Fm^-_*(f),\vphi_{\pi}}(x)
=
\sum_{a^n=-x}
f_1(a)\chi_{\pi_\BR}(a)
\cdot
\int_{\SL_n(\BR)}
f_2(h)
\vphi_{\pi}(h\cdot \theta)\ud_1h,
\end{align}
with $x\in\BR_{<0}$. For $y=-x>0$, the functions of the form
\[
\sum_{a^n=y}
f_1(a)\chi_{\pi_\BR}(a)
\cdot
\int_{\SL_n(\BR)}
f_2(h)
\vphi_{\pi}(h\cdot \theta)\ud_1h
\]
recovers the space $\CC_c^\infty(\BR_{>0})$, as treated in the previous case. Thus, the functions of the
form in \eqref{ar-tf-10} recovers the space $\CC^\infty_c(\BR_{<0})$. This completes the proof for
the converse of \eqref{tf-1} for the Archimedean case. Therefore, we finish the proof of Theorem
\ref{thm:testfunction}.

\subsection{A variant of $\pi$-Poisson summation formulae}\label{ssec-V}

For any $\pi=\otimes_{\nu\in|k|}\in\Pi_\BA(n)$, we define in \eqref{piSS-BA} the space of $\pi$-Schwartz functions on $\BA^\times$:
\[
\CS_\pi(\BA^\times)=\otimes_{\nu}\CS_{\pi_\nu}(k_\nu^\times).
\]
We define $\CS_\pi^\circ(\BA^\times)$ to be the subspace of $\CS_\pi(\BA^\times)$ that is spanned by
the functions of the form: $\phi=\otimes_\nu\phi_\nu\in\CS_\pi(\BA^\times)$ with at least one local component $\phi_\nu$ belonging to $\CC_c^\infty(k_\nu^\times)$. Note that for any $\phi=\otimes_\nu\phi_\nu\in\CS_\pi(\BA^\times)$, there are at most finitely many local components
from $\CC_c^\infty(k_\nu^\times)$. It is also easy to verify from the definition of $\pi$-Fourier operator
$\CF_{\pi,\psi}$ as in \eqref{FO-BA} and Theorem \ref{thm:testfunction} that there exist functions
$\phi=\otimes_\nu\phi_\nu\in\CS_\pi(\BA^\times)$, such that
\[
\CF_{\pi,\psi}(\phi)=\prod_\nu\CF_{\pi_\nu,\psi_\nu}(\phi_\nu)\in\CS^\circ_{\wt{\pi}}(\BA^\times).
\]
We define $\CS_\pi^{\circ\circ}(\BA^\times)$ to be the subspace of $\CS_\pi^\circ(\BA^\times)$ that
is spanned by the functions of the form: $\phi=\otimes_\nu\phi_\nu\in\CS_\pi^\circ(\BA^\times)$ with
property that $\CF_{\pi,\psi}(\phi)\in\CS^\circ_{\wt{\pi}}(\BA^\times)$.

\begin{thm}\label{thm:weakPSL2}
Let $\pi\in \CA(\RG_n)$ be square-integrable.  For any $\phi\in \CS_\pi^{\circ\circ}(\BA^\times)$,
the $\pi$-Poisson summation formula
$$
\Theta_\pi(x,\phi)=\Theta_{\wt{\pi}}(x^{-1},\CF_{\pi,\psi}(\phi))
$$
holds.
\end{thm}

\begin{proof}
By Corollary \ref{cor:AC}, both $\Theta_\pi(x,\phi)$ and $\Theta_{\wt{\pi}}(x^{-1},\CF_{\pi,\psi}(\phi))$
are absolutely convergent. It suffices to show the equality. The proof goes in the same way as
Theorem \ref{thm:PSF} when $\pi\in\CA_\cusp(\RG_n)$. The first key point is that when
$\pi$ is square-integrable, its matrix coefficients can also be realized as the integrals in \eqref{ps-1},
with $\bet_1,\bet_2\in V_\pi$ being not necessarily cuspidal.

The second key point is to prove that
the boundary terms defined in \eqref{ps-9} vanish automatically by the local assumption on $\phi$ at
the two local places $\nu_1$ and $\nu_2$.
More precisely, take $\phi=\phi_{\xi,\vphi_\pi}\in\CS_\pi(\BA^\times)$ and assume that
\[
\phi=\otimes_\nu\phi_\nu=\otimes_\nu\phi_{\xi_\nu,\vphi_{\pi_\nu}}
\]
with $\xi_\nu(g)=|\det g|_\nu^{\frac{n}{2}}f_\nu(g)$ for some $f_\nu\in\CS(\RM_n(k_\nu))$, and
$\vphi_{\pi_\nu}\in\CC(\pi_\nu)$. The assumption at the two local places $\nu_1$ and $\nu_2$ is the
same as that $f_{\nu_1}\in\CC_c^\infty(\RG_n(k_{\nu_1}))$ and
$\CF_{\psi_{\nu_2}}(f_{\nu_2})\in\CC_c^\infty(\RG_n(k_{\nu_2}))$. For $f=\otimes_\nu f_\nu$
and $\CF_\psi(f)=\otimes_\nu\CF_{\psi_\nu}(f_\nu)$ with the above $f_{\nu_1}$ at the given local
place $\nu_1$ and $\CF_{\psi_{\nu_2}}(f_{\nu_2})$ at the given local place $\nu_2$, the boundary terms $\RB_f(h,g)$ in \eqref{ps-9} must vanish automatically. Therefore, the summation identity is
established. We refer the other details of the proof to the proof of Theorem \ref{thm:PSF}.
\end{proof}

\subsection{Refinement of Conjecture \ref{cnj:main}}\label{ssec-RCmain}

We are going to state our conjecture on $(\sig,\rho)$-Poisson summation formula on $\GL_1$ with more
details, which refines Conjecture \ref{cnj:main}. We will continue with the discussions in Section
\ref{ssec-AC-sig-rho-TF}. By Assumption \ref{FarS}, for $\sigma\in\CA_\cusp(G)$, there exists a $\pi=\otimes_\nu\pi_\nu\in\Pi_\BA(n)$ with $\pi_\nu=\pi_\nu(\sig_\nu,\rho)$ for all $\nu$. We define
the space $\CS_{\sig,\rho}(\BA^\times)$ of $(\sig,\rho)$-Schwartz functions as in \eqref{localSS} and \eqref{globalSS}; and the $(\sig,\rho)$-Fourier operator $\CF_{\sig,\rho,\psi}$ as in \eqref{localFO} and \eqref{globalFO}. Finally we define the space $\CS_{\sig,\rho}^{\circ\circ}(\BA^\times)$ to be equal to
$\CS_\pi^{\circ\circ}(\BA^\times)$.

\begin{cnj}[Refinement of Conjecture \ref{cnj:main}]\label{cnj:ref}
Let $G$ be a $k$-split reductive group, and
$\rho\colon G^\vee(\BC)\to\GL_n(\BC)$ be a representation of the complex
dual group $G^\vee(\BC)$.
With Assumption \ref{FarS}, for any $\sigma\in\CA_\cusp(G)$, there exist $k^\times$-invariant linear functionals
$\CE_{\sigma,\rho}$ and $\CE_{\wt{\sigma},\rho}$ on $\CS_{\sigma,\rho}(\BA^\times)$ and
$\CS_{\wt{\sigma},\rho}(\BA^\times)$, respectively, such that the
$(\sigma,\rho)$-Poisson Summation Formula:
\begin{align}\label{RPSF}
\CE_{\sigma,\rho}(\phi)
=
\CE_{\wt{\sigma},\rho}(\CF_{\sigma,\rho,\psi}(\phi))
\end{align}
holds for $\phi\in\CS_{\sigma,\rho}(\BA^\times)$. Moreover,
if $\phi\in\CS_{\sig,\rho}^{\circ\circ}(\BA^\times)$, then the identity in \eqref{RPSF} holds for
\[
\CE_{\sigma,\rho}(\phi)(x)=\Theta_{\sig,\rho}(x,\phi) = \sum_{\alp\in k^\times}\phi(\alp x)
\]
with $x\in\BA^\times$.
\end{cnj}

In Corollary \ref{cor:CE}, we have proved that
if the global Langlands functoriality is valid for $(G,\rho)$ and the image of $\sig$
under the functorial transfer is cuspidal on to $\RG_n(\BA)$, then Conjecture \ref{cnj:main} and
Conjecture \ref{cnj:ref} hold with
$$
\CE_{\sigma,\rho}(\phi)(x)=\Theta_{\sig,\rho}(x,\phi) = \sum_{\alp\in k^\times}\phi(\alp x)
$$
for any $\phi\in \CS_{\sig,\rho}(\BA^\times)$ and any $x\in\BA^\times$.


\section{Critical Zeros of $L(s,\pi\times\chi)$}\label{sec-CZ}


We provide a spectral interpretation of critical zeros of the automorphic $L(s,\pi\times\chi)$
for any $\pi\in\CA_\cusp(n)$ and character $\chi$ of the idele class group
$\CC_k=k^\times\bs\BA^\times$ for a number field $k$.
This can be viewed as a reformulation of \cite[Theorem 2]{S01} (see also \cite{D01})
in the adelic formulation of A. Connes
\cite{Cn99}, and is a extention of \cite[Theorem III.1]{Cn99} from the Hecke $L$-functions
$L(s,\chi)$ to the standard automorphic $L$-functions $L(s,\pi\times\chi)$.

\subsection{P\'olya-Hilbert-Connes Pairs}\label{ssec-PHC}

For a number field $k$, denote $\BA^1=\BA^1_k:=\ker(|\cdot|_\BA)$ be the norm one ideles of $k$.
Denote by $\CC_k:=k^\times\bs\BA^\times$ the idele
class group of $k$, and $\CC_k^1:=k^\times\bs\BA^1$. Then $\BA^\times$ has a non-canonical decomposition
\begin{align}\label{non-can}
\BA^\times=\BA^1\times \BR^\times_+
\end{align}
where $\BR^\times_+=|\BA^\times|_\BA$ is the connected component of $1$. In the following, we choose and fix a section of
the following short exact sequence:
\[
1\to\BA^1\to\BA^\times\to\BR^\times_+\to1.
\]
This gives a fixed non-canonical decomposition
\begin{align}\label{non-CC}
\CC_k=\CC_k^1\times\BR^\times_+.
\end{align}

For any $\delta>0$, define $L^2_\delta(\CC_k)$ to the space consisting of measurable functions
\[
\theta\colon\CC_k\rightarrow \BC
\]
with a finite Sobolev norm $\|\cdot\|_\delta$ as defined by
\begin{align}\label{sob-n}
\|\theta\|_\delta^2:=\int_{\CC_k}|\theta(x)|^2(1+(\log|x|_\BA)^2)^{\frac{\delta}{2}}\ud^\times x.
\end{align}
It is clear that the space $L^2_\delta(\CC_k)$ is a $\CC_k$-module via the right translation $\Fr_\delta$:
\begin{align}\label{Fr}
\Fr_\delta(a)(\theta)(x):=\theta(xa)
\end{align}
for any $\theta\in L^2_\delta(\CC_k)$ and $a,x\in\CC_k$. Note that the $\CC_k$-module $L^2_\delta(\CC_k)$
is not unitary, but has the property:
\begin{align}\label{asmp-1}
\|\Fr_\delta(x)\|=o(\log|x|_\BA)^{\frac{\delta}{2}},\quad (|x|_\BA\to\infty).
\end{align}

\begin{prp}\label{theta-decay}
For a $\pi\in\CA_\cusp(n)$, take any $\phi\in\CS_\pi(\BA^\times)$. For any $\kappa>0$, there exists a
positive constant $c_{\kappa,\phi}$ such that the $\pi$-theta function $\Theta_\pi(x,\phi)$ enjoys the
following property:
\[
|\Theta_\pi(x,\phi)|\leq c_{\kappa,\phi}\cdot\min\{|x|_\BA,|x|_\BA^{-1}\}^\kappa.
\]
\end{prp}

\begin{proof}
This is a reformulation of Part (ii) of \cite[Theorem 1]{S01} and can be proved accordingly. We
omit the detail.
\end{proof}

By Proposition \ref{theta-decay}, for any $\phi\in\CS_\pi(\BA^\times)$, with $\pi\in\CA_\cusp(n)$,
the $\pi$-theta function $\Theta_\pi(\cdot,\phi)$ decays rapidly when $|x|_\BA\to 0$ or
$|x|_\BA\to \infty$, and hence belongs to $L^2_\delta(\CC_k)$. Define
\begin{align}\label{sob-n-phi}
\|\phi\|_\delta^2:=\int_{\CC_k}|\Theta_\pi(x,\phi)|^2(1+(\log|x|_\BA)^2)^{\frac{\delta}{2}}\ud^\times x
\end{align}
for any $\phi\in\CS_\pi(\BA^\times)$. Then the following embedding
\begin{align}\label{isometry}
\Theta_\pi\colon\phi\in\CS_\pi(\BA^\times)\mapsto\Theta_\pi(\cdot,\phi)\in L^2_\delta(\CC_k)
\end{align}
with respect to the sobolev norms defined in \eqref{sob-n} and \eqref{sob-n-phi}, respectively.

Denote
by $\ovl{\Theta_\pi}$ the completion of the image $\Theta_\pi(\CS_\pi(\BA^\times))$ in $L^2_\delta(\CC_k)$.
Since
\[
\Fr_\delta(y)(\Theta_\pi(\cdot,\phi))(x)=\Theta_\pi(x,\Fr_\delta(y)\phi)
\]
for any $\phi\in\CS_\pi(\BA^\times)$, with $x,y\in\CC_k$, the closed subspace $\ovl{\Theta_\pi}$ is also
a $\CC_k$-module. Define the quotient space
\begin{align}\label{quotient}
\CH_{\pi,\delta}:=L^2_\delta(\CC_k)/\ovl{\Theta_\pi},
\end{align}
which is also a $\CC_k$-module. The associated representation is denoted by $\Fr_{\pi,\delta}$.
It is clear that the restriction of the $\CC_k$-module to $\CC_k^1$ is unitary
and has the following decomposition
\begin{align}\label{eigen-space}
\CH_{\pi,\delta}|_{\CC_k^1}=\oplus_{\chi\in\wh{\CC_k^1}}\ \CH_{\pi,\delta,\chi}.
\end{align}
By the fixed (non-canonical) decomposition in \eqref{non-CC}, each eigenspace $\CH_{\pi,\delta,\chi}$ is a
module of $\BR^\times_+$. The associated representation is denoted by $\Fr_{\pi,\delta,\chi}$.
Note that $\Fr_{\pi,\delta,\chi}$ is also a representation of $\CC_k=\CC_k^1\times\BR^\times_+$ on
$\CH_{\pi,\delta,\chi}$.
The action of $\BR^\times_+$ on $\CH_{\pi,\delta,\chi}$ generates a flow
with the infinitesimal generator
\begin{align}\label{infinitesimal}
\FD_{\pi,\delta,\chi}(\theta):=\lim_{\epsilon\to 0}\frac{1}{\epsilon}\left(\Fr_{\pi,\delta,\chi}(\exp(\epsilon)-1)\right)\theta
\end{align}
for any $\theta\in\CH_{\pi,\delta,\chi}$. As in \cite{Cn99}, one should take the pair
\begin{align}\label{PHC}
(\CH_{\pi,\delta,\chi}, \FD_{\pi,\delta,\chi})
\end{align}
to be a candidate of the P\'olya-Hilbert space. We call it a {\bf P\'olya-Hilbert-Connes pair}.

For any $\chi\in\wh{\CC_k^1}$, by the fixed non-canonical decomposition
$\CC_k=\CC_k^1\times\BR^\times_+$ as in \eqref{non-CC}, the character $\chi$ has a unique extension to
$\CC_k$ by defining that it is trivial on $\BR^\times_+$. We may still denote the extended character by
$\chi$.

\begin{thm}[Critical Zeros of $L(s,\pi\times\chi)$]\label{zero}
Given any $\pi\in\CA_\cusp(n)$ and a character $\chi\in\wh{\CC_k^1}$, take $\FD_{\pi,\delta,\chi}$ as in \eqref{infinitesimal} with $\delta>1$.
\begin{enumerate}
\item The spectrum $\Sp(\FD_{\pi,\delta,\chi})$ is discrete and is contained in $i\cdot\BR$ with $i=\sqrt{-1}$.
\item $\mu\in\Sp(\FD_{\pi,\delta,\chi})$ if and only if $L(\frac{1}{2}+\mu,\pi\times\chi)=0$.
\item The multiplicity $m_{\Sp(\FD_{\pi,\delta,\chi})}(\mu)$ is equal to the largest integer $m<\frac{1+\delta}{2}$ with $m\leq m_{L(s,\pi\times\chi)}(\frac{1}{2}+\mu)$, the multiplicity of $\frac{1}{2}+\mu$
as a zero of the automorphic $L$-function $L(s,\pi\times\chi)$.
\end{enumerate}
\end{thm}

Note Theorem \ref{zero} can be viewed as a reformulation of \cite[Theorem 2]{S01} in the adelic framework of \cite{Cn99} and is an extension of \cite[Theorem III.1]{Cn99} from the Hecke $L$-functions $L(s,\chi)$
to the standard automorphic $L$-functions $L(s,\pi\times\chi)$. See also \cite{D01} for relevant discussion.

\subsection{Proof of Theorem \ref{zero}}\label{ssec-PThm}

We are going to prove Theorem \ref{zero} by using an argument that combines the approach of \cite{Cn99}
and that of \cite{S01}.

Consider the following pairing
\begin{align}\label{sob-pair}
\begin{matrix}
L^2_{\delta}(\CC_k)\ \times\ L^2_{-\delta}(\CC_k)&\rightarrow&\BC\\
(\theta,\eta)&\mapsto&<\theta,\eta>.
\end{matrix}
\end{align}
where the pairing is defined by the following integral
\[
<\theta,\eta>:=\int_{\CC_k}\theta(x)\eta(x)\ud^\times x.
\]
For any $y\in\CC_k$, we have
\[
<\Fr_\delta(y)\theta,\eta>=<\theta,\Fr_{-\delta}(y^{-1})\eta>
\]
for any $\theta\in L^2_{\delta}(\CC_k)$ and $\eta\in L^2_{-\delta}(\CC_k)$.

Consider a function $\eta\in L^2_{-\delta}(\CC_k)$ as a distribution on the eigenspace
$\CH_{\pi,\delta,\chi}$. Then we must have
\begin{align}\label{zero-1}
<\theta,\eta>=0
\end{align}
for any $\theta\in\ovl{\Theta_\pi}$, and for any $t\in\CC_k^1$,
\[
\Fr_{-\delta}(t)\eta=\chi^{-1}(t)\eta
\]
as a distribution on $\CH_{\pi,\delta,\chi}$. Hence we may write,
for $x=ta\in\CC_k=\CC_k^1\times\BR^\times_+$, the fixed non-canonical decomposition,
\begin{align}\label{zero-2}
\eta(x)=\chi^{-1}(t)\beta(a)
\end{align}
where $\beta(a)$ is a measurable function on $\BR^\times_+$ with
\[
\|\beta\|_\delta=\int_{\BR^\times_+}|\beta(a)|^2(1+(\log|a|)^2)^{-\frac{\delta}{2}}\ud^\times a<\infty.
\]
The orthogonality in \eqref{zero-1} can be written as
\begin{align}\label{zero-3}
\int_{\CC_k}\Theta_\pi(x,\phi)\eta(x)\ud^\times x=0
\end{align}
for any $\phi\in\CS_\pi(\BA^\times)$. As in \cite{S01}, we prove the following lemma,
which is a reformulation of Lemma 1 of \cite{S01}.

\begin{lem}\label{pi-dense}
The subspace of $\ovl{\Theta_\pi}$ generated by functions of type:
\[
(b*\Theta_\pi(\cdot,\phi))(t)=\int_{\CC_k}b(x)\Theta_\pi(x^{-1}t,\phi)\ud^\times x
\]
with all $b(x)\in\CC_c^\infty(\CC_k)$ is dense in $\ovl{\Theta_\pi}$.
\end{lem}

\begin{proof}
The proof is a reformulation of the proof of \cite[Lemma 1]{S01}. For any
$\theta\in\ovl{\Theta_\pi}$, we have
\[
(b*\theta)(t)=\int_{\CC_k}b(x)\theta(x^{-1}t)\ud^\times x
=
\int_{\CC_k}b(x)\theta^\vee(t^{-1}x)\ud^\times x
=
\Fr_\delta(b)(\theta^\vee)(t^{-1})
\]
for any $b(x)\in\CC_c^\infty(\CC_k)$. since $\ovl{\Theta_\pi}$ is a closed subspace of
$L^2_{\delta}(\CC_k)$ and is a $\CC_k$-module, it is clear that $b*\theta$ belongs to
$\ovl{\Theta_\pi}$.
In particular, we have that $b*\Theta_\pi(\cdot,\phi)$ belongs to $\ovl{\Theta_\pi}$ for
all $b(x)\in\CC_c^\infty(\CC_k)$ and all $\phi\in\CS_\pi(\BA^\times)$.

Next, by \cite[Lemma 5]{Cn99}, there exists a sequence of functions $\{f_n\}$ with $f_n$ belonging
to the space $\CS(\CC_k)$ of the Bruhat-Schwartz functions on $\CC_k$, such that
$\Fr_\delta(f_n)$ tends strongly to one in $L^2_\delta(\CC_k)$ and the norm of $\Fr_\delta(f_n)$
are bounded. Now following the same argument as in the proof of \cite[Lemma 1]{S01}, we obtain that there exists a sequence of functions $b_n\in\CC_c^\infty(\CC_k)$ with property that
\begin{enumerate}
\item $\Fr_\delta(b_n)$ converges strongly to one;
\item the norm of $\Fr_\delta(b_n)$ is bounded;
\item $b_n*\Theta_\pi(\cdot,\phi)$ converges to $\Theta_\pi(\cdot,\phi)$ for any
$\phi\in\CS_\pi(\BA^\times)$.
\end{enumerate}
Therefore the linear span of $b*\Theta_\pi(\cdot,\phi)$ with $b(x)\in\CC_c^\infty(\CC_k)$ and
$\phi\in\CS_\pi(\BA^\times)$ is dense in $\ovl{\Theta_\pi}$. We are done.
\end{proof}

By Lemma \ref{pi-dense}, it is enough to consider the orthogonality
\begin{align}\label{zero-4}
\int_{\CC_k}(b*\Theta_\pi(\cdot,\phi))(x)\eta(x)\ud^\times x=0
\end{align}
for any $\phi\in\CS_\pi(\BA^\times)$ and $b(x)\in\CC_c^\infty(\CC_k)$.

\begin{lem}\label{orthogonality}
For any $\eta\in L^2_{-\delta}(\CC_k)$,
the integral
\[
\int_{\CC_k}(b*\Theta_\pi(\cdot,\phi))(x)\eta(x)\ud^\times x
\]
is zero for any $b\in\CC_c^\infty(\CC_k)$ and any $\phi\in\CS_\pi(\BA^\times)$ if and only if
\[
L(\frac{1}{2}+i\mu,\pi\times\chi)\cdot\CM(\eta)(\chi_{i\mu})
\]
is zero as a function in $\chi_{i\mu}$, where $\chi_{i\mu}$ is any unitary character of $\CC_k$ that
can be written as $\chi_{i\mu}(x)=\chi(t)a^{i\mu}$ for $x=ta\in\CC_k=\CC_k^1\times\BR^\times_+$, the fixed non-canonical decomposition.
\end{lem}

\begin{proof}
We are going to apply the Parseval formula for the Fourier transform from $\CC_k$ to its unitary dual $\wh{\CC_k}$ to \eqref{zero-4}. Since $\chi_{i\mu}(x)=\chi(t)a^{i\mu}$, the Fourier transform
for $\CC_k$ is
\[
\CM(\theta)(\chi_{i\mu})=\int_{\CC_k}\theta(x)\chi_{i\mu}^{-1}(x)\ud^\times x.
\]
By applying the Parseval formula to the integral
\[
\int_{\CC_k}(b*\Theta_\pi(\cdot,\phi))(x)\eta(x)\ud^\times x,
\]
we obtain that \eqref{zero-4}
is equivalent to
\begin{align}\label{zero-5}
\int_{\wh{\CC_k}}
\CM(b)(\chi_{i\mu})
\CM(\Theta_\pi(\cdot,\phi))(\chi_{i\mu})
\CM(\eta)(\chi_{i\mu})
\ud\chi_{i\mu}=0
\end{align}
for any $\phi\in\CS_\pi(\BA^\times)$ and $b(x)\in\CC_c^\infty(\CC_k)$. It is easy to verify from definition
that
\[
\CM(\Theta_\pi(\cdot,\phi))(\chi_{i\mu})=\CZ(\frac{1}{2}+i\mu,\phi,\chi),
\]
where the right-hand side is the global ($\GL_1$) zeta integral as defined in \eqref{1-gzi}. From Corollary \ref{zeta-zeta} and \cite[Proposition 13.9]{GJ72}, the global zeta integral $\CZ(\frac{1}{2}+i\mu,\phi,\chi)$
is a bounded function in $\mu$. Hence the product
\[
\CT_{\phi,\eta}(\chi_{i\mu}):=\CZ(\frac{1}{2}+i\mu,\phi,\chi)\cdot\CM(\eta)(\chi_{i\mu})
\]
is a tempered distribution on $\wh{\CC_k}$. Hence \eqref{zero-5} is the same as
\begin{align}\label{zero-5-1}
\int_{\wh{\CC_k}}
\CM(b)(\chi_{i\mu})
\CT_{\phi,\eta}(\chi_{i\mu})
\ud\chi_{i\mu}=0
\end{align}
for any $\phi\in\CS_\pi(\BA^\times)$ and $b(x)\in\CC_c^\infty(\CC_k)$.
Denote by $\wh{\CT}_{\phi,\eta}(x)$ the (inverse) Fourier transform of
$\CT_{\phi,\eta}(\chi_{i\mu})$. By using the Parseval formula for the (inverse) Fourier transform, we obtain
that \eqref{zero-5-1} is equivalent to
\begin{align}\label{zero-6}
\int_{\CC_k}b(x)\wh{\CT}_{\phi,\eta}(x)\ud^\times x=0
\end{align}
for all $\phi\in\CS_\pi(\BA^\times)$ and $b(x)\in\CC_c^\infty(\CC_k)$. Hence we must have that
\eqref{zero-6} holds if and only if
$\wh{\CT}_{\phi,\eta}(x)=0$ as distribution on $\CC_k$,
which is equivalent to that $\CT_{\phi,\eta}(\chi_{i\mu})=0$ as distribution on $\wh{\CC_k}$. In other words, we obtain that for any $\eta\in L^2_{-\delta}(\CC_k)$,
the integral
\[
\int_{\CC_k}(b*\Theta_\pi(\cdot,\phi))(x)\eta(x)\ud^\times x
\]
is zero for any $b\in\CC_c^\infty(\CC_k)$ and any $\phi\in\CS_\pi(\BA^\times)$ if and only if
\begin{align}\label{zero-7}
\CZ(\frac{1}{2}+i\mu,\phi,\chi)\cdot\CM(\eta)(\chi_{i\mu})=0
\end{align}
for all $\phi\in\CS_\pi(\BA^\times)$. By Corollary \ref{zeta-zeta} and \cite[Theorem 13.8]{GJ72}, there exist
a finite many $\phi_1,\cdots,\phi_\ell\in\CS_\pi(\BA^\times)$ such that
\[
\CZ(\frac{1}{2}+i\mu,\phi_1,\chi)+\cdots+\CZ(\frac{1}{2}+i\mu,\phi_\ell,\chi)=L(\frac{1}{2}+i\mu,\pi\times\chi).
\]
Thus we obtain that \eqref{zero-7} implies
\begin{align}\label{zero-8}
L(\frac{1}{2}+i\mu,\pi\times\chi)\cdot\CM(\eta)(\chi_{i\mu})=0
\end{align}
as a function in $\chi_{i\mu}$.

To prove the converse, we consider factorizable data: $\phi=\otimes_\nu\phi_\nu\in\CS_\pi(\BA^\times)$
and $\chi=\otimes_\nu\chi_\nu$. The global zeta integral factorizes into an Euler product
\[
\CZ(s,\phi,\chi)=\prod_\nu\CZ(s,\phi_\nu,\chi_\nu)
\]
By Theorem \ref{thm:1-zeta}, we obtain that
\[
\CZ(s,\phi,\chi)=L(s,\pi\times\chi)\cdot
\prod_{\nu\in S}\frac{\CZ(s,\phi_\nu,\chi_\nu)}{L(s,\pi_\nu\times\chi_\nu)}
\]
where $S$ is the finite set of local places, including all Archimedean local places of $k$, such that
for any $\nu\not\in S$, the data $\pi_\nu$ and $\chi_\nu$ are unramified, and the quotient
\[
\frac{\CZ(s,\phi_\nu,\chi_\nu)}{L(s,\pi_\nu\times\chi_\nu)}
\]
is holomorphic in $s\in\BC$. Hence if $\eta\in L^2_{-\delta}(\BA^\times)$ satisfies
\[
L(\frac{1}{2}+i\mu,\pi\times\chi)\cdot\CM(\eta)(\chi_{i\mu})=0
\]
as a function in $\chi_{i\mu}$, i.e. \eqref{zero-8} holds,
then \eqref{zero-7} holds for factorizable data: $\phi=\otimes_\nu\phi_\nu\in\CS_\pi(\BA^\times)$
and $\chi=\otimes_\nu\chi_\nu$. Hence it holds for all $\phi\in\CS_\pi(\BA^\times)$ and all $\chi$.
We are done.
\end{proof}

The rest of the proof of Theorem \ref{zero} is exactly the same as that in the proof of Theorem 2 of \cite[Page 178]{S01}, which follows from the same argument of Connes (pp. 86--87, in the proof
of Theorem III.1 of \cite{Cn99}). We omit the details.




\begin{thebibliography}{jiang2020}

\bibitem[AG08]{AG08}
Aizenbud, A.; Gourevitch, D.
{\it Schwartz functions on Nash manifolds}. Int. Math. Res. Not. IMRN 2008, no. 5, Art. ID rnm 155, 37 pp.


\bibitem[Ar13]{Ar13}
Arthur, James
{\it The endoscopic classification of representations. Orthogonal and symplectic groups}. American Mathematical Society Colloquium Publications, 61. American Mathematical Society, Providence, RI, 2013.

\bibitem[Ber84]{Ber84}
Bernstein, J.
{\it Le ``centre" de Bernstein.}
(French) Edited by P. Deligne. Representations of reductive groups over a local field, Travaux en Cours, Hermann, Paris, 1984, 1-32.


\bibitem[BerK14]{BerK14}
Bernstein, J.; Kr\"otz, B.
{\it Smooth Fréchet globalizations of Harish-Chandra modules}. Israel J. Math. 199 (2014), no. 1, 45--111.

\bibitem[BerZ76]{BerZ76}
Bernsteinn, I.; Zelevinski, A.
{\it Representations of the group $GL(n,F),$ where $F$ is a local non-Archimedean field.}(Russian), Uspehi Mat. Nauk, Vol.31, 1976, no. 3(189), 5--70.


\bibitem[BK00]{BK00}
Braverman, A.; Kazhdan, D.
{\it $\gamma$-functions of representations and lifting.}
With an appendix by V. Vologodsky. GAFA 2000 (Tel Aviv, 1999).
Geom. Funct. Anal. 2000, Special Volume, Part I, 237--278.



\bibitem[C79]{C79}
Cartier, P.
{\it Representations of $p$-adic groups: A survey}. Automorphic forms, representations and $L$-functions, Part 1, pp. 111--155, Proc. Sympos. Pure Math., XXXIII, Amer. Math. Soc., Providence, R.I., 1979.

\bibitem[Cas80]{Cas80}
Casselman, W.
{\it The unramified principal series of $p$-adic groups. I. The spherical function}.
Compositio Math. vol 40 (1980), no. 3, 387--406.

\bibitem[Cas89]{Cas89}
Casselman, W.
{\it Canonical extensions of Harish-Chandra modules to representations of G}. Canadian Journal of Mathematics 41 (1989), 385--438.

\bibitem[Cn99]{Cn99}
Connes, A.
{\it Trace formula in noncommutative geometry and the zeros of the Riemann zeta function}. Selecta Math. (N.S.) 5 (1999), no. 1, 29--106.

\bibitem[D01]{D01}
Deitmar, A.
{\it A Polya-Hilbert operator for automorphic $L$-functions}. Indag. Math. (N.S.) 12 (2001), no. 2, 157--175.

\bibitem[FS21]{FS21}
Fargues, L; Scholze, P.
{\it Geometrization of the local Langlands correspondence}.
arXiv:2102.13459.

\bibitem[GPSR87]{GPSR87}
Gelbart, S.; Piatetski-Shapiro, I.; Rallis, S.
{\it Explicit constructions of automorphic $L$-functions. }
Lecture Notes in Mathematics, 1254. Springer-Verlag, Berlin, 1987.



\bibitem[GL21]{GL21}
Getz, J.; Liu, Baiying,
{\it A refined Poisson summation formula for certain Braverman-Kazhdan spaces}. Sci. China Math. 64,
1127--1156 (2021).

\bibitem[GJ72]{GJ72}
Godement, R.; Jacquet, H.
{\it Zeta functions of simple algebras.}
Lecture Notes in Mathematics, Vol. 260. Springer-Verlag, Berlin-New York, 1972.

\bibitem[HT01]{HT01}
Harris, M.; Taylor, R.
{\it The geometry and cohomology of some simple Shimura varieties. With an appendix by Vladimir G. Berkovich}. Annals of Mathematics Studies, 151. Princeton University Press, Princeton, NJ, 2001.

\bibitem[H93]{H93}
Henniart, G.
{\it Caract\'erisation de la correspondance de Langlands locale par les facteurs $\epsilon$ de paires}. (French) Invent. Math. 113 (1993), no. 2, 339--350.

\bibitem[H00]{H00}
Henniart, G.
{\it Une preuve simple des conjectures de Langlands pour $\GL(n)$ sur un corps $p$-adique}. (French) Invent. Math. 139 (2000), no. 2, 439--455.



\bibitem[Ig78]{Ig78}
Igusa, J.
{\it Forms of higher degree.} Tata Institute of Fundamental Research Lectures on Mathematics and Physics, 59. Tata Institute of Fundamental Research, Bombay; by the Narosa Publishing House, New Delhi, 1978.

\bibitem[Ig00]{Ig00}
Igusa, J.
{\it An introduction to the theory of local zeta functions.} AMS/IP Studies in Advanced Mathematics, 14. American Mathematical Society, Providence, RI; International Press, Cambridge, MA, 2000.


\bibitem[J71]{J71}
Jacquet, H.
{\it Représentations des groupes linéaires $p$-adiques.(French)} Theory of group representations and Fourier analysis. (Centro Internaz. Mat. Estivo (C.I.M.E.), II Ciclo, Montecatini Terme, 1970), 119--220, 1971.

\bibitem[J09]{J09}
Jacquet, H.
{\it Archimedean Rankin-Selberg integrals}. Automorphic forms and L-functions II. Local aspects, 57--172, Contemp. Math., 489, Israel Math. Conf. Proc., Amer. Math. Soc., Providence, RI, 2009.

\bibitem[JL]{JL}
Jiang, Dihua; Luo, Zhilin
{\it Certain Fourier operators on $\GL_1$ and local Langlands gamma functions}.
in preparation 2021.


\bibitem[JLZ20]{JLZ20}
Jiang, Dihua; Luo, Zhilin; Zhang, Lei
{\it Harmonic analysis and gamma functions on symplectic groups}.
arXiv:2006.08126.



\bibitem[Lf14]{Lf14}
Lafforgue, L.
{\it Noyaux du transfert automorphe de Langlands et formules de Poisson non lin\'eaires}. (French)  Jpn. J. Math. 9 (2014), no. 1, 1--68.

\bibitem[Lf16]{Lf16}
Lafforgue, L.
{\it Du transfert automorphe de Langlands aux formules de Poisson non lin\'eaires}. (French)  Ann. Inst. Fourier (Grenoble) 66 (2016), no. 3, 899--1012.

\bibitem[L70]{L70}
Langlands, R.
{\it Problems in the theory of automorphic forms}. Lectures in modern analysis and applications, III, pp. 18--61. Lecture Notes in Math., Vol. 170, Springer, Berlin, 1970.

\bibitem[L71]{L71}
Langlands, R.
{\it Euler products}. A James K. Whittemore Lecture in Mathematics given at Yale University, 1967. Yale Mathematical Monographs, 1. Yale University Press, New Haven, Conn.-London, 1971.

\bibitem[L79]{L79}
Langlands, R.
{\it On the notion of an automorphic representation}. Automorphic forms, representations and $L$-functions, Part 1, pp. 203--212, Proc. Sympos. Pure Math., XXXIII, Amer. Math. Soc., Providence, R.I., 1979.

\bibitem[L89]{L89}
Langlands, R.
{\it On the classification of irreducible representations of real algebraic groups}. Representation theory and harmonic analysis on semisimple Lie groups, 101--170, Math. Surveys Monogr., 31, Amer. Math. Soc., Providence, RI, 1989.






\bibitem[Li19]{Li19}
Li, Wen-Wei
{\it Generalized zeta integrals on certain real prehomogeneous vector spaces}.
arXiv: 1912.00809.


\bibitem[Luo19]{Luo19}
Luo, Zhilin
\textit{On the Braverman-Kazhdan proposal for local factors: spherical case}.
Pacific J. Math. 300 (2019), no. 2, 431--471.

\bibitem[Luo]{Luo}
Luo, Zhilin
\textit{An introduction to the proposal of Braverman and Kazhdan}.
Singapore Lecture Notes.


\bibitem[Mac71]{Mac71}
Macdonald, I. G.
\textit{
Spherical functions on a group of {$p$}-adic type
}
Publications of the Ramanujan Institute, Ramanujan Institute, Centre for Advanced Study in Mathematics,University of Madras, Madras
No. 2 (1971).


\bibitem[N20]{N20}
Ng\^o, B. C.
{\it Hankel transform, Langlands functoriality and functional equation of automorphic $L$-functions}. Jpn. J. Math. 15 (2020), no. 1, 121--167.




\bibitem[S63]{S63}
Satake, I.
{\it Theory of spherical functions on reductive algebraic groups over $p$-adic fields}.
Inst. Hautes \'Etudes Sci. Publ. Math. 1963, no. 18, 5--69.

\bibitem[Sc13]{Sc13}
Scholze, P.
{\it The local Langlands correspondence for $\GL_n$ over $p$-adic fields}.
Invent. Math. 192 (2013), no. 3, 663--715.

\bibitem[S01]{S01}
Soul\'e, C.
{\it On zeroes of automorphic $L$-functions}. Dynamical, spectral, and arithmetic zeta functions (San Antonio, TX, 1999), 167--179, Contemp. Math., 290, Amer. Math. Soc., Providence, RI, 2001.

\bibitem[SZ11]{SZ11}
Sun, Binyong; Zhu, Chen-Bo
{\it A general form of Gelfand-Kazhdan criterion}. Manuscripta Math. 136 (2011), no. 1-2, 185--197.

\bibitem[Td86]{Td86}
Tadi\'c, M.
{\it  Spherical unitary dual of general linear group over non-Archimedean local field}. Ann. Inst. Fourier (Grenoble) 36 (1986), no. 2, 47--55.

\bibitem[Tm63]{Tm63}
Tamagawa, T.
{\it On the $\zeta$-functions of a division algebra}. Ann. of Math. (2) 77 (1963), 387--405.

\bibitem[T50]{Tt50}
Tate, J.
{\it Fourier Analysis in Number Fields and Hecke's Zeta-Functions.} Thesis (Ph.D.), Princeton University. 1950. 60 pp; Algebraic Number Theory (Proc. Instructional Conf., Brighton, 1965), 305--347, Thompson, Washington, D.C., 1967.

\bibitem[Wal88]{Wal88}
Wallach, N. Real reductive groups. I. Pure and Applied Mathematics, 132-I. Academic Press, Inc., Boston, MA, 1988.

\bibitem[Wal92]{Wal92}
Wallach, N. Real reductive groups. II. Pure and Applied Mathematics, 132-II. Academic Press, Inc., Boston, MA, 1992.


\bibitem[W65]{W65}
Weil, A.
{\it Sur la formule de Siegel dans la th\'eorie des groupes classiques}. (French)
Acta Math. 113 (1965), 1--87.

\bibitem[W73]{W73}
Weil, A.
{\it Basic number theory}. Reprint of the second (1973) edition.
Classics in Mathematics. Springer-Verlag, Berlin, 1995.



\end{thebibliography}
\end{document}